%% file: main.tex
\documentclass[11pt]{amsart}
\input{command}
\input{packages}

\title{Iwasawa cohomology of pseudorigid families of $(\varphi, \Gamma)$-modules}
\author{Yashi Jain}

\begin{document}
\maketitle
\begin{abstract}
   We study the Iwasawa cohomology complex attached to a pseudorigid family of $(\varphi, \Gamma)$-modules and prove coadmissibility of its cohomology groups, perfectness and compatibility with base change. We apply this result to give another proof of perfectness of the Herr complex in pseudorigid family. We also use it to prove perfectness of Selmer complexes and obtain a result on variation of ranks of Selmer groups in analytic family. 
\end{abstract}
\begin{flushright}
\begin{minipage}{0.45\textwidth}
\itshape
Selmer groups are dead. 
Long live Selmer complexes!

\medskip
\hfill --- Jan Nekov\'a\v{r} \cite{nekovar2006selmer} 
\end{minipage}
\end{flushright}
\tableofcontents
\section{Introduction}
Iwasawa theory studies the variation of arithmetic invariants in infinite towers of number fields. One of its fundamental objects is Iwasawa cohomology, which packages the Galois cohomology of a representation over all finite layers of a $p$-adic Lie extension into a single module over an Iwasawa algebra. An example of a $p$-adic Lie extension is the cyclotomic extension
$\mathbb{Q}_p(\mu_{p^{\infty}})/\mathbb{Q}_p$, whose Galois group is a
$p$-adic Lie group. For such an extension, the Iwasawa cohomology is defined as
$$
H^i_{\mathrm{Iw}}({\mathbb{Q}_p}(\mu_{p^\infty})/\mathbb{Q}_p,T)
:=\varprojlim_F H^i(G_F,T),
\qquad
H^i_{\mathrm{Iw}}({\mathbb{Q}_p}(\mu_{p^\infty})/\mathbb{Q}_p,V):=H^i_{\mathrm{Iw}}({\mathbb{Q}_p}(\mu_{p^\infty})/\mathbb{Q}_p,T)
  \otimes_{\mathbb Z_p}\mathbb Q_p
$$
where $T$ is a lattice of $V$ and $F$ runs through the finite extensions of $\mathbb{Q}_p$ contained in
$\mathbb{Q}_p(\mu_{p^{\infty}})$, $G_F=\operatorname{Gal}(\overline{\mathbb{Q}}_p/F)$,
and the transition maps are the corestriction maps. This provides a natural framework for studying the variation of arithmetic objects such as Galois groups, Selmer groups, and $p$-adic $L$-functions.

In this article, we aim to study Iwasawa cohomology of representations of Galois groups of $p$-adic fields on modules over rings in a family whose coefficients are varying \textit{pseudorigid} analytically. These rings, which occur naturally while studying eigenvarieties at the boundary of the weight space \cite{johansson2019extended,andreatta2018halo}, are Tate algebras where $p$ is not necessarily invertible and can also be 0 (unlike rigid analytic rings studied by Tate where $p$ is invertible) and have non-Archimedean topology. The study of continuous representations of $\textrm{Gal}_{\mathbb{Q}_p}$ over $p$-adic vector spaces has been exceptionally hard because of similar topologies on both the Galois group and the vector space on which it acts. Such representations are studied using the theory of $(\varphi, \Gamma)$-modules over rings with complicated coefficients (called the \textit{Robba} rings), turning the data of the entire representation into semilinear data of a module with an action of a Frobenius operator $\varphi$ and a one-dimensional $p$-adic Lie group $\Gamma$. While the theory of rigid analytic families of $(\varphi,\Gamma)$-modules has been extensively developed by Fontaine, Colmez, Kedlaya, Berger, and others, recent work of Bellovin \cite{bellovin2023galois} extends this framework to the pseudorigid setting.

Given a $(\varphi, \Gamma)$-module $D$ that comes from a Galois representation, there exists an explicit two term complex $C_{\psi}^{\bullet}(D)$ that studies the Iwasawa cohomology of the Galois representation associated to $D$. Given such a module with coefficients in a pseudorigid space, we prove the finiteness of Iwasawa cohomology over a certain distribution algebra, extending the result of \cite{kedlaya2014cohomology}. The main results of this paper are the following:
\begin{theorem}
    Let $D = D_{(0, b]}$ be a $(\varphi, \Gamma_K)$ module over $\Lambda_{R, (0, b], K}$. Then for sufficiently small $a, b$ satisfying $0<a \leq b/p^3$, there exists a constant $\mu(a)$ depending on $a$ and $K$ such that $D_{[a, b]}^{\psi = 1}$ is finitely generated over $\Lambda_{R, [\mu(a), \infty]}(\Gamma_K)$.
\end{theorem}

\begin{theorem}(Theorem \ref{perfectiwa})
\label{perf02iw}
        Let $D = D_{(0, b]}$ be a $(\varphi, \Gamma_K)$ module over $\Lambda_{R, (0, b], K}$. 
        \begin{enumerate}
           
            \item The Iwasawa complex $C_{\psi}^{\bullet}(D) \in \textbf{D}_{\text{perf}}^{[0, 2]}(\Lambda_{R, (0, \infty]}(\Gamma_K))$.
            \item  The first cohomology group of the Iwasawa complex is finitely generated and coadmissible over $\Lambda_{R, (0, \infty]}(\Gamma_K)$ of rank $[K: \mathbb{Q}_p] \text{rk}(D)$.
            \item (Euler Characteristic formula) We have $\chi_{\Lambda_{R, (0, \infty]}(\Gamma_K)}(C_{\psi}^{\bullet}(D)) = -[K: \mathbb{Q}_p] \text{rk}(D)$. 
        \end{enumerate}
      \begin{remark}
          Note that it's perfect in degree $[0, 2]$ and not $[1, 2]$ as the first Iwasawa cohomology can have torsion.
      \end{remark}  
\end{theorem}
As a consequence of above, we prove:
\begin{theorem} (Theorem \ref{cyclotomicdeform})
Let $M$ be a $(\varphi,\Gamma_K)$-module over
$\Lambda_{R,\mathrm{rig},K}$. Put $A=\Lambda_{R,(0,\infty]}(\Gamma_K)$ and let
$M^*=M^\vee(1)$ be the Cartier dual.
Let $\iota:A\longrightarrow A$ be the involution induced
by $[\gamma]\longmapsto[\gamma^{-1}]$. There is a module $\textbf{Dfm}(M)$ over a relative distribution algebra with coefficients in $\Lambda_{R, (0, \infty]}(\Gamma_K)$ such that
we have a natural quasi-isomorphism of complexes of $A$-modules
$$
C^\bullet_{\psi,\gamma_K}(\mathbf{Dfm}(M))
\;\tilde{\longrightarrow}\;
C^\bullet_\psi(M),
$$
Consequently, there is a natural Iwasawa duality quasi-isomorphism
$$
C^\bullet_\psi(M^*)
\;\tilde{\longrightarrow}\;
\mathbf{R}\!\operatorname{Hom}_A
\left(C^\bullet_\psi(M),A\right)^\iota[-2].
$$
 \end{theorem}
(An informal guide to the rings appearing is given in Section \ref{notsec}.)
In \cite{kedlaya2014cohomology}, the authors establish the finiteness of Iwasawa complex using which they establish the perfectness of Herr complex (that computes Galois cohomology) in rigid analytic families. In the paper \cite{bellovin2024cohomology}, Bellovin proves finiteness of Herr complex for pseudorigid family of $(\varphi, \Gamma)$-modules without using the Iwasawa complex. Using Bellovin's field Euler-characteristic formula as
an input, we obtain Herr perfectness from Iwasawa
perfectness and character base change, following the
strategy of \cite{kedlaya2014cohomology}.

The main technical difficulty in the pseudorigid setting is that $p$ is not invertible in pseudoaffinoid algebras. As a consequence, one cannot use Fontaine’s element $t$ in the de Rham period ring to shift slopes or carry out arguments as in \cite{kedlaya2014cohomology}. Our approach circumvents this issue using an idea inspired by a proof of Kedlaya, Liu and Bellovin on the perfectness of the Herr complex. 

A key input is a theorem attributed to Cartan–Serre, which states that if the cokernel cohomology of a completely continuous morphism of complexes of Banach modules over a Noetherian Banach ring $R$ has finite $R$-cohomology, then the target complex also has finite $R$-cohomology.
The precise version of this theorem that we use is Corollary \ref{integralcompcontcohomology}, which is closely modeled on Lemma 1.10 of \cite{kedlaya2016finiteness}, modified slightly to work integrally so that finiteness carries out.
We aim to apply this result to the Iwasawa complex. A priori, however, the coefficient ring for the Iwasawa complex (which is a distribution algebra) is an inverse limit of Banach algebras and is far from being Noetherian or Banach. Sections 4 and 5 are devoted to developing the necessary machinery to overcome this difficulty and work at Banach level. Rather than showing complete continuity of the desired map between Iwasawa complexes over a "Banach type'' distribution algebra, we show that it is completely continuous up to homotopy (we refer the reader to Proposition \ref{firstcoh} for the precise homotopy). Another place where we take a different approach from \cite{kedlaya2014cohomology} is that we directly prove coadmissibility of the first Iwasawa cohomology group by explicitly writing it down as an inverse limit of some finitely generated modules. 

We give an application of this result in understanding the variation of the ranks of Selmer groups in analytic family, generalizing a result of Bella\"{\i}che to pseudorigid setting and making it more explicit. We prove the theorem given below, where the complexes denote a pseudorigid version of Selmer complexes:
\begin{theorem}
Let $R$ be a pseudoaffinoid algebra, and let $V$ be a finite projective $R$-module of rank $d$ equipped with a continuous $R$-linear action of $G_{\mathbb{Q},S}$. Assume that $H^1(I_\ell,V)$ is projective over $R$ for every $\ell\in S$, $\ell\neq p$.
Put
$D=D_{\mathrm{rig}}\left(V|_{G_{\mathbb{Q}_p}}\right),
$
and suppose that $D$ admits a fixed saturated triangulation
$0=\mathrm{Fil}^0D\subset \mathrm{Fil}^1D\subset\cdots\subset \mathrm{Fil}^dD=D
$
whose graded pieces are finite projective. Fix $0\leq a\leq d$, and put
$D^a=\mathrm{Fil}^aD.$

Let $\Delta$ be the torsion subgroup of $\Gamma_{\mathbb{Q}_p}$, and fix a primitive idempotent $e\in\mathbb{Z}_p[\Delta]$ corresponding to one character of $\Delta$.
For $x\in\mathrm{Max}(R)$, put
$\Lambda_{\kappa(x)}
= e\Lambda_{\kappa(x),(0,\infty]}(\Gamma_{\mathbb{Q}_p}),
$
and define
$h_{\varnothing}(x)
=\dim_{\kappa(x)} H^1\left(C^\bullet(V_x,D_x^a)\right)$
and
$h_{\mathrm{Iw}}(x)
=
\operatorname{rank}_{\Lambda_{\kappa(x)}}
\left(
eH^1\left(C_{\mathrm{Iw}}^\bullet(V_x,D_x^a)\right)
\right).$
For each $\star\in\{\varnothing,\mathrm{Iw}\}$ and each integer $q\geq0$, the locus
$\left\{
x\in\mathrm{Max}(R):
h_\star(x)\geq q
\right\}$
is Zariski closed. On each irreducible component of $\mathrm{Max}(R)$, the function $h_\star$ has a generic minimum, and the locus where it is strictly larger than this minimum is cut out on that component by a finitely generated ideal of $R$.

If, moreover, $R$ is a regular pseudoaffinoid domain of dimension $1$, then there are integers $n_a^\star\geq0$ such that
$h_\star(x)\geq n_a^\star$
for every $x\in\mathrm{Max}(R)$, with equality away from a finite set; in this case the rank-jump loci admit the more explicit descriptions of Theorem~\ref{selmerrankjump}.
Finally, at a characteristic-zero point $x$ for which the specialized fixed triangulation is non-critical and non-exceptional and has exactly $a$ strictly negative Hodge--Tate weights, one has
$h_{\varnothing}(x)=
\dim_{\kappa(x)}H_f^1(\mathbb{Q},V_x).$
\end{theorem}
Finally, we give a geometric picture of what we are trying to do. The motivation of studying such representations comes from the study of eigencurve over the boundary of weight space. The eigencurve is a moduli space of all modular forms of finite slope that have eigenvalues with respect to the Hecke operators acting on them. The usual modular forms studied in number theory have an integral weight $k$, which specifies the symmetry that the modular form satisfies.  One can think of $k \in \mathbb{Z}$ as a function from $\mathbb{Z}_p^{\times} \to \mathbb{Z}_p^{\times}$ given by $a \mapsto a^k$. The general notion of weight of a modular form with coefficients in $R$ is given by all continuous characters from $\mathbb{Z}_p^{\times} \to R^{\times}$. Thus, in general, the group algebra $\mathbb{Z}_p[[\mathbb{Z}_p^{\times}]]$ serves as a natural candidate to parametrize weights of modular forms. One can study the eigencurve or the weight space in different (related) contexts by taking formal/analytic/adic spectrum of the group algebra $\mathbb{Z}_p[[\mathbb{Z}_p^{\times}]]$. 
The analytic adic weight space $\text{Spa}(\mathbb{Z}_p[[\mathbb{Z}_p^{\times}]])^{an}$ is isomorphic to finitely many copies of $(\text{Spa } \mathbb{Z}_p[[T]])^{an}$, which naturally also contains pseudorigid coefficients. We show the picture of $\text{Spa } \mathbb{Z}_p[[T]]$ below and refer the reader to \cite{scholze2020berkeley}, Proposition 4.2.6 for an understanding of what this means. 
\begin{figure}[H]
\centering
\includegraphics[width=0.5\textwidth]{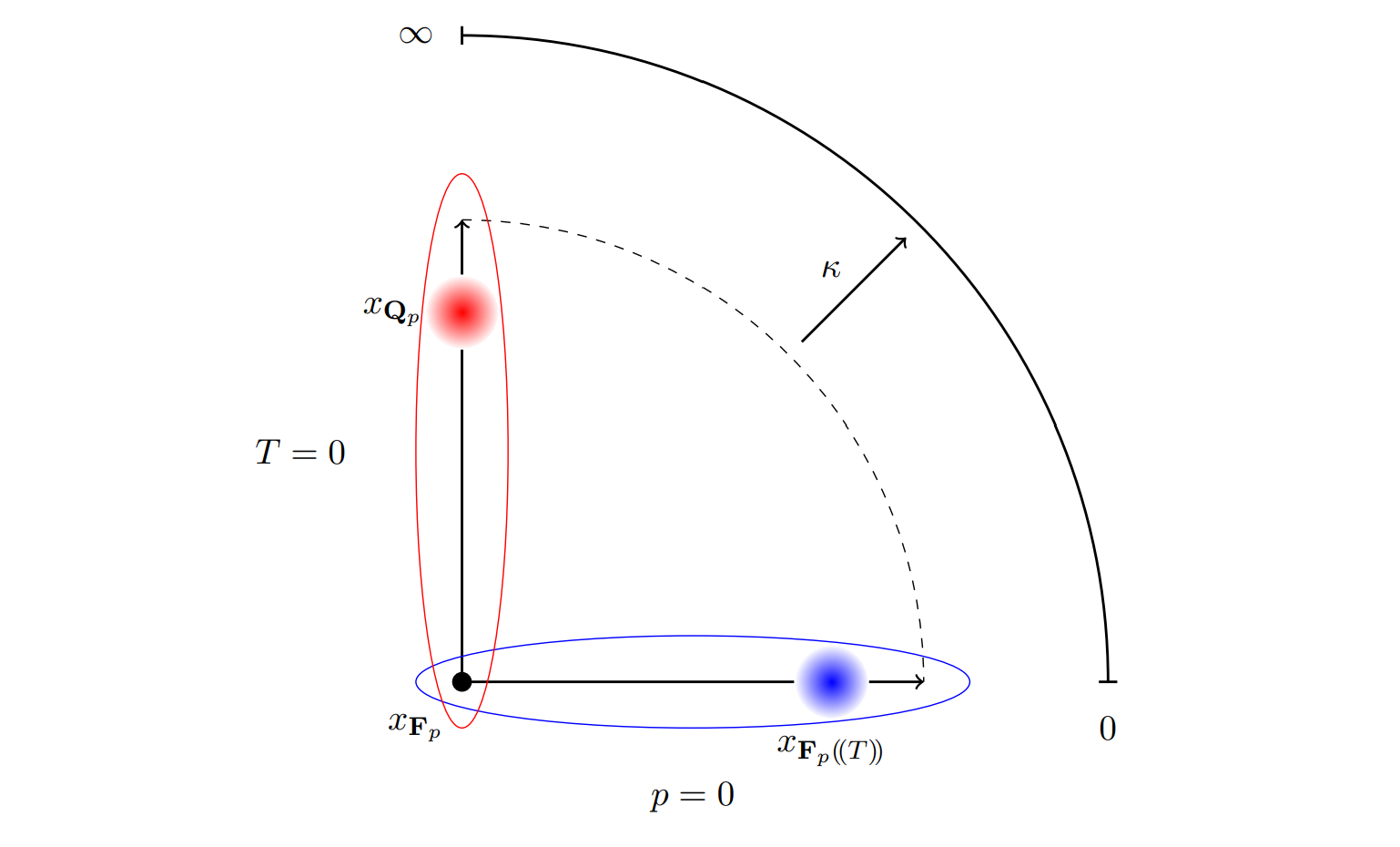}
\caption{Figure from \cite{scholze2020berkeley} that describes $\text{Spa } \mathbb{Z}_p[[T]]$}
\end{figure}
The earlier work of \cite{kedlaya2014cohomology} deals with proving perfection of Iwasawa complex for families in neighborhood of the $y$-axis shown above (where $p$ is invertible). Our work deals with proving perfection of Iwasawa complex for families over the lines dictated by $\kappa$.

\section{Notations and structure}
\label{notsec}
 Throughout this paper, we fix a prime number $p\neq 2$. We fix a number $\omega = p^{-1/(p-1)}$\\
 $\Gamma_n$ is the subgroup of $\mathbb{Z}_p^{\times}$ that consists of elements of the form $1 +p^n\mathbb{Z}_p$.\\
 $\chi$ is fixed to be the cyclotomic character. \\

 \subsection{Structure of the paper}
 Section 3 of the paper introduces the definitions of Robba rings in the pseudorigid setting. All these definitions are taken from \cite{bellovin2023galois} and \cite{bellovin2024cohomology}. In section 3.4, we define "group versions" of these Robba rings for one-dimensional $p$-adic Lie groups. These are defined because the Iwasawa cohomology groups are finite over such rings. We explain through various lemmas and theorems in section 5 on why the Iwasawa cohomology groups naturally live over the group version of Robba rings.
 Section 4 deals with giving definition of $(\varphi, \Gamma)$-modules, again from \cite{bellovin2024cohomology} and explaining why $(\varphi,\Gamma)$ modules viewed as sheaves (or vector bundles) on disks are the same as actual modules.
 Section 5 introduces the definition of Iwasawa complex, a base change theorem from Iwasawa to Galois (Herr) cohomology and develops machinery necessary to prove its perfectness.
 In section 6, we prove perfectness of the Selmer complex and study variation of ranks of Selmer group in analytic pseudorigid family.
 We now give an informal guide to the coefficient rings that appear below. This is an imprecise but intuitive way of thinking about these rings. Proper definitions are given in the main body of the paper. 
\begin{table}[H]
\begin{tabular}{|l|l|}
\hline
Rings                                             & Intuitive (but imprecise) way to think about this ring                                                                                                                                                                                                                                                                                \\ \hline
$\Lambda_{R, [a, b], K}$                          & \begin{tabular}[c]{@{}l@{}}Power series ring (negative powers allowed) with coefficients in the \\ pseudoaffinoid algebra $R$ and formal variable $\pi_K$ and growth\\ conditions of $\pi_K$ depending on $a, b$. The growth condition of \\ negative powers of $\pi_K$ depends on $b$ and positive powers depend on $a$.\end{tabular} \\ \hline
$\Lambda_{R, [a, \infty]}$                     & \begin{tabular}[c]{@{}l@{}}Power series ring (negative powers not allowed) with growth conditions\\ on positive powers of $\pi$ dependent on $a$. \\
This is a Banach algebra containing the distribution algebra below\end{tabular}                                                                                                                                                                      \\ \hline
$\Lambda_{R, (0, \infty]}$                     & \begin{tabular}[c]{@{}l@{}}Power series ring (negative powers not allowed) such that the power series \\ has to converge for all numbers $r<1$ such that $|\pi| = r$. This is a distribution algebra.\end{tabular}                                                                                                                                                   \\ \hline
$\Lambda_{R, [a, b]}(1+p\mathbb{Z}_p)$ & \begin{tabular}[c]{@{}l@{}}The ring $\Lambda_{R, [a, b], \mathbb{Q}_p}$ with formal variable\\ $\pi_{\mathbb{Q}_p}$ replaced by $\gamma-1$, where $\gamma$ is \\ topological generator of $1+p\mathbb{Z}_p$.\end{tabular}                                                                                                                \\ \hline
\end{tabular}
\end{table}
 \subsection{Acknowledgements} I would like to thank my advisor, David Savitt, for introducing me to pseudorigid coefficients and weekly meetings and helpful conversations. I am grateful to Toan Pham, Milton Lin and Rustam Steingart for helpful conversations. I am particularly grateful to Muhammad Manji for going through the initial draft of the paper and providing detailed feedback.

\subsection{Statement of AI use}
A complete first draft of the article was written without any AI use. The ideas and strategy are entirely due to the author and not AI, building on previous works of \cite{kedlaya2014cohomology} and \cite{bellovin2024cohomology}. The author then shared the draft with her PhD advisor David Savitt, who then used ChatGPT (5.5, 5.6 and 6) for pointing out errors and giving feedback. The LLM pointed out remaining gaps in the initial drafts that were resolved by the author.

The feedback generated in this process by the LLM suggested the use of Lemma \ref{uniformperfectcomplex} and the proof was given by the AI model, which was understood and written by the author. The author also used ChatGPT for stylistic editing of the paper.

 \section{Robba rings and their group versions}
\subsection{Robba rings}
Let $K$ be a finite extension of $\mathbb{Q}_p$.
\begin{definition}
A pseudoaffinoid $\mathcal{O}_K$-algebra is a complete Tate $\mathcal{O}_K$-algebra $A$ having a Noetherian ring of definition $A_0$, which is formally of finite type over $\mathcal{O}_K$. A homomorphism of pseudoaffinoid $\mathcal{O}_K$-algebras $A$ and $B$ is a continuous $\mathcal{O}_K$-homomorphism $\phi : A \to B$.
\end{definition}
For $K$ as above, we denote $\varpi_K$ as a chosen uniformizer for $\mathcal{O}_K$

\textbf{Example:}
Let $\lambda = \frac{n}{m} \in \mathbb{Q}_{>0}$ be a positive rational number 
with $(n,m) = 1$, and set 
$D^\circ_\lambda := \mathbb{Z}_p [[u]]\left\langle \frac{p^m}{u^n} \right\rangle$ 
and $D_\lambda := D^\circ_\lambda\left[\frac{1}{u}\right]$. 
Then $D_\lambda$ is a pseudoaffinoid algebra. \\
Every element of $D_\lambda$ can be written uniquely as a power series $\sum_{i \in \mathbb{Z}} a_i u^i$ with $a_i \in \mathbb{Z}_p$, and we define a valuation $v_{D_\lambda}$ on $D_\lambda$ via
$$
v_{D_\lambda}\!\left(\sum_{i \in \mathbb{Z}} a_i u^i\right)
:= \inf_i \left\{ v_p(a_i) + \frac{i}{\lambda} \right\}.
$$
We define the norm on $D_{\lambda}$ such that $|u|_{D_{\lambda}} = p^{-1/p}$.
Every pseudoaffinoid algebra $R$ is a topologically finite type 
$D_\lambda$-algebra for some sufficiently small $\lambda > 0$, 
by Lemma 4.8 \cite{lourencco2017riemannian}.\\

Let $K/\mathbb{Q}_p$ be a finite extension, let $\chi_{cyc}: \text{Gal}_K \to \mathbb{Z}_p^{\times}$ be the cyclotomic character, we let $H_K = \text{ker}(\chi_{cyc})$ and $\Gamma_K:= \text{Gal}_K/H_K$. Let $A_{inf} = W(\mathcal{O}_{\mathbb{C}_K}^{\flat})$. Let $\varepsilon:= (\varepsilon^0, \varepsilon^1, \ldots)$ be a compatible sequence of primitive $p$-power roots of unity. Let $\bar{\pi} = \varepsilon-1$ and $\pi = [\varepsilon]-1 \in A_{inf}$. The norm of the reduction of the element modulo $p$ is given by $|\bar{\pi}|_{\mathbb{C}_p^{\flat}} = p^{-p/(p-1)} $.\\
For any unramified extension $F/\mathbf{Q}_p$, the choice of $\varepsilon$ gives us a natural map
$k_F((\overline{\pi})) \to \mathbf{C}_K^{\flat}$; let $E_F$ denote its image, and let $\mathbf{E} \subset \mathbf{C}_K^{\flat}$ be its separable
closure. Then $\mathrm{Gal}(\mathbf{E}/E_F) \cong H_F$ (by the theory of the field of norms), and
for any extension $K/F$, we set $\mathbf{E}_K := \mathbf{E}^{H_K}$. Then $\mathbf{E}_K$ is a discretely valued field, and we may choose a uniformizer $\overline{\pi}_K$. Lifting its minimal polynomial to characteristic $0$, Hensel's lemma gives a lift
$
\pi_K \in W(\mathbf{C}_K^\flat)
$
which is integral over
$
\mathcal{O}_F[[\pi]]\left[\frac{1}{\pi}\right]^{\wedge}.
$
Here $F$ denotes the maximal subfield of $K(\mu_{p^\infty})$ that is unramified over $\mathbf{Q}_p$. We fix such a choice of $\pi_K$ and use it throughout. When $K=F$ is unramified over $\mathbf{Q}_p$, we take $\pi_F=\pi$.

The following is Definition 3.40, \cite{bellovin2023galois}.
\begin{definition}
    \label{lambdaqp}
   
Let $R$ be a pseudoaffinoid $\mathbb{Z}_p$-algebra such that $R$ is topologically of finite type over $D_\lambda$, and let $K / \mathbb{Q}_p$ be a finite extension. 
Fix rational numbers $a \in \mathbb{Q}_{>0}$ and $b \in \mathbb{Q}_{\geq 0}$ with $a \leq b $ such that 
$1/(a \cdot v_{\mathbb{C}_p^{\flat}}(\overline{\pi}_K))$ and $1/(b \cdot v_{\mathbb{C}_p^{\flat}}(\overline{\pi}_K))$ belong to $\mathbb{Z}$. 
Then we define the $\mathbb{Z}_p$-algebra $\Lambda_{R_0,[a,b],K}$ to be the evaluation of the sheaf of rings $\mathcal{O}_{(R_0 \otimes \mathcal{O}_F)[[\pi_K]]}$ 
on the affinoid subspace of $\operatorname{Spa}(R_0 \otimes \mathcal{O}_F)[[\pi_K]]$ defined by the conditions 
$u \leq \pi_K^{1/(b \cdot v_{\mathbb{C}_p^{\flat}}(\overline{\pi}_K))}$ and $\pi_K^{1/(a \cdot v_{\mathbb{C}_p^{\flat}}(\overline{\pi}_K))} \leq u$. 

This is a pseudoaffinoid algebra with ring of definition\\
$\Lambda_{R_0,[a,b],0,K} := (R_0 \otimes \mathcal{O}_F)[[\pi_K]] 
\left\langle 
\frac{u}{\pi_K^{1/(b \cdot v_{\mathbb{C}_p^{\flat}}(\overline{\pi}_K))}}, 
\frac{\pi_K^{1/(a \cdot v_{\mathbb{C}_p^{\flat}}(\overline{\pi}_K))}}{u}
\right\rangle.$\\
For an interval $I$, $\Lambda_{R, I, K}:= \Lambda_{R_0, I, K}\left[\dfrac{1}{u}\right]$
In the case $b = \infty$, define $$\Lambda_{R_0, [a, \infty]} := R_0 [[\pi]] \left\langle \frac{\pi^{1/(a \cdot v_{\mathbb{C}_p^{\flat}}(\overline{\pi}))}}{u}
\right \rangle $$
and $\Lambda_{R, [a, \infty]} := \Lambda_{R_0, [a, \infty]}[1/u]$.\\
Clearly $\Lambda_{R, [a, \infty]}$ is a pseudoaffinoid algebra with Noetherian ring of definition $\Lambda_{R_0, [a, \infty]}$. \\
We define $\Lambda_{R, (0, b], ?}:= \varprojlim_{a \to 0}\Lambda_{R,[a,b],?}$ and $\Lambda_{R, \text{rig}, ?}:= \varinjlim_{b \to 0}\Lambda_{R, (0, b], ?}$, for $? \in \{\emptyset, K\}$.
\end{definition}

\begin{lemma}
\label{ringofdef}
    Let $R$ be a pseudoaffinoid $D_{\lambda}$ algebra. The ring $\Lambda_{R_0,[a,b],0,\mathbb{Q}_p}[1/u]$ is a subset of the ring containing power series $f(\pi)=\sum_{i\in\mathbb{Z}}d_i\pi^i$ such that $|d_i|\omega^{ai}\to0$ as $i\to\infty$ and $|d_i|\omega^{bi}\to0$ as $i\to-\infty$, with $d_i\in R$.
\end{lemma}

\begin{proof}
    Let $f(\pi)\in\Lambda_{R_0,[a,b],0,\mathbb{Q}_p}[1/u]$, then for some power $m$, $g(\pi)=u^m f(\pi)\in\Lambda_{R_0,[a,b],0,\mathbb{Q}_p}$. Let $g(\pi)=g_+(\pi)+g_-(\pi)$, where $g_+$ consists of non-negative powers of $\pi$ and $g_-$ consists of negative powers of $\pi$. Let
    $$
    g_+(\pi)=g_0(\pi)+\frac{c_1}{u}g_1(\pi)+\cdots+\frac{c_n}{u^n}g_n(\pi)+\cdots
    $$
    with $g_n(\pi)$ a power series in $\pi$ with all its coefficients in $R_0$ and $|c_n|\to0$. From the definition of $\Lambda_{R_0,[a,b],0,\mathbb{Q}_p}$, we must have that the lowest power of $\pi$ in $g_n(\pi)$ is $\pi^{\frac{n}{av_{\mathbb{C}_p^\flat}(\bar{\pi})}}$. Write $g_n(\pi)=\sum_j a_{n,j}\pi^j$. Then $|a_{n,j}|\leq1$ and $a_{n,j}=0$ unless $j\geq\frac{n}{a v_{\mathbb{C}_p^\flat}(\bar{\pi})}$.

    By our normalization, $|u|=\omega^{1/v_{\mathbb{C}_p^\flat}(\bar{\pi})}$. Hence, whenever $a_{n,j}\neq0$, we have
    $$
    |u|^{-n}\omega^{aj}
    \leq |u|^{-n}\omega^{n/ v_{\mathbb{C}_p^\flat}(\bar{\pi})}
    =1.
    $$
    Therefore,
    $$
    \left|\frac{c_n}{u^n}a_{n,j}\right|\omega^{aj}\leq |c_n|.
    $$

    Suppose that $g_+(\pi)=\sum_{j\geq0}d_j\pi^j$. We show that $|d_j|\omega^{aj}\to0$. Let $\epsilon>0$. Since $|c_n|\to0$, there exists $N$ such that $|c_n|<\epsilon$ for all $n>N$. For $n>N$, the observation above gives
    $$
    \left|\frac{c_n}{u^n}a_{n,j}\right|\omega^{aj}<\epsilon.
    $$
    On the other hand, for each of the finitely many $n\leq N$, we have
    $$
    \left|\frac{c_n}{u^n}a_{n,j}\right|\omega^{aj}
    \leq |c_n||u|^{-n}\omega^{aj}\to0
    $$
    as $j\to\infty$. Therefore, by the non-archimedean triangle inequality, $|d_j|\omega^{aj}<\epsilon$ for all sufficiently large $j$. Hence $|d_j|\omega^{aj}\to0$ as $j\to\infty$.

   First assume $b< \infty$. Put $e_b=1/(b\,v_{\mathbb{C}_p^\flat}(\bar{\pi}))$,
which is a positive integer by our convention.
The integral annulus presentation gives
$$
g_-(\pi)
=
\sum_{s=0}^{e_b-1}\sum_{i\geq1}
r_{i,s}u^i\pi^{s-ie_b},
$$
where $r_{i,s}\in R_0$ and $|r_{i,s}|\to0$
as $i\to\infty$, for each $s$.
This expansion includes every negative exponent of $\pi$.

Since $be_b=1/v_{\mathbb{C}_p^\flat}(\bar{\pi})$
and $|u|=\omega^{1/v_{\mathbb{C}_p^\flat}(\bar{\pi})}$,
we have
$$
|r_{i,s}u^i|\omega^{b(s-ie_b)}
\leq
|r_{i,s}||u|^i\omega^{bs-i/v_{\mathbb{C}_p^\flat}(\bar{\pi})}
=
|r_{i,s}|\omega^{bs}
\longrightarrow0.
$$
There are only finitely many residue classes $s$.
Hence, writing $g_-(\pi)=\sum_{j<0}d_j\pi^j$,
we obtain $|d_j|\omega^{bj}\to0$ as $j\to-\infty$.

In case $b = \infty$, there are only positive powers of $\pi$, so that case is handled.
\end{proof}

Given an element $f(\pi) = \sum_{i \in \mathbb{Z}} c_i \pi^{i}$ in $\Lambda_{R, [a, b], \mathbb{Q}_p}$, define $|.|_r = sup_{i \in \mathbb{Z}} |c_i| |\omega^{ri}|$. One can define a norm on $\Lambda_{R_0, [a, b], \mathbb{Q}_p}$ $|.| = max(|.|_a, |.|_b)$.\\

Since $\Lambda_{R, [a, b], \mathbb{Q}_p}$ is a Tate algebra with Noetherian ring of definition, the adic space $\mathrm{Spa}(\Lambda_{R, [a, b], \mathbb{Q}_p}, \Lambda_{R_0, [a, b], 0, \mathbb{Q}_p})$ (denoted shorthand by $\mathrm{Spa}(\Lambda_{R, [a, b], \mathbb{Q}_p})$) behaves like a Noetherian affine space.

\begin{lemma}
    \label{Noetheriantate}
    The ring $\Lambda_{R, [a, b],K}$ is Noetherian.
\end{lemma}
\begin{proof}
    Since $\Lambda_{R, [a, b],K}$ is a Tate ring with a Noetherian ring of definition, our assertion follows from Remark 6.37(3), \cite{wedhorn2019adic}.
\end{proof}

    \begin{remark}
\label{compdifflamb}
For $K$ a finite extension of $\mathbb{Q}_p$ and $b$
sufficiently small, the ring $\Lambda_{R,(0,b],K}$ is
naturally a finite etale algebra over
$\Lambda_{R,(0,b],\mathbb{Q}_p}$, free of rank
$[K(\mu_{p^\infty}):\mathbb{Q}_p(\mu_{p^\infty})]$.
The same comparison holds over closed intervals $[a,b]$.
We explain the difference between this statement and
Remark~2.2.4 of \cite{kedlaya2014cohomology}.
Write $\nu=p/(p-1)$, and let $\tilde e_K$ be the
ramification index of
$K(\mu_{p^\infty})/\mathbb{Q}_p(\mu_{p^\infty})$.
Then
$
v_{\mathbb{C}_p^\flat}(\bar\pi_K)
=
\frac{\nu}{\tilde e_K}.
$
Our Definition~\ref{lambdaqp} incorporates this valuation
into the interval parameter. More precisely, with our
normalization $|u|=p^{-1/p}$ and
$\omega=p^{-1/(p-1)}$, the coordinate radius associated
with a parameter $t$ is
\[
|\pi_K|_t
=
|u|^{t v_{\mathbb{C}_p^\flat}(\bar\pi_K)}
=
\omega^{t/\tilde e_K}.
\]
In contrast, Definition~2.2.2 of
\cite{kedlaya2014cohomology} uses the coordinate radius
$\omega^r$ for $\pi_K$. Thus our interval $[a,b]$
has coordinate endpoints corresponding to
$[a/\tilde e_K,b/\tilde e_K]$ in their normalization.
Consequently, the factor $1/\tilde e_K$ appearing in
their comparison is already incorporated into our
definition. In our notation, the comparison therefore
uses the same interval $[a,b]$ on both sides.

Passing to the full Robba rings, we obtain that
$\Lambda_{R,\mathrm{rig},K}$ is finite etale and free
over $\Lambda_{R,\mathrm{rig},\mathbb{Q}_p}$ of rank
$[K(\mu_{p^\infty}):\mathbb{Q}_p(\mu_{p^\infty})]$. We define $e_K = [\Gamma: \Gamma_K]$, where $\Gamma = \Gamma_{\mathbb{Q}_p}$
\end{remark}

\begin{lemma}
    \label{robbaflat}
    The rings $\Lambda_{R, \text{rig}, K}$ and $\Lambda_{R, (0, b], K}$ are flat over $R$.
\end{lemma}
\begin{proof}
First consider a closed annulus. Put
$e_a=1/(a\,v_{\mathbb{C}_p^{\flat}}(\overline\pi_K))$ and
$e_b=1/(b\,v_{\mathbb{C}_p^{\flat}}(\overline\pi_K))$.
As a topological $R_0\otimes_{\mathbb{Z}_p}\mathcal O_F$-module,
its ring of definition is the $u$-adic completion of the
free module with normal monomials
$$
\left\{
u^{-n}\pi_K^{ne_a+r}:
n\geq0,\ 0\leq r<e_a
\right\}
\ \cup\
\left\{
u^n\pi_K^{r-ne_b}:
n\geq1,\ 0\leq r<e_b
\right\}.
$$
Since $\mathcal O_F$ is finite free over $\mathbb{Z}_p$,
this is also the completion of a free $R_0$-module
after choosing a $\mathbb{Z}_p$-basis of $\mathcal O_F$.
The coefficient completion theorem
\cite[Tag~06LE]{stacks-project} therefore gives flatness over
$R_0$. Inverting $u$ gives flatness of
$\Lambda_{R,[a,b],K}$ over $R=R_0[1/u]$.

For the open annulus, take a decreasing sequence $a_n\to0$
and let $J\subset R$ be a finitely generated ideal.
The finite modules
$J\otimes_R\Lambda_{R,[a_n,b],K}$ form a compatible
coadmissible system. A finite presentation of $J$,
together with exactness of the coadmissible inverse-limit
functor, identifies
$
J\otimes_R\Lambda_{R,(0,b],K}
\cong
\varprojlim_n
\left(J\otimes_R\Lambda_{R,[a_n,b],K}\right).
$
At every closed stage, flatness makes
$J\otimes_R\Lambda_{R,[a_n,b],K}
\longrightarrow\Lambda_{R,[a_n,b],K}$ injective.
The inverse limit is therefore injective as well.
The ideal criterion for flatness proves that
$\Lambda_{R,(0,b],K}$ is $R$-flat.
Finally, the Robba ring is a filtered colimit of these
open-annulus rings and is consequently $R$-flat.
\end{proof}

\subsection{Topology on Robba rings}
The contents of this section are standard, but we briefly recall the relevant topologies for completeness.

The ring $\Lambda_{R,[a,b],\mathbb{Q}_p}$ is a Banach $R$-algebra equipped with the norm
$|\cdot|_{a,b}:=\max(|\cdot|_a,|\cdot|_b).$
If $f(T)=\sum_{i\in\mathbb Z}a_iT^i\in\Lambda_{R,[a,b]},$
then
$
|f|_{a,b}
=
\max\left(
\sup_{i\ge0}|a_i|\omega^{ai},
\;
\sup_{i\le0}|a_i|\omega^{bi}
\right).
$

The ring
$\Lambda_{R,(0,b],\mathbb Q_p}
=
\varprojlim_a
\Lambda_{R,[a,b],\mathbb Q_p}$
is equipped with the inverse limit topology in the category of topological $R$-modules. Equivalently, it is the coarsest topology for which every projection
$
\pi_a:\Lambda_{R,(0,b],\mathbb Q_p}
\longrightarrow
\Lambda_{R,[a,b],\mathbb Q_p}
$
is continuous. This topology is defined by the family of seminorms
$
f\mapsto |\pi_a(f)|_{a,b}.
$
A fundamental system of neighbourhoods of $0$ is given by finite intersections of inverse images of open balls in the Banach $R$-algebras
$\Lambda_{R,[a,b],\mathbb Q_p}$.
Since we can choose a countable cofinal sequence of values of $a$ tending to $0$, this makes $\Lambda_{R,(0,b],\mathbb Q_p}$ a complete metrizable topological $R$-module.

Finally, $\Lambda_{R,\mathrm{rig},\mathbb Q_p}
=
\varinjlim_b
\Lambda_{R,(0,b],\mathbb Q_p}
$
is equipped with the inductive limit topology in the category of topological $R$-modules. That is, it is the finest topology of a topological $R$-module for which every natural inclusion
$
i_b:
\Lambda_{R,(0,b],\mathbb Q_p}
\longrightarrow
\Lambda_{R,\mathrm{rig},\mathbb Q_p}
$
is continuous. Since each
$\Lambda_{R,(0,b],\mathbb Q_p}$
is a complete metrizable topological $R$-module, this gives a countable inductive limit of complete metrizable topological $R$-modules.
\subsection{Duality of Robba rings}

Denote $\Omega_{\Lambda_{R, \text{rig}, \mathbb{Q}_p}} = \widehat{\Omega}_{\Lambda_{R, \text{rig}, \mathbb{Q}_p}/R}$, the module of continuous $R$-linear differentials of $\Lambda_{R, \text{rig}, \mathbb{Q}_p}$. If one denotes $\pi_{\mathbb{Q}_p}$ by $T$, the basis of $\Omega_{\Lambda_{R, \text{rig}, \mathbb{Q}_p}}$ is generated by $dT$, which we prove in the lemma below:

\begin{lemma}
\label{nongroupdiff}
    The module $\Omega_{\Lambda_{R, \text{rig}, \mathbb{Q}_p}}$ over $\Lambda_{R, \text{rig}, \mathbb{Q}_p}$ is free of rank 1 generated by $dT$.
\end{lemma}
\begin{proof}
    It is enough to show that $\Omega_{\Lambda_{R, [a, b], \mathbb{Q}_p}}$ is free of rank one generated by $dT$ over $\Lambda_{R, [a, b], \mathbb{Q}_p}$. Given $f(T) = \sum_{i \in \mathbb{Z}} a_i T^i \in \Lambda_{R, [a, b], \mathbb{Q}_p}$, for $i \geq 0$ $ \lim_{i \to \infty}|a_i| \omega^{ia} = 0$ and $\lim_{i \to -\infty} |a_i| \omega^{bi} = 0$ then $d f(T) = \sum_{i \in \mathbb{Z}} \bar{i}a_i T^{i-1} dT$. For $i \geq 0, |a_i| \omega^{a(i-1)} \to 0 $, and the same thing holds for $i \leq 0$ and $b$. So, $\sum_{i \in \mathbb{Z}} \bar{i}a_i T^{i-1} \in \Lambda_{R, [a, b], \mathbb{Q}_p}$. Hence, the assertion is proved.
\end{proof}
We define the residue map $Res: \Omega_{\Lambda_{R, [a, b], \mathbb{Q}_p}} \to R$ as $Res(\sum a_n T^n dT) = a_{-1}$
\begin{equation}
    \label{pairr}
    R:\Lambda_{R, \text{rig}, \mathbb{Q}_p} \times \Omega_{\Lambda_{R, \text{rig}, \mathbb{Q}_p}} \to R 
\end{equation}
that sends $(f, \omega) \to Res(f \omega)$, that is the coefficient of $\frac{dT}{T}$ in $f \omega$.
We will now check that the pairing in \ref{pairr} defines a perfect pairing. The idea of the proof is the same as that of \cite{crew1998finiteness}, theorem 5.4 and \cite{kedlaya2014cohomology}, Lemma 2.1.19.

\begin{theorem}
    \label{robbaduality}
    The pairing defined by:
    \begin{equation}
    \label{pair}
    \Phi:\Lambda_{R, \text{rig}, \mathbb{Q}_p} \times \Omega_{\Lambda_{R, \text{rig}, \mathbb{Q}_p}} \to R 
\end{equation}
that sends $(f, \omega) \mapsto Res(f \omega)$ is perfect.\\
This gives us a natural pairing:
\begin{equation}
    \label{pairrobba}
    \Lambda_{R, \text{rig}, \mathbb{Q}_p} \times \Lambda_{R, \text{rig}, \mathbb{Q}_p} \to R
\end{equation}
sending $(f(T), g(T)) \mapsto Res(f(T) g(T) dT)$.
\end{theorem}
\begin{proof}
To prove that the pairing \eqref{pair} is perfect, it suffices to show that the map
$\Phi:\Lambda_{R,\mathrm{rig},\mathbb{Q}_p}\to
\operatorname{Hom}_{R,\mathrm{cont}}(\Lambda_{R,\mathrm{rig},\mathbb{Q}_p}, R)$,
given by $\Phi(f)(g)=\operatorname{Res}(fg dT)=a_{-1}(fg)$, is a continuous
isomorphism.

Let $f\in\Lambda_{R,(0,b'],\mathbb{Q}_p}$ and
$g\in\Lambda_{R,(0,b''],\mathbb{Q}_p}$. For every closed interval
$I=[a,b]$, we have
$$
|a_{-1}(fg)|\,\omega^{-b}\le |f|_I|g|_I,
$$
and hence
$|a_{-1}(fg)|\le |f|_I|g|_I\omega^b$. Thus, for every $b$, the restriction
of $\Phi$ to $\Lambda_{R,(0,b],\mathbb{Q}_p}$ is continuous. Since
$\Lambda_{R,\mathrm{rig},\mathbb{Q}_p}
=\varinjlim_b\Lambda_{R,(0,b],\mathbb{Q}_p}$
is equipped with the LF topology, it follows that $\Phi$ is continuous. 

Given
$\mu\in\operatorname{Hom}_{R,\mathrm{cont}}
(\Lambda_{R,\mathrm{rig},\mathbb Q_p},R)$, put
$$
f_\mu(T)=\sum_{n\in\mathbb Z}\mu(T^{-n-1})T^n.
$$
We check that this series belongs to the Robba ring.
For any $b_0>0$, continuity of the restriction of $\mu$
to $\Lambda_{R,(0,b_0],\mathbb Q_p}$ gives
$0<a\leq b\leq b_0$ and $C>0$ such that
$$
|\mu(g)|\leq C|g|_{[a,b]}
\qquad
(g\in\Lambda_{R,(0,b_0],\mathbb Q_p}).
$$
In particular this holds for all Laurent monomials.
For $q\geq0$, it gives
$|\mu(T^q)|\leq C\omega^{qa}$.
If $0<l<a$, then
$$
|\mu(T^q)|\omega^{-(q+1)l}
\leq C\omega^{-l}\omega^{q(a-l)}\longrightarrow0.
$$
Thus the negative-power part
$\sum_{q\geq0}\mu(T^q)T^{-q-1}$ converges on all
sufficiently small Robba annuli.

For the nonnegative-power part, fix $l>0$ and first choose
$0<b_0<l$. Apply the preceding continuity bound on that
model, so that $b\leq b_0<l$. Then
$$
|\mu(T^{-q-1})|\omega^{ql}
\leq C\omega^{-(q+1)b}\omega^{ql}
=C\omega^{-b}\omega^{q(l-b)}\longrightarrow0.
$$
Hence $\sum_{q\geq0}\mu(T^{-q-1})T^q$ converges at
every positive radius. Together these estimates show that
$f_\mu\in\Lambda_{R,\mathrm{rig},\mathbb Q_p}$.
Moreover,
$\operatorname{Res}(f_\mu T^m\,dT)=\mu(T^m)$
for every $m\in\mathbb Z$.
By continuity and density of Laurent polynomials, the
functional defined by $f_\mu$ is $\mu$.
Conversely, the coefficient formula recovers $f$ from
the functional $g\mapsto\operatorname{Res}(fg\,dT)$.
This constructs the asserted inverse.
\end{proof}
\begin{remark}
    
\label{indetind}
    The residue map $\text{Res}$ depends only on $\Lambda_{R, \text{rig}, \mathbb{Q}_p}$ up to $R$-linear continuous
isomorphism, and not on the choice of indeterminate $T$.
\end{remark}
\subsection{Group version of Robba rings}
\begin{definition}
    Assume that we are given the data of a topological abelian group $C$ whose torsion group $C_{tor}$ is finite and $C/C_{tor}$ is non-canonically isomorphic to $\mathbb{Z}_p$. \\
    Let $R$ be a pseudoaffinoid $\mathbb{Z}_p$ algebra such that it is topologically of finite type over $D_{\lambda}$. Assume for now that $C$ is torsion free. For any interval $I$, define $\Lambda_{R,I,\mathbb{Q}_p}(C)$, also denoted simply by $\Lambda_{R,I}(C)$, as in \cite{kedlaya2014cohomology}: starting from the definition of $\Lambda_{R,I,\mathbb{Q}_p}$, replace the variable $\pi$ by $\gamma-1$, where $\gamma$ is a fixed topological generator of $C$. Since the coefficient field plays no role in this group-theoretic construction, we usually omit $\mathbb{Q}_p$ from the notation. As in Definition \ref{lambdaqp}, we also allow $I = [a, b], (0, b]$ where $b=\infty$. Define $\Lambda_{R, \text{rig}}(C) = \varinjlim_{b \to 0} \Lambda_{R, (0, b]}(C)$ \\
For a torsion-free group $C$, let $T=\gamma-1$.
In mixed characteristic, the change of variable
$T\mapsto(1+T)^p-1$ satisfies
$$
|f((1+T)^p-1)|_r=|f(T)|_{pr}
$$
provided $0<r<\lambda/p$.
Indeed, $|p|\leq p^{-\lambda/p}$, and this condition implies
$|p|\omega^{ir}\leq\omega^{pr}$ for $1\leq i<p$.
For $p=0$ in $R$, the equality holds for every $r>0$.
Consequently,
$$
\Lambda_{R,[a,b]}(C)
\cong\Lambda_{R,[pa,pb]}(C^p)
\otimes_{\mathbb{Z}[C^p]}\mathbb{Z}[C]
$$
for $0<a\leq b<\lambda/p$, and
$$
\Lambda_{R,[a,\infty]}(C)
\cong\Lambda_{R,[pa,\infty]}(C^p)
\otimes_{\mathbb{Z}[C^p]}\mathbb{Z}[C]
$$
for $0<a<\lambda/p$.
In characteristic $p$, no such small-radius restriction is needed.

For a group $C$ with finite torsion, choose a torsion-free open
subgroup $D\subset C$ and use the full subgroup index:
$$
\Lambda_{R,[a,b]}(C)
=\Lambda_{R,[[C:D]a,[C:D]b]}(D)
\otimes_{\mathbb{Z}[D]}\mathbb{Z}[C].
$$
The same convention applies when $b = \infty$. To compare two subgroup
descriptions, we pass to a common torsion-free open subgroup and
apply the preceding formulas at sufficiently small radii.
For an index $p^h$ between torsion-free groups, it suffices that
$p^hb<\lambda$ on an annulus or $p^ha<\lambda$ when $b = \infty$. When
a finite torsion factor has already rescaled the coordinate radii, apply these inequalities to those
rescaled radii. Thus, the subgroup descriptions are canonically compatible at sufficiently small radii; in particular, the resulting ring \(\Lambda_{R,\mathrm{rig}}(C)\) is independent of the choice of torsion-free open subgroup.
\end{definition}

\textbf{Description of $\Omega_{\Lambda_{R, \text{rig}, \mathbb{Q}_p}(C)}$:}
    Given any non-torsion $c' \in C$ such that the closed subgroup generated by it is $C' \subset C$, then we define $\omega = \frac{1}{[C: C']}\frac{dc'}{c'} \in \Omega_{\Lambda_{R, \text{rig}, \mathbb{Q}_p}(C)}$ whenever $[C:C']$ is a $p$-adic unit. 
     \begin{lemma}
      The space $\Omega_{\Lambda_{R, \text{rig}, \mathbb{Q}_p}(C)}$ is free of rank 1 over $\Lambda_{R, \text{rig}, \mathbb{Q}_p}(C)$ generated by $\omega$ as defined above whenever $|C_{tors}|$ is coprime to $p$.
    \end{lemma}
    \begin{proof}
If $C$ is torsion free, the result follows from Lemma \ref{nongroupdiff}.
Otherwise choose a torsion-free direct factor $D\subset C$ with
$[C:D]=|C_{\mathrm{tors}}|$ prime to $p$.
For $c\in C_{\mathrm{tors}}$ of order $m$, differentiating $c^m=1$ gives
$mc^{m-1}dc=0$. Since $m$ and $c$ are units, $dc=0$, and therefore
$d(f\otimes c)=(1\otimes c)df$.
Moreover, the extension from the group Robba ring for $D$ to that for $C$
is finite \'etale, since $|C_{\mathrm{tors}}|$ is invertible. Its module of
continuous differentials is the scalar extension of the module for $D$.
Thus it is free of rank one generated by $\omega$.
\end{proof}
Fix a torsion free subgroup $D \subset C$ so that $[C:D]$ is prime to $p$. The residue pairing map $\text{Res}_{{\Lambda_{R, rig}}(D)}$ is defined the same way on $\Lambda_{R, \text{rig}, D}$, then we can extend $Res_C$ to all of $\Omega_{\Lambda_{R, \text{rig}}(C)}$ as follows:
    \begin{equation}
        \label{grouppair}
         \text{Res}_C : \Omega_{\Lambda_{R, \text{rig}, \mathbb{Q}_p}(C)} \to R, \qquad f\omega \mapsto \text{Res}_D(\text{Tr}_{\Lambda_{R, \text{rig}, \mathbb{Q}_p}(C)/\Lambda_{R, \text{rig}, \mathbb{Q}_p}(D)}f \omega)
    \end{equation}
    This gives us natural pairing:
    \begin{equation}
        \{,\} : \Lambda_{R, rig, \mathbb{Q}_p}(C) \times \Lambda_{R, rig, \mathbb{Q}_p}(C) \to R, 
        (f,g) \mapsto Res_C(fg\omega)
    \end{equation}
    This induces topological isomorphisms\footnote{
We equip $\operatorname{Hom}_{R,\mathrm{cont}}(M,R)$
with the strong dual topology, i.e., the topology
of uniform convergence on bounded subsets of $M$.
More precisely, for every bounded subset $B\subset M$,
we consider the seminorm
$
p_B(f)=\sup_{m\in B}|f(m)|_R.$
For quotient modules, we use the quotient topology
on the source. These are the dual topologies
considered in Lemma 2.1.19 of \cite{kedlaya2014cohomology}.
}:\\
    \begin{equation}
\label{cycloimp}
\begin{aligned}
\Lambda_{R,\mathrm{rig}}(C)
&\cong
\operatorname{Hom}_{R,\mathrm{cont}}
\bigl(\Lambda_{R,\mathrm{rig}}(C),R\bigr),\\
\Lambda_{R,(0,\infty]}(C)
&\cong
\operatorname{Hom}_{R,\mathrm{cont}}
\bigl(
\Lambda_{R,\mathrm{rig}}(C)/
\Lambda_{R,(0,\infty]}(C),
R
\bigr).
\end{aligned}
\end{equation}
\begin{lemma}\label{max}
Write $\Gamma_K=\Delta_K\times\Gamma_K^{\mathrm{tf}}$, put
$d_K=|\Delta_K|$, and choose a generator $\gamma$ of
$\Gamma_K^{\mathrm{tf}}$. For an admissible radius $a$, put
$$e(a)=\frac{1}{d_Ka\,v_{\mathbb{C}_p^\flat}(\bar\pi)}.$$
Every maximal point of $\Lambda_{R,[a,\infty]}(\Gamma_K)$
has an evaluation map with residue field $L$, a coefficient map
$\lambda:R\to L$ whose kernel is maximal, and a continuous
character $\eta:\Gamma_K\to L^\times$ satisfying
$$|\eta(\gamma)-1|_L^{e(a)}\leq|\lambda(u)|_L.$$
Here $L$ is finite over $R/\ker(\lambda)$, and is generated over
that field by the values of $\eta$. Conversely, such compatible
data determine a maximal point, up to isomorphism of the residue
field. The coefficient map is part of the point data.

The same data define an evaluation map on
$\Lambda_{R,(0,\infty]}(\Gamma_K)$ whenever
$|\eta(\gamma)-1|_L<1$. Its kernel is the actual maximal
evaluation ideal, and its quotient is $L$.
\end{lemma}
\begin{proof}
The residue field at a maximal point of the Banach stage is a
finite extension of the coefficient residue field. Evaluating the
group elements gives $\eta$, and the integral presentation
$$R_0[\Delta_K][[T]]\left\langle T^{e(a)}/u\right\rangle,
\qquad T=\gamma-1,$$
gives precisely the displayed radius inequality.
Conversely, the coefficient map and character define evaluation
on this presentation when that inequality holds. The image
contains the field generated by the coefficient residue field and
the character values. This finite extension is complete, so the
values of convergent power series also lie in it. With the stated
choice of $L$, evaluation is surjective and its kernel is maximal.

For the full distribution algebra, evaluation at $T=\eta(\gamma)-1$
is defined when $|T|_L<1$. This strict inequality also follows
from continuity of a character on the pro-$p$ factor:
$\eta(\gamma)^{p^n}\to1$.
For large $n$, reduction in the characteristic-$p$ residue field
gives $(\overline{\eta(\gamma)}-1)^{p^n}=0$, hence
$\overline{\eta(\gamma)}=1$.
Thus $|\eta(\gamma)-1|_L<1$.
\end{proof}

\section{$(\varphi, \Gamma)$ bundles and $(\varphi, \Gamma)$-modules}
\begin{definition}
    A pseudorigid space $Y$ is said to be quasi-stein if it admits an admissible covering by a rising union of pseudoaffinoid subdomains $Y_0 \subset Y_1 \subset \ldots Y_n \subset \ldots$  such that the transition maps $\Gamma(Y_{n+1}, \mathcal{O}_{Y_{n+1}}) \to \Gamma(Y_{n}, \mathcal{O}_{Y_{n}})$ are flat with dense image.
\end{definition}
\textbf{Example:} The space $X = \bigcup_{a \to 0} \textrm{Spa}(\Lambda_{R, [a, b], K})$ is a quasi-Stein space where we denote $X_n := \textrm{Spa}(\Lambda_{R, [1/p^n, b], K})$.\\
A coherent sheaf
$\mathscr{F}$ on $Y$ is simply a compatible system of coherent sheaves
$\{\mathscr{F}_n\}$ on $\{Y_n\}$. Then
$$
F_\infty := \Gamma(Y,\mathscr{F})
= \varprojlim_n \Gamma(Y_n,\mathscr{F}_n)
$$
is a coadmissible module over $A_\infty = \varprojlim \Gamma(Y_n, \mathcal{O}_{Y_n})$, in the sense of \cite{schneider2002algebras}.\\
We take the following definition from \cite{bellovin2024cohomology}, section 2.2.
\begin{definition}
    A $\varphi$ module over $\Lambda_{R, (0, b], K}$ is a coherent sheaf of modules over the pseudorigid quasi-Stein space $\bigcup_{a \to 0} \textrm{Spa}(\Lambda_{R, [a, b], K})$ equipped with an isomorphism
    $$\varphi_D: \varphi^*D \to \Lambda_{R, (0, b/p], K} \otimes_{ \Lambda_{R, (0, b], K}}D$$
    If $a \in (0,b/p]$, a $\varphi$-module over $\Lambda_{R,[a,b],K}$ is a finite 
$\Lambda_{R,[a,b],K}$-module $D$ equipped with an isomorphism
$$
\varphi_{D,[a,b/p]} :
\Lambda_{R,[a,b/p],K} \otimes_{\Lambda_{R,[a/p,b/p],K}} \varphi^{*}D
\xrightarrow{\sim}
\Lambda_{R,[a,b/p],K} \otimes_{\Lambda_{R,[a,b],K}} D.
$$

A $(\varphi,\Gamma_{K})$-module over $\Lambda_{R,(0,b],K}$ 
(resp. $\Lambda_{R,[a,b],K}$) is a $\varphi$-module over 
$\Lambda_{R,(0,b],K}$ (resp. $\Lambda_{R,[a,b],K}$) equipped with a 
semilinear action of $\Gamma_{K}$ which commutes with 
$\varphi_{D}$ (resp. $\varphi_{D,[a,b/p]}$).

Unless explicitly stated otherwise, the $(\varphi,\Gamma_K)$-modules
considered below are finite projective, and their semilinear
$\Gamma_K$-actions are continuous. In the sheaf formulation, this
means that the underlying sheaf is locally free of finite rank and
that its restrictions to closed annuli are finite projective modules.

A $(\varphi,\Gamma_{K})$-module over $R$ is a module $D$ over 
$\Lambda_{R,\mathrm{rig},K}$ which arises via base change from a 
$(\varphi,\Gamma_{K})$-module over $\Lambda_{R,(0,b],K}$ for some $b>0$.

If $D$ is a $(\varphi,\Gamma)$-module over $\Lambda_{R,(0,b],K}$ and 
$I \subset (0,b]$ is a sub-interval, we write
$$
D_{I} := \Lambda_{R,I,K} \otimes_{\Lambda_{R,(0,b],K}} D
$$
to denote its restriction to the annulus $I$.
\end{definition}

\begin{lemma}
\label{uniformfinite}
Let $M$ be a coadmissible module over
$\Lambda_{R,(0,\infty]}(\Gamma_K)$, and suppose that there exists an integer
$n$ and an admissible covering
$\{[a_i,b_i]\}_{i\geq 0}$ of $(0,\infty]$ such that
$$
M\otimes_{\Lambda_{R,(0,\infty]}(\Gamma_K)}
\Lambda_{R,[a_i,b_i]}(\Gamma_K)
$$
can be generated by at most $n$ elements over
$\Lambda_{R,[a_i,b_i]}(\Gamma_K)$ for every $i$. Then $M$ is finitely
generated over $\Lambda_{R,(0,\infty]}(\Gamma_K)$.
\end{lemma}

\begin{proof}
We follow the argument of \cite[Proposition 2.1.13]{kedlaya2014cohomology}. After refining
and reordering the admissible covering, we may assume that the two
sub-collections
$
\{[a_{2i},b_{2i}]\}_{i\geq 0}
\qquad\text{and}\qquad
\{[a_{2i+1},b_{2i+1}]\}_{i\geq 0}
$
consist of pairwise disjoint intervals, ordered towards $0$. We will show
that there exist $n$ elements of $M$ whose images generate
$
M\otimes_{\Lambda_{R,(0,\infty]}(\Gamma_K)}
\Lambda_{R,[a_{2i},b_{2i}]}(\Gamma_K)
$
for every $i$. The same argument for the odd intervals then gives another
$n$ elements. These $2n$ elements generate $M$.

For every $i$, choose generators
$
f_{2i,1},\ldots,f_{2i,n}
$
of
$
M\otimes_{\Lambda_{R,(0,\infty]}(\Gamma_K)}
\Lambda_{R,[a_{2i},b_{2i}]}(\Gamma_K).
$
Fix a Banach norm $|\cdot|_{M,[a_{2i},b_{2i}]}$ on this module. Since it
is a finite Banach module, there exists $\epsilon_{2i}>0$ such that any
$n$ elements
$
f'_{2i,1},\ldots,f'_{2i,n}
$
satisfying
$
|f'_{2i,j}-f_{2i,j}|_{M,[a_{2i},b_{2i}]}
<\epsilon_{2i}$
for every $j$ still generate
$
M\otimes_{\Lambda_{R,(0,\infty]}(\Gamma_K)}
\Lambda_{R,[a_{2i},b_{2i}]}(\Gamma_K).
$
This follows from the Banach open mapping theorem.

On an annulus not containing $T=0$, the image of $M[1/T]$ is dense,
by the quasi-Stein approximation used in
\cite[Proposition 2.1.13]{kedlaya2014cohomology}.
Choose approximations to the $f_{2i,j}$ in $M[1/T]$, and clear powers of
$T$ to obtain $g_{2i,1},\ldots,g_{2i,n}\in M$ that still generate on
$[a_{2i},b_{2i}]$, since $T$ is invertible there.
On the initial disk use density of $M$ itself.
After decreasing $\epsilon_{2i}$, the same perturbation property holds
for these chosen generators.

Let $T=\gamma-1$, where $\gamma$ is a topological generator of an open
torsion-free subgroup of $\Gamma_K$. Put
$c_{2i,j}=\alpha_{2i,j}T^{\lambda_{2i,j}}$.
Following the same separated-annulus estimates as in
\cite[Proposition 2.1.13]{kedlaya2014cohomology}, choose
$\alpha_{2i,j}\in u^{\mathbb Z}$ and $\lambda_{2i,j}\geq0$ so that, for $k<i$,
$$\left|\frac{c_{2i,j}}{c_{2k,j}}g_{2i,j}\right|_{M,[a_{2k},b_{2k}]}
<\epsilon_{2k},\qquad
\left|\frac{c_{2k,j}}{c_{2i,j}}g_{2k,j}\right|_{M,[a_{2i},b_{2i}]}
<\epsilon_{2i}.$$
The exponents may also be chosen so that $c_{2i,j}g_{2i,j}$ tends to zero
on every fixed closed interval. The strict separation of the intervals
allows these finitely many inequalities at each step: increasing
$\lambda_{2i,j}$ separates the required lower and upper bounds for
$|\alpha_{2i,j}|$, and a power of the topologically nilpotent unit $u$
can then be chosen between them. On an initial interval containing
$\infty$, take $c_{0,j}=1$.

Consequently, for $1\leq j\leq n$, the series
$
g_j^{\mathrm{even}}
=
\sum_{i\geq 0}
\alpha_{2i,j}T^{\lambda_{2i,j}}g_{2i,j}
$
converges in $M$. By construction, on $[a_{2i},b_{2i}]$,
$$\left|c_{2i,j}^{-1}g_j^{\mathrm{even}}-g_{2i,j}\right|_{M,[a_{2i},b_{2i}]}
<\epsilon_{2i}.$$
Since $c_{2i,j}$ is a unit on this interval (and is $1$ on the initial disk),
the elements
$
g_1^{\mathrm{even}},\ldots,g_n^{\mathrm{even}}
$
still generate
$
M\otimes_{\Lambda_{R,(0,\infty]}(\Gamma_K)}
\Lambda_{R,[a_{2i},b_{2i}]}(\Gamma_K).
$
Thus the same $n$ global elements generate on every even interval.

Repeating the construction for the odd intervals gives elements
$
g_1^{\mathrm{odd}},\ldots,g_n^{\mathrm{odd}}\in M
$
which generate after base change to
$\Lambda_{R,[a_{2i+1},b_{2i+1}]}(\Gamma_K)$ for every $i$.

Hence the $2n$ elements
$
g_1^{\mathrm{even}},\ldots,g_n^{\mathrm{even}},
g_1^{\mathrm{odd}},\ldots,g_n^{\mathrm{odd}}
$
generate $M$ after base change to every member of the admissible covering.
Since $M$ is coadmissible, the cokernel of the resulting map
$$
\Lambda_{R,(0,\infty]}(\Gamma_K)^{2n}\longrightarrow M
$$
is coadmissible, and its base change to every member of the admissible
covering is zero. It is therefore zero. Hence, the above map is
surjective, and $M$ is finitely generated over
$\Lambda_{R,(0,\infty]}(\Gamma_K)$.
\end{proof}
\begin{lemma}
\label{uniformperfectcomplex}
Put $A=\Lambda_{R,(0,\infty]}(\Gamma_K)$, and choose a
sequence of Banach stage algebras $A_n = \Lambda_{R, [m_n, \infty]}(\Gamma_K)$ with $A=\varprojlim_nA_n$.
Let $C_n$ be compatible perfect complexes over $A_n$, meaning that
$$
C_{n+1}\otimes_{A_{n+1}}^{\mathbf L}A_n\simeq C_n.
$$
Suppose that their Tor-amplitudes are contained in $[0,e]$
and that, for each $0\leq j\leq e$, there is an integer $b_j$
independent of $n$ and of the maximal ideal $\mathfrak m\subset A_n$
such that
$$
\dim_{\kappa(\mathfrak m)}
H^j(C_n\otimes_{A_n}^{\mathbf L}\kappa(\mathfrak m))
\leq b_j.
$$
Then $C=\mathbf R\varprojlim_nC_n$ is quasi-isomorphic to
finite projective $A$-modules in degrees $[0,e]$.
The representative restricts to $C_n$ at every stage.
\end{lemma}
\begin{proof}
The flat transition maps imply that the modules $H^j(C_n)$
are coadmissible for every $j$. Their higher inverse limits vanish.
Hence, $H^j(C)=\varprojlim_nH^j(C_n)$ is coadmissible,
with $n$-th stage module $H^j(C_n)$, and $H^j(C)=0$ outside $[0,e]$.
The Krull dimensions of the Banach stages are uniformly bounded.

We first construct a bounded-above finite free representative.
Start with $D_e=C$ and $D_{e,n}=C_n$.
Suppose inductively that $D_q$ is the inverse limit of compatible
perfect complexes $D_{q,n}$ whose derived fibres vanish above $q$,
and whose derived-fibre dimensions in each fixed degree have a
bound independent of $n$ and of the maximal point.
Since $q$ is the top degree, we have
$$
H^q(D_{q,n})\otimes_{A_n}\kappa(\mathfrak m)
\simeq
H^q(D_{q,n}\otimes_{A_n}^{\mathbf L}\kappa(\mathfrak m)).
$$
Forster--Swan therefore bounds, uniformly in $n$, the number of
generators of $H^q(D_{q,n})$. Lemma~\ref{uniformfinite} implies
that the coadmissible module $H^q(D_q)$ is finitely generated over $A$.
Choose a finite free module $F^q$ mapping onto $H^q(D_q)$,
represent this map by $F^q[-q]\to D_q$, and put
$$
D_{q-1}=\mathrm{Cone}(F^q[-q]\to D_q).
$$
Use the corresponding cones to define $D_{q-1,n}$.
The map is surjective on top cohomology at every stage and after
every residue-field specialization. Thus the new derived fibres
vanish above $q-1$; in degree $q-1$ their dimensions are bounded by
the previous bound plus $\mathrm{rank}(F^q)$, and in lower degrees
they agree with the previous derived fibres.
The cohomology remains coadmissible because that category is
abelian and restriction to a stage is exact.
This proves the induction, including when $q<0$.

The resulting free-resolution construction gives a complex
$P^\bullet$ of finite free $A$-modules, with $P^j=0$ for $j>e$,
and quasi-isomorphisms
$$
P^\bullet\simeq C,\qquad
P^\bullet\otimes_AA_n\simeq C_n.
$$
Put $Q=\mathrm{coker}(P^{-1}\to P^0)$.
It is finitely presented and coadmissible.
Since there is no negative cohomology, truncation gives
$
C\simeq[Q\longrightarrow P^1\longrightarrow\cdots
\longrightarrow P^e].
$
Write $Q_n=Q\otimes_AA_n$. Exactness of restriction on
coadmissible modules gives the analogous representative of $C_n$.
The triangle with the projective subcomplex in degrees $[1,e]$
shows that, for every maximal ideal $\mathfrak m\subset A_n$,
$
\operatorname{Tor}^{A_n}_1(Q_n,\kappa(\mathfrak m))
\simeq
H^{-1}(C_n\otimes_{A_n}^{\mathbf L}\kappa(\mathfrak m))=0.
$
For $e=0$ the same identity follows directly.
The local criterion for flatness implies that $Q_n$ is finite
projective over $A_n$.

To prove projectivity over $A$, let $I\subset A$ be a finitely
generated ideal. It is coadmissible. A finite presentation of $Q$
shows that $I\otimes_AQ$ is coadmissible as well.
At every stage the map $I\otimes_AQ\to Q$ becomes
$
(IA_n)\otimes_{A_n}Q_n\longrightarrow Q_n,
$
which is injective because $Q_n$ is flat.
Its coadmissible kernel therefore has zero restriction to every
stage and is zero. The ideal criterion for flatness implies
that $Q$ is flat; finite presentation then implies that $Q$
is finite projective. This proves the assertion.
\end{proof}

\begin{remark}
\label{sheafmodule}
A priori, a $(\varphi, \Gamma_K)$ module over $\Lambda_{R,(0,b],K}$ is a sheaf over $\bigcup_{a \to 0} \textrm{Spa}(\Lambda_{R,[a,b],K})$ and not a module over $\Lambda_{R,(0,b],K}$. This gives a module $D_a$ over $\Lambda_{R,[a,b],K}$ for every $a$. Viewing $\bigcup_{a \to 0} \textrm{Spa}(\Lambda_{R,[a,b],K})$ as a quasi-Stein space (or $\Lambda_{R,(0,b],K}$ as a quasi-Stein algebra), we can go back and forth between these points of view by adapting the argument of Lemma 2.1.7 in \cite{bellovin2015p} to our pseudorigid setting, using the bounded-generation and projectivity results established above (for instance in Lemma \ref{uniformfinite}).
\end{remark}
 \textbf{Convention for completed tensor product:} Let $M, N$ be normed modules over a ring $S$, then $M {\otimes}_S N$ would have the tensor seminorm defined by $$|v|:= \inf_{r} \max |x_i| |y_i|$$ where $r$ ranges over all presentations on $v$ as a sum of elementary tensors $v = \sum_{i} x_i \otimes y_i$. The projective completed tensor product of normed $S$-modules is defined as the completion of the usual tensor product with respect to the tensor product norm. If the topologies on $M$ and $N$ are defined by a family of semi-norms, we can extend this notion in the obvious way. We write
\[
M \widehat{\otimes}_S N := M \widehat{\otimes}_{S,\pi} N
\]
for the projective completed tensor product. If one of $M$ or $N$ is a quasi-stein system given by a family of compatible norm, then we extend this notion in the obvious way.

\begin{definition}
Let $K/L$ be a finite extension, and let $D$ be a $(\varphi,\Gamma_K)$-module
over $\Lambda_{R,\mathrm{rig},K}$. We define the induced
$(\varphi,\Gamma_L)$-module by
$$\mathrm{Ind}_K^L(D)=
\operatorname{Hom}_{R[\Gamma_K]}(R[\Gamma_L],D),$$
with the usual induced $\Gamma_L$-action and coefficient restriction from
$\Lambda_{R,\mathrm{rig},L}$ to $\Lambda_{R,\mathrm{rig},K}$.
The Frobenius acts on the values of a homomorphism.
As an $R$-module this is a direct sum of $[\Gamma_L:\Gamma_K]$ copies of $D$.
The same construction is used at sufficiently small radii, with the
necessary radius change for the coefficient-ring inclusion.
\end{definition}
We now fix a $b$ and discuss some properties of the ring $\Lambda_{R, (0, b], K}$. Given $a'<a$, we note that the map $\Lambda_{R, [a', b], K} \to \Lambda_{R, [a, b], K}$ is flat and has topologically dense image from \cite{bellovin2024cohomology} and \cite{bellovin2023galois}. This makes $\Lambda_{R, (0, b], K}$ a quasi-Stein algebra in the sense of \cite{schneider2002algebras}, section 3. 
\begin{remark}
\label{spamax}
    Since $\text{Max}(\Lambda_{R,[a, b],K})$ is dense in $\mathrm{Spa}(\Lambda_{R,[a, b],K})$, from Prop 3.3(3), \cite{lourencco2017riemannian}, a sheaf over $\bigcup_{a \to 0} \mathrm{Spa}
(\Lambda_{R,[a, b],K})$ is also a sheaf over $\bigcup_{a \to 0} \text{Max}(\Lambda_{R,[a, b],K})$. Hence, we can view a projective $(\varphi, \Gamma_K)$ module over $R$ as a sheaf on a quasi-stein space and from Lemma 2.1.7, \cite{bellovin2015p}, we can view a finite projective $(\varphi, \Gamma_K)$ module (A priori defined as a sheaf) as an actual finite module over $\Lambda_{R, (0, b], K}$. 
\end{remark}   

To define Iwasawa cohomology (which we will do in the next section), we need to define the operator $\psi$, a one sided inverse to $\varphi$. The isomorphism $\varphi^{*}D_{(0,b]} \xrightarrow{\sim} D_{(0,b/p]}$ 
induces an isomorphism
$$
\Lambda_{R,(0,b/p],K} \otimes_{\varphi(\Lambda_{R,(0,b],K})} 
\varphi(D_{(0,b]})
\xrightarrow{\sim}
D_{(0,b/p]}.
$$

We therefore have a surjective homomorphism 
$\psi_{D} : D_{(0,b/p]} \to D_{(0,b]}$ defined by setting 
$\psi_{D}(a \otimes \varphi(d)) = \psi(a)\, d$, where 
$a \in \Lambda_{R,(0,b/p],K}$ and $d \in D_{(0,b]}$.

For a $(\varphi, \Gamma_K)$ module $M$ over the Robba ring over a pseudoaffinoid $R$, we define its Tate dual
$$
M^* =
\operatorname{Hom}_{\Lambda_{R,\mathrm{rig},K}}
(M,\Lambda_{R,\mathrm{rig},K})(\chi_{cyc}).
$$
When $M^*$ is viewed as a sheaf,
$$
M^*_{(0,b]}=
\operatorname{Hom}_{\Lambda_{R,(0,b],K}}
(M_{(0,b]},\Lambda_{R,(0,b],K})(\chi_{cyc}).
$$
The $(\varphi,\Gamma_K)$-structure on the dual is the usual dual
$(\varphi,\Gamma_K)$-structure. In particular, before applying the Tate
twist, for
$f\in
\operatorname{Hom}_{\Lambda_{R,\mathrm{rig},K}}
(M,\Lambda_{R,\mathrm{rig},K})$
we have
$
(\gamma f)(m)
=
\gamma\left(f(\gamma^{-1}m)\right),
$
and $\varphi(f)$ is uniquely characterized by
$
(\varphi f)(\varphi m)=\varphi(f(m)).
$
The operators on $M^*$ are obtained after applying the Tate twist.

\begin{definition}
\label{defpairM}

Write $e_1$ for the standard basis of
$\Lambda_{R,\mathrm{rig},\mathbb Q_p}(1)$, so that
$\varphi(e_1)=e_1$ and
$\gamma(e_1)=\chi_{\mathrm{cyc}}(\gamma)e_1$.
Define
$
\operatorname{Res}_1(fe_1)
=\operatorname{Res}\left(f(T)\frac{dT}{1+T}\right).
$
Here $\operatorname{Res}$ is the residue on differentials
defined in Section~3.3.
For a $(\varphi,\Gamma_{\mathbb Q_p})$-module $M$, let
$\langle\ ,\ \rangle:M\times M^*\to
\Lambda_{R,\mathrm{rig},\mathbb Q_p}(1)$
be the tautological pairing, and put
$\{x,y\}=\operatorname{Res}_1\langle x,y\rangle.
$
\end{definition}

\begin{lemma}
\label{gammaphipsiprop}
Let $M$ be a finite projective $(\varphi,\Gamma_{\mathbb Q_p})$-module,
and use the pairing of Definition~\ref{defpairM}.
For $x\in M$, $y\in M^*$, and $\gamma\in\Gamma_{\mathbb Q_p}$,
$$
\{\gamma(x),\gamma(y)\}
=\{\varphi(x),\varphi(y)\}
=\{x,y\},
$$
and
$$
\{\varphi(x),y\}=\{x,\psi(y)\},
\qquad
\{\psi(x),y\}=\{x,\varphi(y)\}.
$$
\end{lemma}

\begin{proof}
The adjunction identities follow from the fact that
$\operatorname{Res}_1$ extracts the zeroth Frobenius summand.
We first establish the scalar identity
$$
\operatorname{Res}_1\bigl((1+T)^i\varphi(f)e_1\bigr)
=
\begin{cases}
\operatorname{Res}_1(fe_1),&i=0,\\
0,&1\leq i\leq p-1.
\end{cases}
$$
By continuity, it suffices to check this for Laurent monomials
$f=T^n$. For $n\geq0$, both sides vanish.
For $n=-m<0$, write
$$
h(T):=\varphi(T^{-m})
=T^{-pm}\left(1+
\sum_{j=1}^{p-1}\binom pj T^{j-p}\right)^{-m},
$$
using the convergent expansion on sufficiently small Robba
annuli. This expansion contains only powers at most $-pm$.
Since each coefficient $\binom pj$ is divisible by $p$,
the expansion also converges at $T=-1$, where
$h(-1)=(-1)^{-m}$.
For a negative Laurent series $g$ whose coefficients tend
to zero, coefficient extraction gives
$$
\operatorname{Res}\left(g(T)\frac{dT}{1+T}\right)=-g(-1).
$$
Applying this to $h$ and $T^{-m}$ proves the identity for
$i=0$. For $1\leq i\leq p-1$, the series
$(1+T)^{i-1}h(T)$ has no $T^{-1}$ term, since its powers
are at most $-pm+p-2\leq-2$.
This proves the scalar identity, including in
characteristic $p$, without dividing by $p$.

In particular,
$\operatorname{Res}_1(\varphi z)=\operatorname{Res}_1(z)$.
For $\gamma\in\Gamma_{\mathbb Q_p}$, put
$a=\chi_{\mathrm{cyc}}(\gamma)$ and $U=(1+T)^a-1$.
Since $\gamma(fe_1)=a f(U)e_1$ and
$dU/(1+U)=a\,dT/(1+T)$, invariance of residue under this
invertible change of variable gives
$
\operatorname{Res}_1\bigl(\gamma(fe_1)\bigr)
=\operatorname{Res}\left(f(U)\frac{dU}{1+U}\right)
=\operatorname{Res}_1(fe_1).
$
Equivariance of the tautological pairing now gives
$
\{\gamma(x),\gamma(y)\}
=\{\varphi(x),\varphi(y)\}
=\{x,y\}.
$

Finally, write
$$
x=\sum_{i=0}^{p-1}(1+T)^i\varphi(x_i),
\qquad
y=\sum_{i=0}^{p-1}(1+T)^i\varphi(y_i).
$$
By the definition of $\psi$, we have
$\psi(x)=x_0$ and $\psi(y)=y_0$.
Bilinearity, Frobenius equivariance and the scalar identity
therefore give
$$
\begin{aligned}
\{\varphi(x),y\}
&=\sum_{i=0}^{p-1}
\operatorname{Res}_1\bigl((1+T)^i
\varphi\langle x,y_i\rangle\bigr)
=\{x,y_0\}=\{x,\psi(y)\},\\
\{x,\varphi(y)\}
&=\sum_{i=0}^{p-1}
\operatorname{Res}_1\bigl((1+T)^i
\varphi\langle x_i,y\rangle\bigr)
=\{x_0,y\}=\{\psi(x),y\}.
\end{aligned}
$$
\end{proof}
\subsection{Analyzing the analytic action of $\gamma-1$ on $(\varphi, \Gamma)$-modules}

In this section, we analyze how the action of power series in $\gamma-1$ looks like on a given $(\varphi, \Gamma)$-module.

The following Lemma is essentially Proposition 2.2.14 of \cite{kedlaya2014cohomology}, but we state the proof here again, highlighting those aspects of the proof more clearly that will be crucial for us later. This will be used extensively to prove the main theorem.
\begin{lemma}
    \label{2214}
    Consider the ring $\Lambda_{R, (0, b], K}$, and for each
    $0<a\leq b$, let $|\cdot|_{[a, b]}$ be the chosen Banach
    norm on $\Lambda_{R, [a, b], K}$, and
    $\|\cdot\|_{[a, b]}$ the associated operator norm.
    \begin{enumerate}
        \item The action of $\Gamma_K$ on
        $\Lambda_{R, (0, b], K}$ is continuous. Moreover, for
        every $0<a\leq b$,
        $$
        \lim_{\gamma\to 1}\|\gamma-1\|_{[a, b]}=0.
        $$
        \item The action of $R[\Gamma_K]$ on
        $\Lambda_{R, (0, b], K}$ extends continuously to an
        action of $\Lambda_{R, (0, \infty]}(\Gamma_K)$.
        More precisely, fix $n\geq 1$ ($n\geq 2$ if $p=2$)
        such that $\Gamma_n\subseteq\Gamma_K$. For every
        $0<a\leq b$, there exists $k_a=p^{n_a}$ such that
        $$
        \|(\gamma-1)^{k_a}\|_{[a, b]}\leq\omega
        $$
        for every $\gamma\in\Gamma_n$.
        \item We can choose the $k_a$ in (b) such that
        $k_a\to\infty$ as $a\to 0$.
    \end{enumerate}
\end{lemma}
\begin{proof}
For (a), the continuity of the action follows from the
construction of the rings. We prove the stronger statement
about operator norms.

The module
$\operatorname{Ind}_K^{\mathbb{Q}_p}
\Lambda_{R, (0, b], K}$
is non-canonically a direct sum of finitely many
$\Gamma_K$-stable copies of $\Lambda_{R, (0, b], K}$.
By Remark \ref{compdifflamb}, it is finite over
$\Lambda_{R, (0,  b], \mathbb{Q}_p}$.
To reduce to the scalar case, choose finitely many generators
$e_1,\ldots,e_d$. On each closed annulus, finite generation
allows us to write $f=\sum_i f_i e_i$ with
$\max_i|f_i|\leq C|f|$, where $C$ is independent of $f$.
Now
$$
\gamma(f_i e_i)-f_i e_i
=
(\gamma(f_i)-f_i)e_i
+
\gamma(f_i)(\gamma(e_i)-e_i).
$$
Since $\gamma(e_i)\to e_i$ for each $i$, a uniform
operator-norm estimate on the scalar ring gives the required
estimate on the induced module, and hence on the original
ring. We therefore work with $\mathbb{Q}_p$.

Write $\pi=\pi_{\mathbb{Q}_p}$. We may use the usual Gauss
norms, since convergence to zero in operator norm is
unchanged upon replacing a Banach norm by an equivalent one.
We must show that, for every $\epsilon>0$, there exists $k$
such that
$$
|\gamma(f)-f|_{[a, b]}
\leq\epsilon|f|_{[a, b]}
$$
for every $\gamma\in\Gamma_k$ and every
$f\in\Lambda_{R, [a, b], \mathbb{Q}_p}$.

Suppose $\gamma\in\Gamma_k$, so that
$\chi(\gamma)=1+tp^k$ for some $t\in\mathbb{Z}_p$. Then
\begin{equation}
    \label{pii}
    \gamma(\pi)-\pi
    =
    (1+\pi)((1+\pi)^{tp^k}-1)
\end{equation}
and
\begin{equation}
    \label{1bypi}
    \frac{1}{\gamma(\pi)}-\frac{1}{\pi}
    =
    -\frac{1+\pi}{\pi}
    \cdot
    \frac{(1+\pi)^{tp^k}-1}
    {(1+\pi)((1+\pi)^{tp^k}-1)+\pi}.
\end{equation}
Since $|\pi|_{[a, b]}<1$, taking higher powers of $p$ gives
$$
\sup_{t\in\mathbb{Z}_p}
|(1+\pi)^{tp^k}-1|_{[a, b]}
\longrightarrow 0.
$$
Thus, for $0<\epsilon<1$, we can choose $k$ such that
$$
\left|\frac{\gamma(\pi)}{\pi}-1\right|_r<\epsilon
$$
for every $\gamma\in\Gamma_k$ and both $r=a,b$.
In particular, $|\gamma(\pi)|_r=|\pi|_r$. Using
\eqref{pii} for positive powers and \eqref{1bypi} for
negative powers, we obtain
$
|\gamma(\pi^i)-\pi^i|_r
\leq\epsilon|\pi^i|_r
$
for every $i\in\mathbb{Z}$ and both $r=a,b$.
For a general convergent Laurent series
$f=\sum_{i\in\mathbb{Z}}a_i\pi^i$, it follows that
$
|\gamma(f)-f|_r
\leq
\epsilon\sup_{i\in\mathbb{Z}}|a_i\pi^i|_r
=
\epsilon|f|_r.
$
Taking the maximum over $r=a,b$ proves (a).

For (b), fix $n$ such that $\Gamma_n\subseteq\Gamma_K$.
By (a) and compactness of $\Gamma_n$,
$
C_a=\sup_{\gamma\in\Gamma_n}\|\gamma\|_{[a, b]}<\infty.
$
We first use the equivalent $\Gamma_n$-invariant norm
$$
|f|'_{[a, b]}
=
\sup_{\gamma\in\Gamma_n}|\gamma(f)|_{[a, b]},
$$
for which
$$
|f|_{[a, b]}
\leq |f|'_{[a, b]}
\leq C_a|f|_{[a, b]}.
$$
Let $\|\cdot\|'_{[a, b]}$ be its associated operator norm.
Then
$
\|\gamma-1\|'_{[a, b]}\leq 1
$
for every $\gamma\in\Gamma_n$.

By (a), choose $q_a\geq n$ such that
$
\|\eta-1\|'_{[a, b]}\leq\omega
$
for every $\eta\in\Gamma_{q_a}$.
For $\gamma\in\Gamma_n$, we have
$\gamma^{p^{q_a-n}}\in\Gamma_{q_a}$, and the binomial
identity gives
$
(\gamma-1)^{p^{q_a-n}}
=
(\gamma^{p^{q_a-n}}-1)
+
pP_a(\gamma-1),
$
where $P_a(X)\in\mathbb{Z}[X]$. Consequently,
$
\|P_a(\gamma-1)\|'_{[a, b]}\leq 1.
$
The two terms on the right commute. Raising this identity
to the power $p^h$ therefore gives
$$
\|(\gamma-1)^{p^{q_a-n+h}}\|'_{[a, b]}
\leq
\max_{0\leq j\leq p^h}
\left\{
\omega^{p^h-j}\|p^j\|'_{[a, b]}
\right\},
$$
where $\|p^j\|'_{[a, b]}$ denotes the operator norm of
multiplication by $p^j$. Since $p$ is topologically
nilpotent, this maximum tends to zero as $h\to\infty$.
The bound is independent of $\gamma\in\Gamma_n$.
Since
$$
\|F\|_{[a, b]}\leq C_a\|F\|'_{[a, b]}
$$
for every bounded operator $F$, we can choose $h$ large
enough that, with $n_a=q_a-n+h$ and $k_a=p^{n_a}$,
$
\|(\gamma-1)^{k_a}\|_{[a, b]}\leq\omega
$
for every $\gamma\in\Gamma_n$.

It remains to explain the extension of the action.
Choose a topological generator $\gamma_n$ of $\Gamma_n$.
The binomial formula and the definition of $C_a$ give
$
\|(\gamma_n-1)^r\|_{[a, b]}\leq C_a
$
for every nonnegative integer $r$. Writing $j=qk_a+r$,
where $0\leq r<k_a$, we obtain
$
\|(\gamma_n-1)^j\|_{[a, b]}
\leq C_a\omega^q.
$
If
$
\sum_{j\geq 0}b_j(\gamma_n-1)^j
\in\Lambda_{R, (0, \infty]}(\Gamma_n),
$
then
$
|b_j|\omega^{j/k_a}\longrightarrow 0.
$
It follows that the corresponding series converges in
operator norm on $\Lambda_{R, [a, b], K}$.
These actions are compatible with restriction as $a$
varies. Taking the inverse limit over $a$ and using
$
\Lambda_{R, (0, \infty]}(\Gamma_K)
=
\Lambda_{R, (0, \infty]}(\Gamma_n)
\otimes_{R[\Gamma_n]}R[\Gamma_K]
$
proves (b).

For (c), observe that replacing $n_a$ by any larger integer
preserves the estimate in (b), since
$$
\|(\gamma-1)^{p^{n_a+h}}\|_{[a, b]}
\leq
\|(\gamma-1)^{p^{n_a}}\|_{[a, b]}^{p^h}
\leq\omega^{p^h}\leq\omega.
$$
We may therefore enlarge $n_a$ so that
$
n_a\geq\left\lceil\log_p(b/a)\right\rceil.
$
Then $k_a=p^{n_a}\geq b/a$, and hence $k_a\to\infty$
as $a\to 0$.
\end{proof}
Given $x \in D_{(0, b]}^{\psi = 1}$, we can verify that $(\varphi-1)x \in D_{(0, b]}^{\psi = 0}$. We can also check that $D_{(0, b]}^{\varphi = 1} \subset D_{(0, b]}^{\psi = 1}$.

\begin{proposition}
\label{mapnat}
    The action of $R[\Gamma_K]$ on $\Lambda_{R, [a, b], K}$ extends to an action of $\Lambda_{R, [m_a, \infty]}(\Gamma_K)$ on $\Lambda_{R, [a, b], K}$ for some $m_a$. 
\end{proposition}

\begin{proof}
Choose the smallest integer $n$ such that $\Gamma_n\subset\Gamma_K$. Then
$\Gamma_K=\Gamma_n\times\Gamma_{K,\mathrm{tors}}$, where
$\Gamma_{K,\mathrm{tors}}$ is finite. We construct the action of $\Lambda_{R,[1/k_a,\infty]}(\Gamma_n)$ below and then extend it using the finite-index convention.

By Lemma~\ref{2214}, there exists $k_a\in\mathbb N$ such that
$\|(\gamma_n-1)^{k_a}\|_{\Lambda_{R,[a,b],K}}\le\omega$. We claim that one may take $m_a=1/([\Gamma_K:\Gamma_n]k_a)$.

Let
$f(\gamma_n-1)=\sum_{j\ge0}b_j(\gamma_n-1)^j
\in\Lambda_{R,[1/k_a,\infty]}(\Gamma_n)$.
Since
$\lim_{j\to\infty}|b_j|\omega^{j/k_a}=0$, writing
$j=mk_a+r$ with $0\le r<k_a$, we obtain
$\lim_{m\to\infty}|b_{mk_a+r}|\omega^{m+r/k_a}=0$ for every
$r\in\{0,\ldots,k_a-1\}$. Hence
$$
f(\gamma_n-1)
=
\sum_{r=0}^{k_a-1}
\sum_{m\ge0}
b_{mk_a+r}(\gamma_n-1)^{mk_a+r}.
$$

Now, let
$g(T)=\sum_{i\in\mathbb Z}c_iT^i
\in\Lambda_{R,[a,b],K}$.
By Lemma \ref{2214}, for $j=mk_a+r$, where
$0\le r<k_a$, we have
$$
|b_{mk_a+r}(\gamma_n-1)^{mk_a+r}\cdot g(T)|_{[a,b]}
\le
C_a|b_{mk_a+r}|\omega^m|g(T)|_{[a,b]}.
$$
Since
$f(T)=\sum_{j\ge0}b_jT^j
\in\Lambda_{R,[1/k_a,\infty]}(\Gamma_n)$,
we have
$\lim_{j\to\infty}|b_j|\omega^{j/k_a}=0$.
As $m\le j/k_a<m+1$, the right-hand side
tends to $0$ as $m\to\infty$. Hence
$f(\gamma_n-1)\cdot g(T)$ converges in
$\Lambda_{R,[a,b],K}$, giving an action of
$\Lambda_{R,[1/k_a,\infty]}(\Gamma_n)$ on
$\Lambda_{R,[a,b],K}$.

Finally,
$$
\Lambda_{R,[1/[\Gamma_K:\Gamma_n]k_a,\infty]}(\Gamma_K)
=
\Lambda_{R,[1/k_a,\infty]}(\Gamma_n)
\otimes_{R[\Gamma_n]}
R[\Gamma_K],
$$
since $\Gamma_{K,\mathrm{tors}}$ is finite. Extending scalars therefore
gives the desired action of
$\Lambda_{R,[1/[\Gamma_K:\Gamma_n]k_a,\infty]}(\Gamma_K)$ on
$\Lambda_{R,[a,b],K}$.
\end{proof}

The following lemma is exactly the same as Theorem 3.1.1 and Lemma 3.1.2 from \cite{kedlaya2014cohomology}, we however include the proof for completeness.

\begin{lemma}\label{lem:closed-psi-zero-projective}
Put $\Gamma=\Gamma_{\mathbb{Q}_p}$ and
$\mu(t)=\nu e_K t$, where
$\nu=p/(p-1)$. For small enough $b$ and every $0<a\leq b$,
the module $D_{[a,b]}^{\psi=0}$ is finite projective over
$\Lambda_{R,[\mu(a),\mu(b)]}(\Gamma_K)$.
\end{lemma}

\begin{proof}
We follow the finite-radius argument of
\cite[Lemma~3.1.2 and Corollary~3.1.3]{kedlaya2014cohomology},
keeping track of the subgroup indices. We first deal with
$K=\mathbb{Q}_p$, in which case $\mu(t)=\nu t$.
Choose $r_0>0$ sufficiently small that $D$ is defined and the
Frobenius decompositions below hold. Put
$I_\ell=[r_0/p^{\ell+1},r_0/p^\ell]$ for $\ell\geq 0$.
For $n\geq 2$, choose $\gamma_n\in\Gamma_n$ such that
$\chi(\gamma_n)=1+p^n$. In particular, $\gamma_n$ is a
topological generator of $\Gamma_n$.

Choose a finite projective complement $N_{I_0}$ such that
$D_{I_0}\oplus N_{I_0}$ is finite free over
$\Lambda_{R,I_0}$, and fix a basis $e_1,\ldots,e_d$.
For a closed subinterval $J\subseteq I_0$, let $N_J$ be
the restriction of this complement. We use the Gauss norms
coming from the chosen basis.

On $D_J$, define $G_{\gamma_n}=(1+\pi)\gamma_n-1$. Then
\begin{equation}\label{eq:Ggamma-factor}
G_{\gamma_n}
=
\pi\left(
1+\frac{1+\pi}{\pi}(\gamma_n-1)
\right).
\end{equation}
Extend $G_{\gamma_n}$ to $N_J$ by letting it act as
multiplication by $\pi$.

We first choose $n_0$ so that these operators are uniformly
close to multiplication by $\pi$ on the fixed annulus $I_0$.
The scalar estimates in Lemma \ref{2214} are uniform for
the Gauss norms $|\cdot|_r$ with $r\in I_0$. For a finite
module, the same conclusion follows by writing
$
\gamma_n(fx)-fx
=
(\gamma_n(f)-f)x
+
\gamma_n(f)(\gamma_n(x)-x)
$
for finitely many generators $x$. Indeed, the coordinates
of $\gamma_n(x)-x$ tend to zero in the norm on $I_0$, which
bounds their norms at every $r\in I_0$.
Applying this to the projections of the chosen basis
vectors onto $D_{I_0}$ shows that
$G_{\gamma_n}-\pi$ tends to zero uniformly in these norms.

Since $|\pi|_r$ is bounded away from zero on $I_0$, we may
therefore choose $n_0$ and $0<c<1$ such that, for every
$n\geq n_0$ and every $r\in I_0$,
$
\|\pi^{-1}(G_{\gamma_n}-\pi)\|_r\leq c.
$
Consequently, $G_{\gamma_n}$ is invertible. On $D_J$,
its inverse is given by the convergent series
\[
G_{\gamma_n}^{-1}
=
\sum_{k\geq 0}
\left(
-\frac{1+\pi}{\pi}(\gamma_n-1)
\right)^k\pi^{-1}.
\]
On $N_J$, its inverse is multiplication by $\pi^{-1}$.
We obtain
\begin{equation}\label{eq:Ggamma-close-pi}
\|G_{\gamma_n}^{\pm1}-\pi^{\pm1}\|_r
\leq c|\pi^{\pm1}|_r
\end{equation}
for every $r\in I_0$.

Our choice of $\gamma_n$ gives
\[
\gamma_n((1+\pi)\varphi^n x)-(1+\pi)\varphi^n x
=
(1+\pi)\varphi^n(G_{\gamma_n}(x)).
\]
Thus the action of $\gamma_n-1$ on
$(1+\pi)\varphi^nD_J$ corresponds to the action of
$G_{\gamma_n}$ on $D_J$.

By telescoping positive and negative powers,
\eqref{eq:Ggamma-close-pi} gives
$
\|G_{\gamma_n}^j-\pi^j\|_r
\leq c|\pi^j|_r
$
for every $j\in\mathbb{Z}$ and every $r\in I_0$.
In particular, if
$f(\pi)=\sum_{j\in\mathbb{Z}}a_j\pi^j$ converges on $J$,
then $f(G_{\gamma_n})=\sum_j a_jG_{\gamma_n}^j$
converges in operator norm on $D_J\oplus N_J$.

Consider the two maps from
$\Lambda_{R,J}(\Gamma_n)^d$ to $D_J\oplus N_J$ given by
\[
\Phi\left(\sum_i f_i(\gamma_n-1)e_i\right)
=
\sum_i f_i(\pi)e_i,
\qquad
\Phi_n\left(\sum_i f_i(\gamma_n-1)e_i\right)
=
\sum_i f_i(G_{\gamma_n})e_i.
\]
The map $\Phi$ is a topological isomorphism, and the
preceding estimates give
$\|\Phi_n\Phi^{-1}-1\|_J\leq c<1$.
Hence $\Phi_n\Phi^{-1}$ is invertible by a Neumann series.
Since $\Phi_n$ is linear for the action defined by
$G_{\gamma_n}$, this makes $D_J\oplus N_J$ finite free over
$\Lambda_{R,J}(\Gamma_n)$.

The projection onto $D_J$ commutes with $G_{\gamma_n}$.
It follows that $(1+\pi)\varphi^nD_J$ is finite projective
over $\Lambda_{R,J}(\Gamma_n)$.
Moreover, the maps $\Phi_n$ commute with restriction of
$J$. Thus, for closed subintervals $J'\subseteq J\subseteq I_0$,
we have a natural isomorphism
\[
(1+\pi)\varphi^nD_J
\otimes_{\Lambda_{R,J}(\Gamma_n)}
\Lambda_{R,J'}(\Gamma_n)
\cong
(1+\pi)\varphi^nD_{J'}.
\]

For a closed interval $I$, put
$Q_I=(1+\pi)\varphi^{n_0}D_I$.
We now pass from $I_0$ to $I_\ell$.
For every closed subinterval $J\subseteq I_0$, the
Frobenius decomposition gives
\[
Q_{J/p^\ell}
\cong
(1+\pi)\varphi^{n_0+\ell}D_J
\otimes_{R[\Gamma_{n_0+\ell}]}R[\Gamma_{n_0}].
\]
Since $[\Gamma_{n_0}:\Gamma_{n_0+\ell}]=p^\ell$, our
subgroup convention gives
\[
\Lambda_{R,J/p^\ell}(\Gamma_{n_0})
\cong
\Lambda_{R,J}(\Gamma_{n_0+\ell})
\otimes_{R[\Gamma_{n_0+\ell}]}R[\Gamma_{n_0}].
\]
Consequently,
\[
Q_{J/p^\ell}
\cong
(1+\pi)\varphi^{n_0+\ell}D_J
\otimes_{\Lambda_{R,J}(\Gamma_{n_0+\ell})}
\Lambda_{R,J/p^\ell}(\Gamma_{n_0}).
\]
Applying the preceding argument with $n=n_0+\ell$, we see
that $Q_{J/p^\ell}$ is finite projective over
$\Lambda_{R,J/p^\ell}(\Gamma_{n_0})$.
These isomorphisms commute with restriction of $J$.
In particular, this proves finite projectivity and
compatibility with restriction for $Q_{I_\ell}$ and all
its closed subannuli.

We next use the decomposition
\begin{equation}\label{eq:psi-zero-decomposition}
D_{I_\ell/p^{n_0}}^{\psi=0}
=
\bigoplus_{i\in(\mathbb{Z}/p^{n_0}\mathbb{Z})^\times}
(1+\pi)^{\tilde i}\varphi^{n_0}D_{I_\ell}.
\end{equation}
Equivalently,
\begin{equation}\label{eq:psi-zero-induced}
D_{I_\ell/p^{n_0}}^{\psi=0}
\cong
Q_{I_\ell}\otimes_{R[\Gamma_{n_0}]}R[\Gamma].
\end{equation}
The same identities hold with $I_\ell$ replaced by any
closed subinterval.

Here $[\Gamma:\Gamma_{n_0}]=(p-1)p^{n_0-1}$, so
$
\frac{I_\ell}{[\Gamma:\Gamma_{n_0}]}
=
\frac{\nu I_\ell}{p^{n_0}}.
$
Thus the subgroup change formula is
\begin{equation}\label{eq:distribution-change-subgroup}
\Lambda_{R,\nu I_\ell/p^{n_0}}(\Gamma)
\cong
\Lambda_{R,I_\ell}(\Gamma_{n_0})
\otimes_{R[\Gamma_{n_0}]}R[\Gamma].
\end{equation}
Combining this with \eqref{eq:psi-zero-induced}, we obtain
\begin{equation}\label{eq:psi-zero-base-change}
D_{I_\ell/p^{n_0}}^{\psi=0}
\cong
Q_{I_\ell}
\otimes_{\Lambda_{R,I_\ell}(\Gamma_{n_0})}
\Lambda_{R,\nu I_\ell/p^{n_0}}(\Gamma).
\end{equation}
It follows that $D_{I_\ell/p^{n_0}}^{\psi=0}$ is finite
projective over $\Lambda_{R,\nu I_\ell/p^{n_0}}(\Gamma)$,
compatibly with restriction to every closed subannulus.

It remains to glue these modules. For $\ell\geq 1$, put
$c_\ell=r_0/p^{n_0+\ell}$.
The physical annuli $I_{\ell-1}/p^{n_0}$ and
$I_\ell/p^{n_0}$ meet along $[c_\ell,c_\ell]$.
Their distribution annuli meet along
$[\nu c_\ell,\nu c_\ell]$.
The restriction isomorphisms just proved give
\begin{equation}\label{eq:left-boundary-restriction}
D_{I_{\ell-1}/p^{n_0}}^{\psi=0}
\otimes_{\Lambda_{R,\nu I_{\ell-1}/p^{n_0}}(\Gamma)}
\Lambda_{R,[\nu c_\ell,\nu c_\ell]}(\Gamma)
\cong
D_{[c_\ell,c_\ell]}^{\psi=0}
\end{equation}
and
\begin{equation}\label{eq:right-boundary-restriction}
D_{I_\ell/p^{n_0}}^{\psi=0}
\otimes_{\Lambda_{R,\nu I_\ell/p^{n_0}}(\Gamma)}
\Lambda_{R,[\nu c_\ell,\nu c_\ell]}(\Gamma)
\cong
D_{[c_\ell,c_\ell]}^{\psi=0}.
\end{equation}
These identify the restrictions to the common boundary.
Since the identifications come from restriction maps,
the cocycle condition is automatic.

Take $b\leq r_0/p^{n_0}$ and $0<a\leq b$.
Intersecting $[a,b]$ with the annuli $I_\ell/p^{n_0}$
gives a finite cover by closed subannuli. The associated
finite projective modules therefore glue on the distribution
annulus $[\nu a,\nu b]$.
Their sections identify with $D_{[a,b]}^{\psi=0}$:
sections of $D$ glue on the physical annuli, and restriction
commutes with $\psi$.
By affinoid descent for finite projective modules,
$D_{[a,b]}^{\psi=0}$ is finite projective over
$\Lambda_{R,[\nu a,\nu b]}(\Gamma)$.

Finally, let $K$ be arbitrary. Put
$M=\operatorname{Ind}_K^{\mathbb{Q}_p}D$.
By Remark \ref{compdifflamb}, restriction of $M$ to
$[a,b]$ corresponds to induction of
$D_{[a,b]}$. Induction commutes with $\psi$, so
$
M_{[a,b]}^{\psi=0}
\cong
D_{[a,b]}^{\psi=0}
\otimes_{R[\Gamma_K]}R[\Gamma].
$
For small enough $b$, the case already proved shows that
this module is finite projective over
$\Lambda_{R,[\nu a,\nu b]}(\Gamma)$.

Our original subgroup convention gives
$
\Lambda_{R,[\nu a,\nu b]}(\Gamma)
\cong
\Lambda_{R,[\mu(a),\mu(b)]}(\Gamma_K)
\otimes_{R[\Gamma_K]}R[\Gamma].
$
This ring is finite free over
$\Lambda_{R,[\mu(a),\mu(b)]}(\Gamma_K)$.
Hence
$M_{[a, b]}^{\psi=0}$
is finite projective over the latter ring.

As a $\Gamma_K$-module, the induced module is a finite
direct sum of copies of $D_{[a,b]}$. The projection onto
each copy is continuous and commutes with $\Gamma_K$ and
$\psi$. It also commutes with the inverse of
$\gamma_n-1$, for sufficiently large $n$ with
$\Gamma_n\subseteq\Gamma_K$, since this operator is
invertible on the induced $\psi=0$ module.
By continuity, the projection therefore commutes with
the action of
$\Lambda_{R,[\mu(a),\mu(b)]}(\Gamma_K)$.
Thus $D_{[a,b]}^{\psi=0}$ is a direct summand of a finite
projective module over this ring, and is itself finite
projective.
\end{proof}
\begin{lemma}
\label{moduleaction}
Let $D$ be a finite projective
$(\varphi,\Gamma_K)$-module over
$\Lambda_{R,\mathrm{rig},K}$.
For sufficiently small $b>0$ and every
$0<a\leq b$, the action of $R[\Gamma_K]$
on $D_{[a,b]}$ extends continuously
to an action of
$$
\Lambda_{R,[\mu(a),\infty]}(\Gamma_K),
\qquad
\mu(a)=e_K\frac{p}{p-1}a.
$$
These actions are compatible with
restriction to closed subannuli.
\end{lemma}

\begin{proof}
We first assume that $K=\mathbb Q_p$.
Put $\nu=p/(p-1)$.
Choose $r_0>0$ sufficiently small
that $D$ is defined on the annulus
$
I_0=[r_0/p,r_0]
$
and that the Frobenius decompositions
and the finite-radius results for
the $\psi=0$ part hold on the
annuli considered below.

We first observe that the action
of $\Gamma$ on $D_{I_0}$ extends
to an action of
$\Lambda_{R,[s_0,\infty]}(\Gamma)$
for some $s_0>0$. Indeed, choose finitely many generators
$x_1,\ldots,x_d$ of $D_{I_0}$.
For $f\in\Lambda_{R,I_0}$, we have
$
(\gamma-1)(fx_i)
=
(\gamma(f)-f)x_i
+
\gamma(f)(\gamma(x_i)-x_i).
$
By Lemma~\ref{2214}, the action of
$\Gamma$ on the scalar ring tends to
the identity in operator norm as
$\gamma\to1$.
Since $\gamma(x_i)\to x_i$ for every $i$,
finite generation shows that
$$
\|\gamma-1\|_{D_{I_0}}\longrightarrow0
\qquad(\gamma\to1).
$$
Choosing a sufficiently small open
subgroup $\Gamma_n$ and applying the
same power-series argument as in
Proposition~\ref{mapnat}, we obtain
the asserted action for some $s_0>0$.

For $\ell\geq0$, put
$
I_\ell=p^{-\ell}I_0
=
[r_0/p^{\ell+1},r_0/p^\ell].
$
Using $\psi\varphi=1$, we have a
$\Gamma$-equivariant decomposition
$
D_{I_\ell}
=
\varphi^\ell(D_{I_0})
\oplus
\bigoplus_{j=0}^{\ell-1}
\varphi^j
\left(
D_{p^jI_\ell}^{\psi=0}
\right).
$
We show that, for sufficiently large
$\ell$, each summand admits an action
of
$
A_\ell
=
\Lambda_{R,[\nu r_0/p^{\ell+1},\infty]}
(\Gamma).
$
On the first summand,
$\varphi^\ell(D_{I_0})$, the action
is transported from $D_{I_0}$,
since $\varphi$ commutes with $\Gamma$.
For sufficiently large $\ell$,
$
\nu r_0/p^{\ell+1}\leq s_0.
$
The restriction homomorphism
$
A_\ell
\longrightarrow
\Lambda_{R,[s_0,\infty]}(\Gamma)$
therefore gives the required action
on this summand.

For $0\leq j<\ell$, Lemma \ref{lem:closed-psi-zero-projective} gives an action of
$
\Lambda_{R,
[\nu r_0/p^{\ell-j+1},
 \nu r_0/p^{\ell-j}]}(\Gamma)
$
on
$
D_{p^jI_\ell}^{\psi=0}.
$
Since
$
\nu r_0/p^{\ell+1}
\leq
\nu r_0/p^{\ell-j+1},
$
restriction of analytic functions
gives a homomorphism from $A_\ell$
to this annular distribution algebra.
Transporting the action through
$\varphi^j$ gives the required action
on
$
\varphi^j
\left(D_{p^jI_\ell}^{\psi=0}\right).
$
Thus, for all sufficiently large
$\ell$, the above decomposition
defines a continuous action of
$A_\ell$ on $D_{I_\ell}$.
It extends the original
$R[\Gamma]$-action.
These actions are compatible with
restriction to closed subannuli:
the Frobenius and $\psi$-decompositions
are compatible with restriction,
and all the resulting actions
agree with the original
$R[\Gamma]$-action.
Their agreement on the distribution
algebras follows by continuity
and density of $R[\Gamma]$.

Now choose $b>0$ sufficiently small
that every annulus $I_\ell$
meeting $(0,b]$ has the preceding
action. For $0<a\leq b$, cover
$[a,b]$ by its intersections with
the finitely many annuli $I_\ell$.
The action on each intersection
restricts to an action of
$\Lambda_{R,[\nu a,\infty]}(\Gamma)$,
because its inner endpoint is
at least $a$. Compatibility on overlaps allows
these actions to glue, giving
$$
\Lambda_{R,[\nu a,\infty]}(\Gamma)
\curvearrowright D_{[a,b]}.
$$

Finally, let $K$ be arbitrary.
Apply the result just proved to
the induced $(\varphi,\Gamma_{\mathbb Q_p})$-
module
$
M=\operatorname{Ind}_K^{\mathbb Q_p}D.
$
By the subgroup and radius conventions
of Remark~\ref{compdifflamb},
the distribution action on $M$
restricts to an action of
$
\Lambda_{R,[e_K\nu a,\infty]}(\Gamma_K).$
As a $\Gamma_K$-module, the induced
module is a finite direct sum of
copies of $D$, and the projection
onto any copy is continuous and
$\Gamma_K$-equivariant.
By density, the projection is also
equivariant for the extended
distribution action.
Consequently, the action restricts
to $D_{[a,b]}$, as required.
\end{proof}

\begin{lemma}
\label{311}
[Theorem 3.1.1, \cite{kedlaya2014cohomology}]
    For a $(\varphi, \Gamma_K)$ module $D$, there exists a small enough $b$ such that $D_{(0, b]}^{\psi = 0}$ is a finite projective module of rank $[K: \mathbb{Q}_p] \textrm{rank}(D) $ over $\Lambda_{R, (0, \mu(b)]}(\Gamma_K)$.
\end{lemma}
\begin{proof}
From the observation:
$$
\left( \operatorname{Ind}_{K}^{\mathbb{Q}_p} D \right)^{\psi = 0}
=
\left( \operatorname{Hom}_{\mathbb{Z}[\Gamma_K]}\!\left( \mathbb{Z}[\Gamma], D \right) \right)^{\psi = 0}
=
\operatorname{Hom}_{\mathbb{Z}[\Gamma_K]}\!\left( \mathbb{Z}[\Gamma], D^{\psi = 0} \right).
$$
it is enough to assume that $D$ is a $(\varphi, \Gamma_{\mathbb{Q}_p})$ module over $\Lambda_{R, rig, \mathbb{Q}_p}$.

  From Lemma 3.7, \cite{bellovin2024cohomology}, for some small enough $b$, the action of $\gamma-1$ on $D_{(0, b]}^{\psi = 0}$ admits a continuous inverse. This extends to an action of $\Lambda_{R, (0, \mu(b)]}(\Gamma_K)$. After this, the proof follows from Theorem 3.1.1, \cite{kedlaya2014cohomology} and Lemma \ref{lem:closed-psi-zero-projective}.
\end{proof}
\begin{corollary}
\label{annihilatorbehavior}
    Let $f \in \Lambda_{R, (0, \infty]}(\Gamma_K)$ such that for some $x \in D_{(0, b]}$, $f(\gamma-1). x = 0$ and $f^{-1} \in \Lambda_{R, [\mu(a), \infty]}(\Gamma_K)$, then $x = 0$.
\end{corollary}
\begin{proof}
    For a given $0<a\leq b$, once $\mu(a)$ has been chosen so that the extended action in Lemma \ref{moduleaction} exists, no unit element in $\Lambda_{R, [\mu(a), \infty]}(\Gamma_K)$ can annihilate a nonzero element of $D_{[a, b]}$.
\end{proof}

\begin{definition}
     Let $\iota: \Gamma_K \to \Gamma_K$ be the involution sending $\gamma \mapsto \gamma^{-1}$. This induces an involution $\iota: \Lambda_{R, I}(\Gamma_K) \to \Lambda_{R, I}(\Gamma_K)$. For a $\Lambda_{R, I}(\Gamma_K)$-module $N$, we write $N^{\iota}$ to be the same group $N$ but with module structure twisted through $\iota$.
\end{definition}
\section{Iwasawa cohomology}
\begin{definition}
Let $R$ be a commutative ring. We denote by
$$
D^{[a,b]}_{\mathrm{perf}}(R)
\quad (\text{resp. } D^{b}_{\mathrm{perf}}(R),\; D^{-}_{\mathrm{perf}}(R))
$$
the full subcategory of the derived category of complexes of $R$-modules
consisting of those complexes which are quasi-isomorphic to a complex of
finite projective $R$-modules concentrated in degrees in $[a,b]$
(resp. bounded degrees, degrees bounded above).
\end{definition}
Let $D$ be a sheaf of projective modules over the pseudorigid space $\bigcup_{a \to 0} \textrm{Spa}(\Lambda_{R, [a, b], K})$. Then the Iwasawa complex $C_{\psi}^{\bullet}(D)$ of $D$ is defined as
$$colim_{b' \to 0, b' \leq b}(\begin{tikzcd}
	{D_{(0, b']}} && {D_{(0, b']}}
	\arrow["{\psi-1}", from=1-1, to=1-3]
\end{tikzcd})$$
where the transition maps are standard inclusions.

Using remark \ref{sheafmodule}, one could also view the sheaf $D$ as a module $D_{(0, b]}$ over $\Lambda_{R, (0, b], K}$ and define another Iwasawa complex in degree $1, 2$ denoted by $C'^{\bullet}_{\psi}(D)$ as:
$$\begin{tikzcd}
	{D_{(0, b]} \otimes \Lambda_{R, \text{rig}, K}} && {D_{(0, b]} \otimes \Lambda_{R, \text{rig}, K}}
	\arrow["{\psi-1}", from=1-1, to=1-3]
\end{tikzcd}.$$ Intuitively, these two definitions should be equivalent as $\Lambda_{R, \text{rig}, K} = \varinjlim_{b' \to 0} \Lambda_{R, (0, b'], K}$ and tensoring with $\Lambda_{R, \text{rig}, K}$ involves tensoring with its parts $\Lambda_{R, (0, b'], K}$, which means accounting for $D_{(0, b']}$ for arbitrary small $b'$ in the cohomology as well. Later in Lemma \ref{iwasawadefnequiv}, we prove their equivalence.

Later we will see that to compute Iwasawa cohomology, restricting to a sufficiently small annulus is enough.
\subsection{Base change from Iwasawa to Galois cohomology}

Let $\gamma$ be the topological generator of $\Gamma_K$ (note that this convention can differ in other sections of this paper, where we will take $\gamma$ to be topological generator of torsion free part of $\Gamma_K$). The Herr (Galois cohomology) complex of a $(\varphi, \Gamma)$-module $D$ over $\Lambda_{R, \text{rig}, K}$ is defined in degrees $0, 1, 2$ as:
$$C_{\varphi, \Gamma}^{\bullet}:= \varinjlim_{b \to 0} \begin{tikzcd}
	{D_{(0, b]}} && {D_{(0, b/p]} \oplus D_{(0, b]}} && {D_{(0, b/p]}}
	\arrow["{(\varphi-1, \gamma-1)}", from=1-1, to=1-3]
	\arrow["{((\gamma-1) \oplus (1 -\varphi))}", from=1-3, to=1-5]
\end{tikzcd}$$
 This is quasi-isomorphic to the $C_{\psi, \Gamma}^{\bullet}$ complex defined by 
 $$C_{\psi, \Gamma}^{\bullet}:= \varinjlim_{b \to 0} \begin{tikzcd}
	{D_{(0, b]}} && {D_{(0, b]} \oplus D_{(0, b]}} && {D_{(0, b]}}
	\arrow["{(\psi-1, \gamma-1)}", from=1-1, to=1-3]
	\arrow["{((\gamma-1) \oplus (1 -\psi))}", from=1-3, to=1-5]
\end{tikzcd}$$
Therefore, we will use them interchangeably. The quasi-isomorphism is standard and given in references such as \cite{kedlaya2014cohomology}.
Let $\eta:\Gamma_K\to L^\times$ be a continuous character, where
$L$ is a finite extension of either $\mathbf Q_p$ or
$\mathbf F_p((u))$. In characteristic zero, use the base-character
evaluation map
$$
\Lambda_{D_\lambda,(0,\infty]}(\Gamma_K)
\longrightarrow \mathcal O_L\otimes_{\mathbb{Z}_p}D_\lambda.
$$
In characteristic $p$, use evaluation in $L$ with the specified
coefficient map $D_\lambda\to\mathbf F_p((u))\to L$.
Let $I_\eta$ denote the kernel of the appropriate map. Put
$$
\kappa_\eta
=\Lambda_{D_\lambda,(0,\infty]}(\Gamma_K)/I_\eta,
\qquad
S_\eta=R\widehat\otimes_{D_\lambda}\kappa_\eta.
$$
We write $I_\eta\Lambda_R$ for the extended ideal, where
$\Lambda_R=\Lambda_{R,(0,\infty]}(\Gamma_K)$. Thus
$$
\Lambda_R/I_\eta\Lambda_R\cong S_\eta.
$$
This is a finite $R$-algebra, but it need not be a field or be reduced.
In particular, $I_\eta\Lambda_R$ is not in general a maximal ideal.
For a maximal point $z$ of a Banach stage, we instead use its actual
evaluation map and residue field, as in
Theorem~\ref{pointbasechange} below.

We prove the following base change theorem in detail, analogous to Proposition 4.2.3, \cite{kedlaya2014cohomology}.

  \begin{theorem}[Base-character specialization]
\label{basechange}
With the notation above, there is a natural quasi-isomorphism
$$
C_\psi^\bullet(D)
\otimes_{\Lambda_R}^{\mathbf L}S_\eta
\simeq
C_{\psi,\gamma_K}^\bullet
\left((D\widehat\otimes_R S_\eta)(\eta^{-1})\right).
$$
Here the $\Lambda_R$-action on $S_\eta$ is induced by the
base-character evaluation, and the character twist has trivial
$\varphi$-action.
\end{theorem}
\begin{proof}
Define
$$
\Lambda=\Lambda_{R,(0,\infty]}(\Gamma_K),
\qquad
\kappa_\eta=
\Lambda_{D_\lambda,(0,\infty]}(\Gamma_K)/I_\eta,
$$
and
$
S_\eta=
R\widehat{\otimes}_{D_\lambda}\kappa_\eta.
$
Then
$
\Lambda/I_\eta\Lambda\cong S_\eta.
$
We first separate the change of coefficients from the specialization in the
$\Gamma_K$-variable.
We claim that there is a natural isomorphism
$
\Lambda_{R,(0,\infty]}(\Gamma_K)
\widehat{\otimes}_R S_\eta
\cong
\Lambda_{S_\eta,(0,\infty]}(\Gamma_K).
$
Indeed, this can be checked at the Banach stages. Suppose first that
$\Gamma_K$ is torsion free, choose a topological generator $\gamma$, and put
$T=\gamma-1$. For every admissible $a$, if
$
e_a=\frac{1}{a\,v_{\mathbb C_p^{\flat}}(\bar\pi)},
$
then
$
\Lambda_{R_0,[a,\infty]}(\Gamma_K)
=
R_0[[T]]
\left\langle
\frac{T^{e_a}}{u}
\right\rangle.
$
If $S_\eta^\circ\subset S_\eta$ is the image of a ring of definition under
coefficient extension, then completed tensoring with $S_\eta^\circ$ replaces
the coefficients $R_0$ by $S_\eta^\circ$ and leaves the variable $T$ and its
growth condition unchanged. Hence
$
\Lambda_{R_0,[a,\infty]}(\Gamma_K)
\widehat{\otimes}_{R_0}S_\eta^\circ
\cong
S_\eta^\circ[[T]]
\left\langle
\frac{T^{e_a}}{u}
\right\rangle
=
\Lambda_{S_\eta^\circ,[a,\infty]}(\Gamma_K).
$
After inverting $u$, we obtain
$
\Lambda_{R,[a,\infty]}(\Gamma_K)
\widehat{\otimes}_R S_\eta
\cong
\Lambda_{S_\eta,[a,\infty]}(\Gamma_K).
$
Since
$
\Lambda_{R,(0,\infty]}(\Gamma_K)
=
\varprojlim_{a\to0}
\Lambda_{R,[a,\infty]}(\Gamma_K)
$
with its Frechet topology. Therefore
$$
\Lambda_{R,(0,\infty]}(\Gamma_K)
\widehat{\otimes}_R S_\eta
\cong
\varprojlim_{a\to0}
\Lambda_{S_\eta,[a,\infty]}(\Gamma_K)
=
\Lambda_{S_\eta,(0,\infty]}(\Gamma_K).
$$
If $\Delta_K$ is nontrivial, put $d_K=|\Delta_K|$ and perform the same calculation with
$R_0[\Delta_K][[T]]$ and exponent
$e(s)=1/(d_Ks\,v_{\mathbb{C}_p^\flat}(\bar\pi))$; coefficient
extension leaves both $\Delta_K$ and this rescaled radius unchanged.
Likewise, using the explicit presentation of the closed annulus rings, for
every $0<a\leq b$ we have
$
\Lambda_{R,[a,b],K}
\widehat{\otimes}_R S_\eta
\cong
\Lambda_{S_\eta,[a,b],K}.
$
Passing to the inverse limit over $a$ gives
$
\Lambda_{R,(0,b],K}
\widehat{\otimes}_R S_\eta
\cong
\Lambda_{S_\eta,(0,b],K}.
$
Define
$
D_{S_\eta}:=
D\widehat{\otimes}_R S_\eta.
$
Since $D$ is finite projective over $\Lambda_{R,(0,b],K}$, the preceding
coefficient-base-change identification gives
$
D_{S_\eta}
\cong
D\otimes_{\Lambda_{R,(0,b],K}}
\Lambda_{S_\eta,(0,b],K}.
$
Thus $D_{S_\eta}$ is naturally a $(\varphi,\Gamma_K)$-module over
$\Lambda_{S_\eta,(0,b],K}$.
We now compare the Iwasawa complexes before and after coefficient base
change. The ring $\Lambda_{R,(0,\infty]}(\Gamma_K)$ is flat over $R$ by the
same topological-basis argument as in Lemma \ref{robbaflat}. Moreover,
$\Lambda_{R,(0,b],K}$ is flat over $R$, and $D$ is finite projective over
$\Lambda_{R,(0,b],K}$. Hence $D$ is flat over $R$.
And so,
$$
\begin{aligned}
C_\psi^\bullet(D)
\otimes_\Lambda^{\mathbf L}
\Lambda_{S_\eta,(0,\infty]}(\Gamma_K)
&\cong
C_\psi^\bullet(D)
\otimes_\Lambda^{\mathbf L}
\left(
\Lambda\widehat{\otimes}_R S_\eta
\right) \\
&\cong
C_\psi^\bullet(D)
\otimes_R^{\mathbf L}S_\eta.
\end{aligned}
$$
We first note that $S_{\eta}$ is finite over $R$.\footnote{We justify that $\kappa_\eta$ is finite over $D_\lambda$. Suppose first that $L$ is a finite extension of $\mathbb Q_p$. Since $\mathcal O_L$ is a finite free $\mathbb Z_p$-module, the algebra $\mathcal O_L\otimes_{\mathbb Z_p}D_\lambda$ is finite free over $D_\lambda$ and is already complete. By the definition of $I_\eta$ as the kernel of evaluation, we have an injection
$$\kappa_\eta=\Lambda_{D_\lambda,(0,\infty]}(\Gamma_K)/I_\eta\hookrightarrow\mathcal O_L\otimes_{\mathbb Z_p}D_\lambda.$$
Since $D_\lambda$ is Noetherian, this implies that $\kappa_\eta$ is finite over $D_\lambda$.
Suppose now that $L$ is a finite extension of $\mathbb F_p((u))$, and define evaluation using the scalar map $D_\lambda\to D_\lambda/pD_\lambda=\mathbb F_p((u))\to L$. Then $\kappa_\eta$ is the image of evaluation in $L$. It is therefore a $\mathbb F_p((u))$-subspace of the finite-dimensional vector space $L$, so it is finite over $\mathbb F_p((u))$, and hence over $D_\lambda$. In either case, choose a finite presentation $D_\lambda^s\to D_\lambda^r\to\kappa_\eta\to0$. Tensoring with $R$ gives
$$R^s\longrightarrow R^r\longrightarrow R\otimes_{D_\lambda}\kappa_\eta\longrightarrow0.$$
Thus $R\otimes_{D_\lambda}\kappa_\eta$ is a finite $R$-module. With its canonical finite-module topology it is already complete, so
$$S_\eta=R\widehat{\otimes}_{D_\lambda}\kappa_\eta\cong R\otimes_{D_\lambda}\kappa_\eta.$$
Consequently $S_\eta$ is finite over $R$. This does not imply that $S_\eta$ is a field.} Since both terms of $C_\psi^\bullet(D)$ are $R$-flat, the last derived tensor
product is represented by the ordinary completed coefficient base change.
Therefore
$
C_\psi^\bullet(D)
\otimes_\Lambda^{\mathbf L}
\Lambda_{S_\eta,(0,\infty]}(\Gamma_K)
\cong
C_\psi^\bullet(D_{S_\eta}).
$
We now perform the character specialization.
Write
$
\Gamma_K=\Delta_K\times\Gamma,
$
where $\Delta_K$ is the torsion subgroup and $\Gamma$ is torsion free.
Choose a topological generator $\gamma$ of $\Gamma$ and put
$
c=\eta(\gamma).
$
Since $|\Delta_K|$ is prime to $p$, it is invertible in $S_\eta$. Define
$
e_\eta
=
\frac{1}{|\Delta_K|}
\sum_{\delta\in\Delta_K}
\eta(\delta)^{-1}\delta
\in
\Lambda_{S_\eta,(0,\infty]}(\Gamma_K),
$
and
$
P_\eta=
e_\eta\Lambda_{S_\eta,(0,\infty]}(\Gamma_K).
$
Since $e_\eta$ is an idempotent, $P_\eta$ is a finite projective
$\Lambda_{S_\eta,(0,\infty]}(\Gamma_K)$-module.
There is an exact sequence
$$
0
\longrightarrow
P_\eta
\xrightarrow{\gamma-c}
P_\eta
\xrightarrow{\varepsilon_\eta}
S_\eta
\longrightarrow
0,
$$
where $\varepsilon_\eta$ is evaluation through the character $\eta$. Indeed, on the $e_\eta$-isotypic component for $\Delta_K$ we have
$
P_\eta
\cong
\Lambda_{S_\eta,(0,\infty]}(\Gamma).
$
Writing $T=\gamma-1$ and $\alpha=c-1$, the map $\varepsilon_\eta$ becomes
evaluation at $T=\alpha$. Since $T-\alpha$ is monic, division by
$T-\alpha$ at every sufficiently small Banach stage containing $\alpha$ gives
$
\ker(\varepsilon_\eta)
=
(T-\alpha)\Lambda_{S_\eta,(0,\infty]}(\Gamma).
$
Moreover multiplication by $T-\alpha$ is injective. This proves the exact
sequence above.
We now have
$$
\begin{aligned}
C_\psi^\bullet(D)
\otimes_\Lambda^{\mathbf L}S_\eta
&\cong
\left(
C_\psi^\bullet(D)
\otimes_\Lambda^{\mathbf L}
\Lambda_{S_\eta,(0,\infty]}(\Gamma_K)
\right)
\otimes_{\Lambda_{S_\eta,(0,\infty]}(\Gamma_K)}^{\mathbf L}
S_\eta \\
&\cong
C_\psi^\bullet(D_{S_\eta})
\otimes_{\Lambda_{S_\eta,(0,\infty]}(\Gamma_K)}^{\mathbf L}
S_\eta.
\end{aligned}
$$
Using the projective resolution
$$
0
\longrightarrow
P_\eta
\xrightarrow{\gamma-c}
P_\eta
\longrightarrow
S_\eta
\longrightarrow
0,
$$
the last derived tensor product is represented by the total complex
$$
e_\eta D_{S_\eta}
\xrightarrow{d_0}
e_\eta D_{S_\eta}\oplus e_\eta D_{S_\eta}
\xrightarrow{d_1}
e_\eta D_{S_\eta},
$$
where
$
d_0(x)
=
\bigl(
(\psi-1)x,
(\gamma-c)x
\bigr)
$ and
$
d_1(y,z)
=
-(\gamma-c)y+(\psi-1)z.
$
These maps form a complex because $\psi$ commutes with $\gamma$. Let $
N=D_{S_\eta}(\eta^{-1}).
$
For $x\in D_{S_\eta}$, the corresponding element of $N$ is fixed by
$\Delta_K$ if and only if
$
\eta(\delta)^{-1}\delta(x)=x
$
for every $\delta\in\Delta_K$, or equivalently
$
\delta(x)=\eta(\delta)x.
$
Hence
$
N^{\Delta_K}
\cong
e_\eta D_{S_\eta}.
$

Under this identification, the twisted action of $\gamma$ is
$
\gamma_{\mathrm{tw}}=c^{-1}\gamma.
$
Therefore
$
\gamma_{\mathrm{tw}}-1
=
c^{-1}(\gamma-c).
$
Thus the Herr complex of $N$ is represented by
$$
e_\eta D_{S_\eta}
\xrightarrow{d_0'}
e_\eta D_{S_\eta}\oplus e_\eta D_{S_\eta}
\xrightarrow{d_1'}
e_\eta D_{S_\eta},
$$
where
$
d_0'(x)
=
\bigl(
(\psi-1)x,
c^{-1}(\gamma-c)x
\bigr)
$ and
$
d_1'(y,z)
=
-c^{-1}(\gamma-c)y+(\psi-1)z.
$
Define
$$
F_0(x)=x,
\qquad
F_1(y,z)=(y,c^{-1}z),
\qquad
F_2(x)=c^{-1}x.
$$
Then
$
F_1\circ d_0=d_0'\circ F_0
$
and
$
d_1'\circ F_1=F_2\circ d_1.
$
Since $c\in S_\eta^\times$, each $F_i$ is an isomorphism. Hence
$$
C_\psi^\bullet(D_{S_\eta})
\otimes_{\Lambda_{S_\eta,(0,\infty]}(\Gamma_K)}^{\mathbf L}
S_\eta
\cong
C_{\psi,\gamma}^\bullet
\left(
D_{S_\eta}(\eta^{-1})
\right).
$$
Finally,
$
D_{S_\eta}
=
D\widehat{\otimes}_R
\left(
R\widehat{\otimes}_{D_\lambda}\kappa_\eta
\right)
\cong
D\widehat{\otimes}_{D_\lambda}\kappa_\eta,
$
and
$
S_\eta
\cong
\Lambda_{R,(0,\infty]}(\Gamma_K)/
I_\eta\Lambda.
$
Therefore
$
C_\psi^\bullet(D)
\otimes_{\Lambda_{R,(0,\infty]}(\Gamma_K)}^{\mathbf L}
\frac{\Lambda_{R,(0,\infty]}(\Gamma_K)}
{I_\eta\Lambda}
\cong
C_{\psi,\gamma}^\bullet
\left(
D\widehat{\otimes}_{D_\lambda}
\kappa_\eta(\eta^{-1})
\right).
$
\end{proof}

\begin{theorem}[Specialization at a maximal point]
\label{pointbasechange}
Put $\Lambda_R=\Lambda_{R,(0,\infty]}(\Gamma_K)$ and
$A_a=\Lambda_{R,[\mu(a),\infty]}(\Gamma_K)$.
Let $\mathfrak n_{z,a}$ be a maximal ideal of $A_a$, and put
$$
L_z=A_a/\mathfrak n_{z,a}.
$$
Write $\varepsilon_z:\Lambda_R\to L_z$ for the restricted
evaluation map, $\eta_z:\Gamma_K\to L_z^\times$ for the induced
character, and
$$
N_z=(D\widehat\otimes_R L_z)(\eta_z^{-1}).
$$
There is a natural quasi-isomorphism
$$
C_\psi^\bullet(D)
\otimes_{\Lambda_R,\varepsilon_z}^{\mathbf L}L_z
\simeq C_{\psi,\gamma_K}^\bullet(N_z).
$$
Consequently, if
$C_{\psi,a}^\bullet=C_\psi^\bullet(D)
\otimes_{\Lambda_R}^{\mathbf L}A_a$, then
$$
C_{\psi,a}^\bullet\otimes_{A_a}^{\mathbf L}L_z
\simeq C_{\psi,\gamma_K}^\bullet(N_z).
$$
\end{theorem}

\begin{proof}
The coefficient map $R\to L_z$ factors through a maximal
residue field of $R$, and $L_z$ is finite over that field.
In particular, $L_z$ is a finite $R$-module.
The finite coefficient-base-change identifications used in the
preceding proof give
$$
B_z:=\Lambda_{L_z,(0,\infty]}(\Gamma_K)
\cong\Lambda_R\otimes_R L_z.
$$
Since $\Lambda_R$ and the terms of $C_\psi^\bullet(D)$ are
$R$-flat, it follows that
$$
C_\psi^\bullet(D)\otimes_{\Lambda_R}^{\mathbf L}B_z
\simeq C_\psi^\bullet(D\widehat\otimes_R L_z).
$$
These are finite coefficient-base-change identifications, so no
perfectness statement is needed here.

Write $\Gamma_K=\Delta_K\times\Gamma_K^{\mathrm{tf}}$, and choose
a topological generator $\gamma_K$ of $\Gamma_K^{\mathrm{tf}}$.
Put
$$
c_z=\eta_z(\gamma_K),
\qquad
e_z=\frac{1}{|\Delta_K|}
\sum_{\delta\in\Delta_K}\eta_z(\delta)^{-1}\delta,
\qquad P_z=e_zB_z.
$$
The module $P_z$ is finite projective over $B_z$. Evaluation gives
the projective resolution
$$
0\longrightarrow P_z
\xrightarrow{\gamma_K-c_z}P_z
\longrightarrow L_z\longrightarrow0.
$$
Indeed, on this tame-character component, put $T=\gamma_K-1$.
The kernel of evaluation at $T=c_z-1$ is generated by
$T-(c_z-1)$, and multiplication by this element is injective.
This follows by division on the Banach disks containing the point
and passage to their inverse limit.

Tensoring with this resolution represents the derived specialization
by the three-term total complex of $\psi-1$ and $\gamma_K-c_z$ on
$e_z(D\widehat\otimes_R L_z)$, in degrees $0,1,2$.
Under the identification
$$
N_z^{\Delta_K}\cong e_z(D\widehat\otimes_R L_z),
$$
the twisted generator acts as $c_z^{-1}\gamma_K$, so its
Herr differential is $c_z^{-1}(\gamma_K-c_z)$.
Rescaling the corresponding summands identifies this total complex
with $C_{\psi,\gamma_K}^\bullet(N_z)$.
The Banach-stage assertion follows from associativity of derived
tensor products.
\end{proof}

\subsection{Coadmissibility and perfectness}

\begin{lemma}
\label{basic}
    Let $D$ be a $(\varphi, \Gamma_K)$ module over $\Lambda_{R, (0, pb], K}$. Then $\text{colim}_{n \to \infty} D_{(0, b/p^n]}^{\psi = 1} = D_{(0, b]}^{\psi = 1}$
    In other words, let $n>1$ and $m \in D_{(0, b/p^n]}^{\psi =1}$, then $m$ comes from $D_{(0, b/p^{n-1}]}^{\psi =1}$ under the natural map $D_{(0, b/p^{n-1}]}^{\psi =1} \to D_{(0, b/p^{n}]}^{\psi =1}$ (which we would also refer to as $m$ comes from a bigger radius). We also have for $a \leq b/p^{n+1}$ that the map $D_{[a, b/p^{n-1}]}^{\psi =1} \to D_{[a, b/p^{n}]}^{\psi =1}$ is an isomorphism. 
\end{lemma}

\begin{proof}
Let $m\in D_{(0,b/p^n]}^{\psi=1}$. Using the decomposition
$D_{(0,b/p^n]}=\varphi(D_{(0,b/p^{n-1}]})\oplus D_{(0,b/p^n]}^{\psi=0}$,
write
$m=\varphi(m')+n$,
where $m'\in D_{(0,b/p^{n-1}]}$ and $n\in D_{(0,b/p^n]}^{\psi=0}$.
Applying $\psi$ to this decomposition and using $\psi(m)=m$, we obtain
$m'=m$. Since $m=\varphi(m')+n$, it follows that
$m'-\varphi(m')=n$. Applying $\psi$ once more gives
$\psi(m')-\psi(\varphi(m'))=\psi(n)$, that is,
$\psi(m')-m'=0$, since $\psi\circ\varphi=\mathrm{id}$ and
$\psi(n)=0$. Hence $m'\in D_{(0,b/p^{n-1}]}^{\psi=1}$.
Finally, the equality $m=m'$ shows that every element of
$D_{(0,b/p^n]}^{\psi=1}$ already lies in
$D_{(0,b/p^{n-1}]}^{\psi=1}$. Therefore, the natural inclusion
$D_{(0,b/p^{n-1}]}\hookrightarrow D_{(0,b/p^n]}$
induces an isomorphism on $\psi$-invariants. The second part of the Lemma follows from same radius bookkeeping.
\end{proof}
\begin{lemma}
\label{psicoinv}
    $D_{(0, b/p^{n-1}]}/(\psi-1)D_{(0, b/p^{n-1}]} \simeq D_{(0, b/p^{n}]}/(\psi-1)D_{(0, b/p^{n}]}$ under the natural map. 
\end{lemma}

\begin{proof}
Consider the natural map
$$
D_{(0,b/p^{n-1}]}\longrightarrow
D_{(0,b/p^n]}/(\psi-1)D_{(0,b/p^n]}.
$$
Its kernel is
$(\psi-1)D_{(0,b/p^n]}\cap D_{(0,b/p^{n-1}]}$.
Let $m\in D_{(0,b/p^n]}$ satisfy
$\psi(m)-m\in D_{(0,b/p^{n-1}]}$. Since
$\psi(m)\in D_{(0,b/p^{n-1}]}$, it follows that
$m\in D_{(0,b/p^{n-1}]}$. Hence
$
(\psi-1)D_{(0,b/p^n]}\cap D_{(0,b/p^{n-1}]}
=
(\psi-1)D_{(0,b/p^{n-1}]}.
$
It remains to prove surjectivity. Using the decomposition
$
D_{(0,b/p^n]}
=
\varphi(D_{(0,b/p^{n-1}]})
\oplus
D_{(0,b/p^n]}^{\psi=0},
$
write
$m=\varphi(m')+n$,
where $m'\in D_{(0,b/p^{n-1}]}$ and
$n\in D_{(0,b/p^n]}^{\psi=0}$.
Then
$
\psi(m)-m
=
m'-m,
$
since $\psi(\varphi(m'))=m'$ and $\psi(n)=0$. Therefore
$m\equiv m'\pmod{(\psi-1)D_{(0,b/p^n]}}$,
and $m'\in D_{(0,b/p^{n-1}]}$. Thus every class in
$D_{(0,b/p^n]}/(\psi-1)D_{(0,b/p^n]}$
is represented by an element of
$D_{(0,b/p^{n-1}]}$, proving that the above map is surjective.
\end{proof}

\begin{lemma}
\label{referenceforbanachbasechange}
Put $\Lambda=\Lambda_{R,(0,\infty]}(\Gamma_K)$,
$A_a=\Lambda_{R,[\mu(a),\infty]}(\Gamma_K)$, and
$S_\eta=\Lambda/I_\eta\Lambda$.
Write $\Gamma_K=\Delta_K\times\Gamma_K^{\mathrm{tf}}$,
put $d_K=|\Delta_K|$, choose a generator $\gamma$ of
$\Gamma_K^{\mathrm{tf}}$, and assume that
$$e_a=\frac{1}{d_K\mu(a)\,v_{\mathbb{C}_p^\flat}(\bar\pi)}
\in\mathbb{Z}_{>0}.$$
Then
$$
A_a/I_\eta A_a\cong
S_\eta\left\langle\frac{(\eta(\gamma)-1)^{e_a}}{u}\right\rangle.
$$
Consequently there is a natural isomorphism
$$
\left(D_{(0,b]}\otimes_\Lambda A_a\right)
\widehat\otimes_{A_a}(A_a/I_\eta A_a)
\cong
D_{(0,b]}\widehat\otimes_\Lambda
\left(S_\eta\left\langle\frac{(\eta(\gamma)-1)^{e_a}}{u}\right\rangle\right).
$$
This is an algebra-quotient calculation; the displayed coefficient
algebra is not asserted to be a residue field.
\end{lemma}
\begin{proof}
With $T=\gamma-1$, the full-index convention gives the integral
presentation
$$R_0[\Delta_K][[T]]\left\langle T^{e_a}/u\right\rangle.$$
Under base-character evaluation, $T$ maps to $\eta(\gamma)-1$
and $\delta\in\Delta_K$ maps to $\eta(\delta)$.
The quotient therefore has the asserted rational-localization
presentation over $S_\eta$.
Since $A_a$ is Noetherian, its ideal $I_\eta A_a$ is closed,
and the completed quotient is the ordinary quotient.
Associativity of algebraic tensor products gives
$$
(D_{(0,b]}\otimes_\Lambda A_a)\otimes_{A_a}(A_a/I_\eta A_a)
\cong D_{(0,b]}\otimes_\Lambda(A_a/I_\eta A_a).
$$
The maps on pure tensors and their inverses are continuous for
the tensor topologies, so passage to separated completions gives
the module isomorphism.
\end{proof}

\begin{proposition}
\label{iwasawann}
    The cohomology groups of $C_{\psi}^{\bullet}$ for a $(\varphi, \Gamma_K)$ module $D_{(0, b]}$ over $\Lambda_{R, (0, b], K}$ are computed by the complex $$\begin{tikzcd}
	{D_{(0, b/p^n]}} && {D_{(0, b/p^n]}}
	\arrow["{\psi-1}", from=1-1, to=1-3]
\end{tikzcd}$$ for some $n \geq 1$.
\end{proposition}
\begin{proof}
    This is an easy corollary of Lemma \ref{basic} and Lemma \ref{psicoinv}.
\end{proof}
\begin{lemma}
    \label{iwasawadefnequiv}
    The complexes $C'^{\bullet}_{\psi}(D)$ and $C_{\psi}^{\bullet}(D)$ compute the same cohomology. (These two complexes were defined in the beginning of this section.)
\end{lemma}
\begin{proof}
    $C_{\psi}^{'\bullet}(D)$ is defined as:
    $$\begin{tikzcd}
	{D_{(0, b]} \otimes \Lambda_{R, \text{rig}, K}} && {D_{(0, b]} \otimes \Lambda_{R, \text{rig}, K}}
	\arrow["{\psi-1}", from=1-1, to=1-3]
\end{tikzcd}.$$
Since ${D_{(0, b]} \otimes \Lambda_{R, \text{rig}, K}} = \varinjlim_{b'\to 0} D_{(0, b']}$, and filtered colimits are exact, we have $H^i_{\psi}(D_{\text{rig}}) = \varinjlim_{b' \to 0} H^i_{\psi}(D_{(0, b']})$.
Proposition \ref{iwasawann} tells us that this limit eventually stabilizes. Hence computing the cohomology at $b'=b/p^n$ for large enough $n$ gives us what we need.
\end{proof}
Hence, we just prove the perfectness of $\begin{tikzcd}
	{D_{(0, b']}} && {D_{(0, b']}}
	\arrow["{\psi-1}", from=1-1, to=1-3]
\end{tikzcd}$ for a small enough $b'$. This is different from Proposition 3.8, \cite{bellovin2024cohomology}, where the Herr cohomology can be just computed from $$D_{[a,b]}
\xrightarrow{(\varphi - 1,\ \gamma - 1)}
D_{[a,b/p]} \oplus D_{[a,b]}
\xrightarrow{(\gamma - 1)\ \oplus\ (1 - \varphi)}
D_{[a,b/p]}$$ for sufficiently small $a, b$. This is because taking $\Gamma$ invariants makes the cohomology Herr complex really small and passing to a finite interval is enough. 

\begin{lemma}
\label{phiinv}
    The module $\text{colim}_{b' \to 0} D_{(0, b']}^{\varphi = 1}$ is $R$-finite.
\end{lemma}
\begin{proof}
   We first reduce to the case $K=\mathbb{Q}_p$. Since $\Gamma_K$ is an open
subgroup of $\Gamma_{\mathbb{Q}_p}$, the module $R[\Gamma_{\mathbb{Q}_p}]$
is finite free over $R[\Gamma_K]$. Recall that
$
\operatorname{Ind}_{K}^{\mathbb{Q}_p}D_{(0,b']}
=
\operatorname{Hom}_{R[\Gamma_K]}
\left(
R[\Gamma_{\mathbb{Q}_p}],D_{(0,b']}
\right).
$
Since $\varphi$ acts trivially on the group-ring factor and commutes with
the $\Gamma$-action, we have
$
\left(
\operatorname{Ind}_{K}^{\mathbb{Q}_p}D_{(0,b']}
\right)^{\varphi=1}
=
\operatorname{Hom}_{R[\Gamma_K]}
\left(
R[\Gamma_{\mathbb{Q}_p}],
D_{(0,b']}^{\varphi=1}
\right).
$
As $R[\Gamma_{\mathbb{Q}_p}]$ is finite free over $R[\Gamma_K]$, the
right-hand side is a finite direct sum of copies of
$D_{(0,b']}^{\varphi=1}$. Passing to the colimit over $b'\to0$, it is
therefore enough to prove the lemma for a $(\varphi,\Gamma_{\mathbb{Q}_p})$-module.
    Let $D^*=D^\vee(1)$. The residue pairing gives a
nondegenerate continuous $R$-bilinear pairing
$D\times D^*\longrightarrow R$ satisfying
$
\{x,\psi y\}=\{\varphi x,y\}.
$
Thus, if $\varphi x=x$, then
$
\{x,(\psi-1)y\}
=\{\varphi x,y\}-\{x,y\}=0.
$
The pairing therefore induces an injection
$
D_{\mathrm{rig}}^{\varphi=1}
\longrightarrow
\operatorname{Hom}_R
\left(D_{\mathrm{rig}}^*/(\psi-1)D_{\mathrm{rig}}^*,R\right).
$
Injectivity follows from nondegeneracy of the residue
pairing. The quotient on the right is finite over $R$
by \cite[Proposition~5.2]{bellovin2024cohomology}.
Its $R$-dual is finite, and $R$ is Noetherian.
Hence $D_{\mathrm{rig}}^{\varphi=1}$ is finite over $R$.
\end{proof}

\subsection{Complete continuity of Banach modules}
\begin{definition}
\label{compcontdef}
(Completely continuous morphism of Banach modules)
Let $A$ be a Banach algebra, and let $f:M\to N$ be a bounded morphism of Banach $A$-modules. We say that $f$ is completely continuous if there exists a sequence of finite $A$-submodules $N_i\subset N$ such that the operator norm of $M\to N\to N/N_i$ tends to zero, where $N/N_i$ is equipped with the quotient seminorm.
\end{definition}

We wanted to apply Lemma 1.10, \cite{kedlaya2016finiteness} to a completely continuous map of complexes to deduce finiteness of cohomology groups. Consider $A = \mathbb{Q}_p \langle T \rangle$, $B = \mathbb{Q}_p\langle \dfrac{T}{p} \rangle$. Then the identity map $id: B \to B$ satisfies the above property by taking $N_i = A$. This means that from Lemma 1.10, \cite{kedlaya2016finiteness}, $B$ should be finite $A$-module which it isn't.

This example is however contrived and one can use another criterion inspired from methods of \cite{kedlaya2016finiteness} that will help us prove what we need.

In what follows, $A=A_0[1/u]$ is a Banach Tate algebra, where $A_0$ is a Noetherian, $u$-adically complete ring of definition and $u$ is a topologically nilpotent unit of $A$. By an $A_0$-lattice in a Banach $A$-module we mean an open bounded $A_0$-submodule.

\begin{definition}
\label{integralcompcontdef}
(Integrally completely continuous morphism of Banach modules)
A bounded $A$-linear map $g:M\to N$ is integrally completely continuous if, for $A_0$-lattices $M_0\subset M$ and $N_0\subset N$ and every $k\geq1$, there exists a finitely generated $A_0$-submodule $L_k\subset N$ such that
$$g(M_0)\subset L_k+u^kN_0.$$
\end{definition}

This condition is independent of the chosen lattices. Indeed, given lattices $M'_0,N'_0$, choose $c,d\geq0$ such that $M'_0\subset u^{-c}M_0$ and $N_0\subset u^{-d}N'_0$. Applying the condition at precision $k+c+d$ gives
$$g(M'_0)\subset u^{-c}L_{k+c+d}+u^kN'_0.$$
It also implies Definition~\ref{compcontdef}: take the finite $A$-submodules $AL_k$ and use that $u^kN_0$ tends uniformly to zero.

Writing $L_k=\sum_{j=1}^rA_0v_j$, the integral condition allows us to choose, for every $m\in M_0$,
$$g(m)=\sum_{j=1}^r a_j(m)v_j+u^kn(m),\qquad a_j(m)\in A_0,\quad n(m)\in N_0.$$
These choices need not be additive or $A$-linear.  

The preceding example fails the integral condition. Take $A_0=\mathbb Z_p\langle T\rangle$, $B_0=\mathbb Z_p\langle U\rangle$ and $U=T/p$. If $\operatorname{id}_B$ were integrally completely continuous, there would be a finite $A_0$-submodule $L\subset B$ such that $B_0\subset L+pB_0$. For each $b\in B_0$, any expression $b=\ell+pb'$ with $\ell\in L$ and $b'\in B_0$ has $\ell\in L\cap B_0$. Hence $L\cap B_0\to B_0/pB_0$ would be surjective. Since $A_0$ is Noetherian, $L\cap B_0$ is finite over $A_0$. But
$$B_0/pB_0=\mathbb F_p[U],\qquad T=pU=0\pmod p,$$
so the $A_0$-action on $B_0/pB_0$ factors through $A_0/(p,T)=\mathbb F_p$. This would make $\mathbb F_p[U]$ finite-dimensional over $\mathbb F_p$, a contradiction.

\begin{lemma}
\label{integralcompcontoperations}
Integral complete continuity is preserved under addition and bounded precomposition and postcomposition. If $h:M'\to M$ is a bounded surjection and $g\circ h$ is integrally completely continuous, then $g$ is integrally completely continuous.
\end{lemma}
\begin{proof}
For addition, use the sum of the two finite $A_0$-submodules for each $k$. For bounded precomposition, the image of a source lattice is contained in a $u$-power multiple of a lattice in $M$. For bounded postcomposition $h:N\to P$, choose a target lattice $P_0$ and $c\geq0$ such that $h(N_0)\subset u^{-c}P_0$. Choosing finitely generated module $L_{k+c}$ so that $g(M_0)\subset L_{k+c}+u^{k+c}N_0$ then gives $(h\circ g)(M_0)\subset h(L_{k+c})+u^kP_0$, and $h(L_{k+c})$ is finite over $A_0$. Finally, if $h:M'\to M$ is surjective and $M'_0$ is a lattice in $M'$, the open mapping theorem shows that $h(M'_0)$ is a lattice in $M$. The approximations for $(g\circ h)(M'_0)$ are therefore approximations for $g(h(M'_0))$.
\end{proof}

\begin{lemma}
\label{integralcompcontperturbation}
Let $f,g:M\to N$ be bounded $A$-linear maps between Banach $A$-modules. Suppose that $\operatorname{coker}(f)$ is finite over $A$ and that $g$ is integrally completely continuous. Then $\operatorname{coker}(f+g)$ is finite over $A$.
\end{lemma}
\begin{proof}
We first suppose that $f$ is surjective. Choose a lattice $M_0\subset M$ and put $N_0=f(M_0)$. The open mapping theorem implies that $N_0$ is an open bounded $A_0$-submodule of $N$. Choose $L=\sum_{j=1}^rA_0v_j$ such that $g(M_0)\subset L+uN_0$. Fix $n_0\in N_0$. Recursively choose $m_i\in M_0$, $a_{ij}\in A_0$ and $e_i\in N_0$ such that
$$f(m_i)=n_i,\qquad g(m_i)=\sum_{j=1}^r a_{ij}v_j+ue_i,$$
and put $n_{i+1}=-e_i$. Then $n_i=(f+g)(m_i)-\sum_j a_{ij}v_j+un_{i+1}$, and hence
$$n_0=(f+g)\left(\sum_{i=0}^{s-1}u^im_i\right)-\sum_{j=1}^r\left(\sum_{i=0}^{s-1}u^ia_{ij}\right)v_j+u^sn_s.$$
Since $M_0$ and $N_0$ are bounded, $\sum_{i\geq0}u^im_i$ converges in $M$ and $u^sn_s$ tends to zero in $N$. Completeness of $A_0$ gives $\sum_{i\geq0}u^ia_{ij}\in A_0$. Thus $N_0\subset(f+g)(M)+L$. After inverting $u$, we obtain $N=(f+g)(M)+AL$, so the images of the $v_j$ generate $\operatorname{coker}(f+g)$.

For general $f$, choose a bounded map $q:A^r\to N$ whose image generates $\operatorname{coker}(f)$, and apply the surjective case to $(f,q):M\oplus A^r\to N$ and $(g,0)$. The latter is integrally completely continuous by Lemma~\ref{integralcompcontoperations}. The resulting cokernel is finite. Adjoining the images of $q$ gives finitely many generators of the original cokernel.
\end{proof}

\begin{lemma}
\label{integralcompcontstricttarget}
Let $g:M\to N$ be a bounded $A$-linear map and let $h:N\hookrightarrow N'$ be a strict inclusion of Banach $A$-modules. If $h\circ g$ is integrally completely continuous, then $g$ is integrally completely continuous.
\end{lemma}
\begin{proof}
Identify $N$ with its closed image in $N'$. Choose lattices $M_0\subset M$ and $N'_0\subset N'$, and put $N_0=N\cap N'_0$. For $k\geq1$, choose a finite $A_0$-submodule $L_k\subset N'$ such that $g(M_0)\subset L_k+u^kN'_0$. Since $A_0$ is Noetherian, $L_k\cap(N+u^kN'_0)$ is finite. Choose generators $w_1,\ldots,w_r$ and write $w_j=n_j+u^ke_j$, where $n_j\in N$ and $e_j\in N'_0$. For $m\in M_0$, choose $\ell\in L_k$ with $g(m)-\ell\in u^kN'_0$. Then $\ell\in L_k\cap(N+u^kN'_0)$, so $\ell=\sum_j a_jw_j$ with $a_j\in A_0$. It follows that
$$g(m)-\sum_j a_jn_j\in u^kN'_0\cap N=u^kN_0.$$
Thus $g(M_0)\subset\sum_jA_0n_j+u^kN_0$, as required.
\end{proof}

\begin{corollary}
\label{integralcompcontcohomology} (Cartan-Serre, Kedlaya-Liu, \cite{kedlaya2016finiteness})
Let $F^\bullet:C_1^\bullet\to C_2^\bullet$ be a morphism of complexes of Banach $A$-modules with bounded differentials. Suppose that $F^i$ is integrally completely continuous and that $H^i(C_1^\bullet)\to H^i(C_2^\bullet)$ has finite cokernel over $A$. Then $H^i(C_2^\bullet)$ is finite over $A$.
\end{corollary}
\begin{proof}
Put $Z_j^i=\ker(d_j^i)$; these are closed Banach submodules. Consider $f,g:Z_1^i\oplus C_2^{i-1}\to Z_2^i$ given by
$$f(z,b)=F^i(z)+d_2^{i-1}(b),\qquad g(z,b)=-F^i(z).$$
The cokernel of $f$ is the cokernel of the induced map on cohomology, hence is finite. Lemma~\ref{integralcompcontoperations} shows that $g$, viewed as a map into $C_2^i$, is integrally completely continuous. Lemma~\ref{integralcompcontstricttarget}, applied to $Z_2^i\hookrightarrow C_2^i$, gives the same assertion with target $Z_2^i$. Lemma~\ref{integralcompcontperturbation} now gives the finiteness of $\operatorname{coker}(f+g)=Z_2^i/d_2^{i-1}(C_2^{i-1})=H^i(C_2^\bullet)$.
\end{proof}

The corollary also applies if $F^\bullet$ is homotopic to a morphism whose degree $i$ map is integrally completely continuous, since homotopic morphisms induce the same map on cohomology. Thus it is enough to verify integral complete continuity for the components of the comparison map obtained after the homotopy.

To use complete continuity over the Banach distribution stages, we fix the following integral model.

Write $\Gamma_K=\Delta_K\times\Gamma_K^{\mathrm{tf}}$, put
$d_K=|\Delta_K|$, choose a generator $\gamma$ of
$\Gamma_K^{\mathrm{tf}}$, and put $T=\gamma-1$.
Under our full-index convention,
$$
\Lambda_{R,[s,\infty]}(\Gamma_K)
=\Lambda_{R,[d_Ks,\infty]}(\Gamma_K^{\mathrm{tf}})
\otimes_{\mathbb{Z}[\Gamma_K^{\mathrm{tf}}]}\mathbb{Z}[\Gamma_K].
$$
For admissible $s$ with
$$
e(s)=\frac{1}{d_Ks\,v_{\mathbb{C}_p^\flat}(\bar\pi)}
\in\mathbb{Z}_{>0},
$$
its integral presentation is
$$
\Lambda_{R_0,[s,\infty]}(\Gamma_K)
=R_0[\Delta_K][[T]]
\left\langle\frac{T^{e(s)}}{u}\right\rangle.
$$
In particular, at the stage indexed by $\mu(a)$ we put
$$
e_a=e(\mu(a))
=\frac{1}{d_K\mu(a)\,v_{\mathbb{C}_p^\flat}(\bar\pi)},
\qquad Y=\frac{T^{e_a}}{u}.
$$
Suppressing $\Delta_K$ in the notation does not remove the factor
$d_K$ from $e_a$.
We choose a decreasing sequence of admissible radii for which these
exponents are integers.

\begin{definition}
With the exponent $e_a$ above, put
$$
P_a=R_0[\Delta_K][[T]][Y],
\qquad M_{R,\mu(a)}=1+uP_a,
$$
and
$$
S_{R_0,\mu(a)}^{\mathrm{int}}
=M_{R,\mu(a)}^{-1}P_a,
\qquad
S_{R,\mu(a)}=S_{R_0,\mu(a)}^{\mathrm{int}}[1/u].
$$
Here $P_a$ denotes the subring generated by
$R_0[\Delta_K][[T]]$ and $Y=T^{e_a}/u$.
\end{definition}

\begin{proposition}
\label{localfff}
The natural map $S_{R,\mu(a)}\to\Lambda_{R,[\mu(a),\infty]}(\Gamma_K)$ is faithfully flat. Moreover, every $x\in\Lambda_{R_0,[\mu(a),\infty]}(\Gamma_K)$ can be written as $x=\sum_{i\geq0}u^if_i$ with $f_i\in S_{R_0,\mu(a)}^{\mathrm{int}}$.
\end{proposition}
\begin{proof}
Put $S_0=S_{R_0,\mu(a)}^{\mathrm{int}}$ and $A_0=\Lambda_{R_0,[\mu(a),\infty]}(\Gamma_K)$. The integral presentation identifies $A_0$ with the $u$-adic completion of $P_a$. Every element of $M_{R,\mu(a)}$ is already invertible modulo $u^N$, so
$$S_0/u^NS_0\cong P_a/u^N,\qquad A_0\cong\varprojlim_N S_0/u^NS_0.$$
The ring $S_0$ is Noetherian. Moreover, if $x=a/s\in S_0$, then $1+ux=(s+ua)/s$ is a unit, since $s+ua\in M_{R,\mu(a)}$. Thus $u\in\operatorname{Jac}(S_0)$, and the Noetherian completion theorem shows that $S_0\to A_0$ is faithfully flat. Inverting $u$ gives the asserted faithful flatness. Finally, lift $x$ modulo $u$ to $f_0\in S_0$, then lift $(x-f_0)/u$ modulo $u$ to $f_1$, and continue. Completeness gives $x=\sum_{i\geq0}u^if_i$.
\end{proof}

\begin{lemma}
\label{integralcoefficientaction}
Suppose that the action of $P_a$ on a Banach $R$-module $D$ is uniformly bounded. Then it extends to a uniformly bounded action of $S_0=S_{R_0,\mu(a)}^{\mathrm{int}}$. Moreover, $D$ has an open bounded $S_0$-stable lattice, and this action extends continuously to $\Lambda_{R_0,[\mu(a),\infty]}(\Gamma_K)$ while preserving that lattice.
\end{lemma}
\begin{proof}
Choose a constant $C$ bounding the operators defined by all elements of $P_a$. If $s=1+uh$ with $h\in P_a$, then $h^n$ belongs to the same ring for every $n$, and
$$s^{-1}=\sum_{n\geq0}(-u)^nh^n$$
converges in the ring of bounded operators on $D$. For $a\in P_a$, the operators $ah^n$ are also bounded by $C$. Thus the operators $a/s=\sum_{n\geq0}(-u)^nah^n$ are uniformly bounded, independently of $a$ and $s$.

Choose an open bounded $R_0$-lattice $P_0\subset D$. Uniform boundedness gives $h\geq0$ such that $P_0\subset S_0P_0\subset u^{-h}P_0$, where the span consists of finite sums. Hence $D_0=S_0P_0$ is an open bounded $S_0$-stable lattice. Being an open additive subgroup of a complete module, it is closed and complete, and its topology is the $u$-adic topology. For $\lambda\in\Lambda_{R_0,[\mu(a),\infty]}(\Gamma_K)$, choose compatible representatives $\lambda_n\in S_0$. Then $(\lambda_n-\lambda_m)D_0\subset u^rD_0$ for $n,m\geq r$. Thus $\lambda x=\lim_n\lambda_nx$ defines an action preserving $D_0$, independent of the representatives. The algebra identities pass to these limits.
\end{proof}

The hypothesis of Lemma~\ref{integralcoefficientaction} is satisfied for $D = D_{[a, b]}$ as when $T = \gamma-1$, $\Lambda_{R, [\mu(a), \infty]}(\Gamma_K)$ acts continuously on $D$ and $D$ is Banach.

\begin{lemma}
\label{integralresphi}
Let $D$ be a finite projective $(\varphi,\Gamma_K)$-module and let $a,b>0$ be sufficiently small, with $a\leq b/p^2$. Put $A_{\mu(a)}=\Lambda_{R,[\mu(a),\infty]}(\Gamma_K)$ and $A_{\mu(a)}^0=\Lambda_{R_0,[\mu(a),\infty]}(\Gamma_K)$. Then
$$\operatorname{res}\circ\varphi:D_{[a,b]}\longrightarrow D_{[a,b/p^2]}$$
is integrally completely continuous over $A_{\mu(a)}$.
\end{lemma}
\begin{proof}
We first prove the integral complete continuity over $R_0$. Put $J=[a/p,b/p]$ and $I=[a,b/p^2]$. The map factors as
$$D_{[a,b]}\xrightarrow{\varphi}D_J\xrightarrow{\operatorname{res}}D_I,$$
and $a/p<a\leq b/p^2<b/p$, so restriction strictly improves both boundaries. In the scalar case, use the Laurent norms and choose open bounded $R_0$-lattices $P_J\subset\Lambda_{R,J,K}$ and $P_I\subset\Lambda_{R,I,K}$. For $n\geq0$, put
$$H_n\left(\sum_{j\in\mathbb Z}c_j\pi_K^j\right)=\sum_{|j|\leq n}c_j\pi_K^j.$$
The strict inequalities at both boundaries give
$$\|\operatorname{res}-H_n\|\leq\sup_{|j|>n}\frac{\|\pi_K^j\|_I}{\|\pi_K^j\|_J}\longrightarrow0.$$
For every $k\geq1$, choose $n=n(k)$ such that $(\operatorname{res}-H_n)(P_J)\subset u^kP_I$. Each coefficient map $c_j:\Lambda_{R,J,K}\to R\otimes_{\mathbb Z_p}\mathcal O_F$ is bounded. Since only finitely many coefficients are retained, there is $h_n\geq0$ such that
$$H_n(P_J)\subset L_k:=\sum_{|j|\leq n}u^{-h_n}(R_0\otimes_{\mathbb Z_p}\mathcal O_F)\pi_K^j.$$
The module $L_k$ is finite over $R_0$, since $\mathcal O_F$ is finite over $\mathbb Z_p$. Consequently, $\operatorname{res}(P_J)\subset L_k+u^kP_I$.

For general $D$, choose a bounded splitting $D_J\xrightarrow{\iota}\Lambda_{R,J,K}^{\oplus d}\xrightarrow{\operatorname{pr}}D_J$. Restricting the projection to $I$ gives a bounded map $\operatorname{pr}_I:\Lambda_{R,I,K}^{\oplus d}\to D_I$, and
$$\operatorname{res}\circ\varphi=\operatorname{pr}_I\circ\operatorname{res}^{\oplus d}\circ\iota\circ\varphi.$$
The scalar approximation, applied componentwise, and Lemma~\ref{integralcompcontoperations} over $R$ show that this map is integrally completely continuous over $R_0$.

Now choose open bounded $A_{\mu(a)}^0$-stable lattices $M_0\subset D_{[a,b]}$ and $N_0\subset D_I$. These are also $R_0$-lattices. The preceding argument gives, for every $k\geq1$, a finite $R_0$-submodule $F_k\subset D_I$ such that $(\operatorname{res}\circ\varphi)(M_0)\subset F_k+u^kN_0$. Put $V_k=A_{\mu(a)}^0F_k$. Any finite set of $R_0$-generators of $F_k$ generates $V_k$ over $A_{\mu(a)}^0$, and
$$ (\operatorname{res}\circ\varphi)(M_0)\subset V_k+u^kN_0.$$
The map is $A_{\mu(a)}$-linear because $\varphi$ and restriction are $R$-linear, commute with $\Gamma_K$, and are continuous for the extended coefficient actions. This proves the assertion.
\end{proof}

\begin{lemma}
\label{modpcompletecontinuity}
Let $D$ be a finite projective $(\varphi,\Gamma_K)$-module. Suppose that the Banach coefficient actions and the finiteness and base-change statements for $D^{\psi=0}$ hold at the radii below. For $0<a\leq b/p^3$, the maps
$$\varphi,\quad 1-\varphi\psi+\varphi:D_{[a,b]}\longrightarrow D_{[a,b/p^2]}$$
are integrally completely continuous maps of $\Lambda_{R,[\mu(a),\infty]}(\Gamma_K)$-modules, and hence are completely continuous in the sense of Definition~\ref{compcontdef}. Restrictions to the target interval are implicit in this notation.
\end{lemma}
\begin{proof}
Put $A=\Lambda_{R,[\mu(a),\infty]}(\Gamma_K)$ and $A_0=\Lambda_{R_0,[\mu(a),\infty]}(\Gamma_K)$. The assertion for $\varphi$ follows from Lemma~\ref{integralresphi}. We now consider $1-\varphi\psi$. Since $\psi\varphi=1$, this map takes values in the kernel of $\psi$ on the common interval where the terms are defined, and factors as
$$D_{[a,b]}\xrightarrow{1-\varphi\psi}D_{[a,b/p]}^{\psi=0}\xrightarrow{\operatorname{res}}D_{[a,b/p^2]}^{\psi=0}\hookrightarrow D_{[a,b/p^2]}.$$
The first and last maps are bounded. It therefore suffices to prove integral complete continuity of the middle restriction.

Write $B=\Lambda_{R,[\mu(a),\mu(b/p)]}(\Gamma_K)$ and $C=\Lambda_{R,[\mu(a),\mu(b/p^2)]}(\Gamma_K)$. We first prove that $B\to C$ is integrally completely continuous over $A_0$. Put $T=\gamma-1$ and let
$$H_n\left(\sum_{j\in\mathbb Z}a_jT^j\right)=\sum_{j\geq-n}a_jT^j.$$
The image is contained in $T^{-n}A$. The source and target have the same boundary corresponding to $\mu(a)$, so the Laurent norms show that $T^nH_n:B\to A$ is bounded. Choose open bounded integral lattices $B_0\subset B$ and $C_0\subset C$. For each fixed $n$, there is therefore $h_n\geq0$ such that
$$H_n(B_0)\subset u^{-h_n}T^{-n}A_0.$$
On the negative powers, the change from $\mu(b/p)$ to $\mu(b/p^2)$ gives a constant $0<c<1$ such that
$$\frac{\|T^{-j}\|_C}{\|T^{-j}\|_B}\leq c^j\qquad(j\geq0).$$
Consequently $\|\operatorname{res}-H_n\|\leq c^{n+1}$. For every $k\geq1$, choose $n=n(k)$ such that $(\operatorname{res}-H_n)(B_0)\subset u^kC_0$. Then
$$\operatorname{res}(B_0)\subset L_k+u^kC_0,\qquad L_k=u^{-h_{n(k)}}T^{-n(k)}A_0.$$
Each $L_k$ is generated by one element over $A_0$. This proves integral complete continuity of $B\to C$. The coefficient bound on $H_n(B_0)$ is explained in the previous lemma as well.

By hypothesis, $M=D_{[a,b/p]}^{\psi=0}$ is a finite Banach $B$-module and restriction identifies
$$M\widehat{\otimes}_BC\cong D_{[a,b/p^2]}^{\psi=0}.$$
Choose a bounded $B$-linear surjection $B^{\oplus r}\to M$. It is strict by the open mapping theorem. The composite $B^{\oplus r}\to C^{\oplus r}\to M\widehat{\otimes}_BC$ is integrally completely continuous, by the scalar result and bounded postcomposition. Since it factors through $M$, the surjective descent assertion of Lemma~\ref{integralcompcontoperations} shows that $M\to M\widehat{\otimes}_BC$ is integrally completely continuous. This is the integral version of the argument in Remark~1.7 of \cite{kedlaya2016finiteness}.

Applying bounded precomposition and postcomposition to the above factorization proves integral complete continuity of $1-\varphi\psi$. Adding the result for $\varphi$ proves the assertion for $1-\varphi\psi+\varphi$.
\end{proof}

\begin{lemma}
\label{algebraicphibasechange}
For all sufficiently small $a$, the natural action map
$$
D_{(0,b]}^{\varphi=1}
\otimes_{\Lambda_{R,(0,\infty]}(\Gamma_K)}
\Lambda_{R,[\mu(a),\infty]}(\Gamma_K)
\longrightarrow
D_{[a,b]}^{\varphi=1}
$$
is an isomorphism.
\end{lemma}

\begin{proof}
By Lemma \ref{phiinv},
$D_{(0,b]}^{\varphi=1}$ is finite over $R$.

As in Lemma \ref{resolutionfiniteRmod}, after working componentwise with
respect to the finite torsion subgroup of $\Gamma_K$, continuity of the
$\Gamma_K$-action gives a monic polynomial
$
F(\gamma-1)\in R[\gamma-1]$
which annihilates $D_{(0,b]}^{\varphi=1}$ and whose coefficients other
than the leading coefficient are topologically nilpotent.

For all sufficiently small $a$, monic division by $F$ is valid in
$\Lambda_{R,[\mu(a),\infty]}(\Gamma_K)$ and gives
\[
\frac{\Lambda_{R,(0,\infty]}(\Gamma_K)}
     {F\Lambda_{R,(0,\infty]}(\Gamma_K)}
\cong
\frac{\Lambda_{R,[\mu(a),\infty]}(\Gamma_K)}
     {F\Lambda_{R,[\mu(a),\infty]}(\Gamma_K)}
\cong
R[\gamma-1]/(F).
\]
Indeed, this is the same distinguished-polynomial division used in
Lemma \ref{resolutionfiniteRmod}; after $\mu(a)$ is sufficiently small,
the leading term of $F$ has bigger norm than the lower terms on $\Lambda_{R, [\mu(a), \infty]}(\Gamma_K)$. Since $F$ annihilates $D_{(0,b]}^{\varphi=1}$, we therefore have
\[
\begin{aligned}
&D_{(0,b]}^{\varphi=1}
\otimes_{\Lambda_{R,(0,\infty]}(\Gamma_K)}
\Lambda_{R,[\mu(a),\infty]}(\Gamma_K)
\\
&\qquad\cong
D_{(0,b]}^{\varphi=1}
\otimes_{R[\gamma-1]/(F)}
\frac{\Lambda_{R,[\mu(a),\infty]}(\Gamma_K)}
     {F\Lambda_{R,[\mu(a),\infty]}(\Gamma_K)}
\\
&\qquad\cong
D_{(0,b]}^{\varphi=1}.
\end{aligned}
\]
Under this identification, the natural map in the statement is simply
the restriction map
$
D_{(0,b]}^{\varphi=1}
\longrightarrow
D_{[a,b]}^{\varphi=1}.
$
This map is injective since
$
D_{(0,b]}\longrightarrow D_{[a,b]}
$
is injective. It is an isomorphism as $D_{(0, b]}^{\varphi = 1} = D_{[a, b]}^{\varphi = 1}$ from Corollary 3.6, \cite{bellovin2024cohomology}.
\end{proof}

\begin{proposition}
\label{firstcoh}
  Let $D$ be a finite projective $(\varphi,\Gamma_K)$-module.
Choose $b>0$ sufficiently small for the Banach coefficient
actions and the annular $\psi=0$ base-change statements.
For sufficiently small admissible $a$ with
$0<a\leq b/p^3$, the module $D_{[a,b]}^{\psi=1}$ is finite
over $\Lambda_{R,[\mu(a),\infty]}(\Gamma_K)$.

\end{proposition}
\begin{proof}
Since $(1-\varphi \psi ): D_{[a, b]} \to D_{[a, b/p^2]}$ and $\varphi: D_{[a, b]} \to D_{[a, b/p^2]}$ are integrally completely continuous from Lemma \ref{integralresphi} and Lemma \ref{modpcompletecontinuity}, so is their sum.

Consider the complexes
$
C_{a,b}^{\bullet}
:=
\left[
D_{[a,b]}
\xrightarrow{\psi-1}
D_{[ap,b]}
\right]
$
and
$
C_{a,b/p^2}^{\bullet}
:=
\left[
D_{[a,b/p^2]}
\xrightarrow{\psi-1}
D_{[ap,b/p^2]}
\right],
$
where the complexes are placed in degrees $1$ and $2$, and in the map
$\psi-1$ the identity denotes the natural restriction map. We consider
the same homotopy between the restriction map and the dotted red map:
\[\begin{tikzcd}
	{D_{[a,b]}} &&& {D_{[ap,b]}} \\
	{D_{[a,b/p^2]}} &&& {D_{[ap,b/p^2]}}
	\arrow["{\psi-1}", from=1-1, to=1-4]
	\arrow["{\text{res}}"', from=1-1, to=2-1]
	\arrow["{1-\varphi\psi+\varphi}", shift left=3, color={rgb,255:red,254;green,52;blue,69}, dashed, from=1-1, to=2-1]
	\arrow["\varphi"{description}, color={rgb,255:red,92;green,92;blue,214}, dashed, from=1-4, to=2-1]
	\arrow["{\text{res}}"', from=1-4, to=2-4]
	\arrow["\varphi", shift left=3, color={rgb,255:red,254;green,52;blue,69}, dashed, from=1-4, to=2-4]
	\arrow["{\psi-1}"', dashed, from=2-1, to=2-4]
\end{tikzcd}\]
The diagonal map $\varphi$ gives a homotopy between the restriction
map of complexes and the dotted red map. From
Lemma~\ref{modpcompletecontinuity}, the dotted red map is degreewise integrally
completely continuous as a morphism of Banach
$\Lambda_{R,[\mu(a),\infty]}(\Gamma_K)$-modules.

The first cohomology of the two complexes is given by
$$
H^1(C_{a,b}^{\bullet})=D_{[a,b]}^{\psi=1}
\qquad\text{and}\qquad
H^1(C_{a,b/p^2}^{\bullet})=D_{[a,b/p^2]}^{\psi=1}.
$$
By Lemma~\ref{basic}, the restriction map
$
D_{[a,b]}^{\psi=1}
\longrightarrow
D_{[a,b/p^2]}^{\psi=1}
$
is an isomorphism. Hence, the map induced by the restriction map on
first cohomology has trivial cokernel. Since the dotted red map is
homotopic to the restriction map, it induces the same map on
cohomology. Therefore
$$
\operatorname{coker}\left(
H^1(C_{a,b}^{\bullet})
\longrightarrow
H^1(C_{a,b/p^2}^{\bullet})
\right)=0
$$
for the map induced by the dotted red map as well.

It now follows from Corollary \ref{integralcompcontcohomology} that
$H^1(C_{a,b}^{\bullet})$ is a finitely generated
$\Lambda_{R,[\mu(a),\infty]}(\Gamma_K)$-module. Thus
$
D_{[a,b]}^{\psi=1}
$
is finitely generated over
$\Lambda_{R,[\mu(a),\infty]}(\Gamma_K)$.

\end{proof}
\begin{remark}
We initially attempted to prove directly that the restriction map
$\operatorname{res}:\Lambda_{R,[a,b],K}\rightarrow\Lambda_{R,[a,b/p],K}$
is completely continuous as a morphism of $\Lambda_{R,[\mu(a),\infty]}(\Gamma_K)$-modules. Although these restriction maps are not completely continuous as morphisms of $R$-modules, the additional $\Gamma_K$-action makes it possible to control the distance of most powers of $T = \pi_{\mathbb{Q}_p}$ from suitable finitely generated $\Lambda_{R,[\mu(a),\infty]}(\Gamma_K)$-submodules. The direct argument, however, was not working for the monomials $T^{p^S-1}$ as $S\rightarrow\infty$.

Modulo $p$, these monomials satisfy
$\psi(T^{p^S-1})=T^{p^{S-1}-1}$.
This suggests that the obstruction produced by them should disappear after passing to the cohomology of the $\psi-1$ complex. Moreover, the idempotent operator
$1-\varphi\psi: D_{[a,b]} \to D_{[a,b]}^{\psi=0}$,
together with the series $G_{\gamma_n}$ constructed in
Lemma \ref{lem:closed-psi-zero-projective}, plays a crucial
role in establishing the projectivity of
$D_{[a,b]}^{\psi=0}$ over the annular distribution algebra. No such idempotent exists that projects to the $\psi = 1$ component. Therefore using this operator felt reasonable to try to produce a chain homotopy. Using this observation, we replace the restriction morphism of Iwasawa complexes by a chain-homotopic morphism involving the operator $1-\varphi\psi+\varphi$. We thereby obtain a degree wise completely continuous morphism of complexes which is homotopic to the restriction morphism. In this sense, the restriction morphism is completely continuous up to homotopy. However, it would be interesting to know the complete continuity of the restriction map as well.
\end{remark}

\begin{remark}
    \label{reducednesshypo}
    From now on, we first assume that $R$ is reduced. Under this assumption
we prove the coadmissibility of $H^1_\psi(D)$ and then the perfectness of
$C^\bullet_\psi(D)$. We subsequently remove the reducedness hypothesis by
filtering by powers of the nilradical and applying the
coefficient-module devissage of
\cite[Lemmas~5.1.3--5.1.5]{kedlaya2014cohomology}. Finally, since a perfect complex over
$\Lambda_{R,(0,\infty]}(\Gamma_K)$ is represented by a bounded complex
of finite projective, hence coadmissible, modules, and the category of
coadmissible modules is abelian, its cohomology modules are
coadmissible. Thus the coadmissibility statement will also hold for
arbitrary $R$.
\end{remark}

\begin{theorem}
\label{redinj}
Let $R$ be reduced and $D$ be a $(\varphi, \Gamma_K)$ module over $\Lambda_{R, \text{rig}, K}$. Then the first Iwasawa cohomology group of $D$ is coadmissible.
\end{theorem}
\begin{proof}
Assume first that $R$ is reduced. Choose a decreasing sequence $a_n\to0$, and put $A_n=\Lambda_{R,[\mu(a_n),\infty]}(\Gamma_K)$. After discarding finitely many terms, Proposition \ref{firstcoh} shows that $D_{[a_n,b]}^{\psi=1}$ is finite over $A_n$ for every $n$. Let
$$
V_n=\operatorname{im}\left(D_{(0,b]}^{\psi=1}\otimes_{\Lambda_{R,(0,\infty]}(\Gamma_K)}A_n\longrightarrow D_{[a_n,b]}^{\psi=1}\right).
$$
Since $A_n$ is Noetherian, $V_n$ is a finite $A_n$-module. We know that $D_{[a_n,b]}^{\varphi=1}$ are precisely the restrictions of $D_{(0,b]}^{\varphi=1}$. Since these restrictions lie in $V_n$, we obtain
$$
0\longrightarrow D_{(0,b]}^{\varphi=1}\longrightarrow V_n\xrightarrow{\varphi-1}(\varphi-1)V_n\longrightarrow0.
$$
Moreover, $(\varphi-1)V_n\subseteq D_{[a_n,b/p]}^{\psi=0}$. By the annular $\psi=0$ structure theorem and flatness of the corresponding annular localization, $D_{[a_n,b/p]}^{\psi=0}$ is torsion-free over $A_n$. Hence $(\varphi-1)V_n$ is a finite torsion-free $A_n$-module. We also have, compatibly with restriction,
$$
D_{(0,b]}^{\varphi=1}\otimes_{A_m}A_n\cong D_{(0,b]}^{\varphi=1}.
$$
Fix $n$. For $m\geq n$, put $V_{m,n}=V_m\otimes_{A_m}A_n$. We first show that restriction induces a surjection $V_{m+1,n}\twoheadrightarrow V_{m,n}$. By definition there is a surjection
$$
D_{(0,b]}^{\psi=1}\otimes_{\Lambda_{R,(0,\infty]}(\Gamma_K)}A_m\twoheadrightarrow V_m.
$$
Tensoring with $A_n$ over $A_m$ gives a surjection
$$
D_{(0,b]}^{\psi=1}\otimes_{\Lambda_{R,(0,\infty]}(\Gamma_K)}A_n\twoheadrightarrow V_{m,n}.
$$
The maps for $m+1$ and $m$ are compatible with annular restriction, so we have a commutative diagram
$$
\begin{array}{ccc}
D_{(0,b]}^{\psi=1}\otimes_{\Lambda_{R,(0,\infty]}(\Gamma_K)}A_n&\twoheadrightarrow&V_{m+1,n}\\
\Vert&&\downarrow\\
D_{(0,b]}^{\psi=1}\otimes_{\Lambda_{R,(0,\infty]}(\Gamma_K)}A_n&\twoheadrightarrow&V_{m,n}.
\end{array}
$$
Since the lower horizontal map is surjective, the right vertical map is surjective.
We claim that these surjections are eventually isomorphisms. Since $A_m\to A_n$ is flat, tensoring
$$
0\longrightarrow D_{(0,b]}^{\varphi=1}\longrightarrow V_m\longrightarrow(\varphi-1)V_m\longrightarrow0
$$
with $A_n$ gives
$$
0\longrightarrow D_{(0,b]}^{\varphi=1}\longrightarrow V_{m,n}\longrightarrow(\varphi-1)V_m\otimes_{A_m}A_n\longrightarrow0.
$$
The transition maps give a commutative diagram with exact rows
$$
\begin{array}{ccccccccc}
0&\longrightarrow&D_{(0,b]}^{\varphi=1}&\longrightarrow&V_{m+1,n}&\longrightarrow&(\varphi-1)V_{m+1}\otimes_{A_{m+1}}A_n&\longrightarrow&0\\
&&\Vert&&\downarrow&&\downarrow\\
0&\longrightarrow&D_{(0,b]}^{\varphi=1}&\longrightarrow&V_{m,n}&\longrightarrow&(\varphi-1)V_m\otimes_{A_m}A_n&\longrightarrow&0.
\end{array}
$$
Consequently,
$$
\ker(V_{m+1,n}\to V_{m,n})\hookrightarrow(\varphi-1)V_{m+1}\otimes_{A_{m+1}}A_n.
$$
The module on the right is torsion-free over $A_n$. Indeed, since $R$ is reduced, $A_{m+1}$ is reduced Noetherian, and every finite torsion-free $A_{m+1}$-module embeds into a finite free $A_{m+1}$-module. Thus $(\varphi-1)V_{m+1}\hookrightarrow A_{m+1}^d$ for some $d$. Since $A_{m+1}\to A_n$ is flat, tensoring gives
$$
(\varphi-1)V_{m+1}\otimes_{A_{m+1}}A_n\hookrightarrow A_n^d.
$$
Hence $\ker(V_{m+1,n}\to V_{m,n})$ is torsion-free over $A_n$.
We now show that this kernel eventually vanishes. Let $\mathfrak p$ be a minimal prime of $A_n$, and choose a maximal ideal $\mathfrak m\supseteq\mathfrak p$. For every $m\geq n$, the preceding surjection gives
$
D_{(0,b]}^{\psi=1}\otimes_{\Lambda_{R,(0,\infty]}(\Gamma_K)}\kappa(\mathfrak m)\twoheadrightarrow V_{m,n}\otimes_{A_n}\kappa(\mathfrak m).
$
The left-hand side is finite-dimensional over $\kappa(\mathfrak m)$ by the base change theorem and the finiteness of the relevant $\operatorname{Tor}_2$ term. Thus there is an integer $d_{\mathbf{\mathfrak m}}$, independent of $m$\footnote{Put $\Lambda=\Lambda_{R,(0,\infty]}(\Gamma_K)$ and $\kappa=\kappa(\mathfrak m)$. The hyper-Tor spectral sequence gives an exact sequence
$
\begin{aligned}
\operatorname{Tor}_2^\Lambda(H^2_\psi(D),\kappa)&\longrightarrow H^1_\psi(D)\otimes_\Lambda\kappa\longrightarrow H^1(C_\psi^\bullet(D)\otimes_\Lambda^{\mathbf L}\kappa)\\
&\longrightarrow\operatorname{Tor}_1^\Lambda(H^2_\psi(D),\kappa)\longrightarrow0.
\end{aligned}
$
By Proposition 5.2 of \cite{bellovin2024cohomology}, $H^2_\psi(D)$ is finite over $R$. By Lemma \ref{resolutionfiniteRmod}, it therefore admits a resolution $P_\bullet\to H^2_\psi(D)$ by finite projective $\Lambda$-modules. In particular,
$$
\dim_\kappa\operatorname{Tor}_2^\Lambda(H^2_\psi(D),\kappa)\leq\dim_\kappa(P_2\otimes_\Lambda\kappa)<\infty.
$$
Apply Theorem~\ref{pointbasechange} to the maximal point
$\mathfrak m$ of $A_n$. With $\kappa=\kappa(\mathfrak m)$ and
$\eta_{\mathfrak m}$ the character induced by evaluation, put
$$
N_{\mathfrak m}
=(D\widehat\otimes_R\kappa)(\eta_{\mathfrak m}^{-1}).
$$
Then
$$
H^1(C_\psi^\bullet(D)\otimes_\Lambda^{\mathbf L}\kappa)
\cong H^1_{\psi,\gamma_K}(N_{\mathfrak m}).
$$
The field-case Euler characteristic formula and Tate duality give
$$
\dim_\kappa H^1_{\psi,\gamma_K}(N_{\mathfrak m})
\leq([K:\mathbb Q_p]+2)d,
$$
where $d$ is an upper bound for the ranks of $D$.
Thus we may take $c_D=([K:\mathbf Q_p]+2)d$.

Consequently,
$$
\dim_\kappa(H^1_\psi(D)\otimes_\Lambda\kappa)\leq c_D+\dim_\kappa\operatorname{Tor}_2^\Lambda(H^2_\psi(D),\kappa)<\infty.
$$
For fixed $\mathfrak m$, the right-hand side is independent of the index $m$. This is the bound needed for the fibres of $V_{m,n}$.}, such that
$$
\dim_{\kappa(\mathfrak m)}(V_{m,n}\otimes_{A_n}\kappa(\mathfrak m))\leq d_{\mathbf{\mathfrak m}}.
$$
By Nakayama's lemma, $(V_{m,n})_{\mathfrak m}$ can therefore be generated by at most $d_{\mathbf{\mathfrak m}}$ elements. Since $A_n$ is reduced, $(A_n)_{\mathfrak p}$ is a field. Localizing further at $\mathfrak p$ gives
$$
\dim_{(A_n)_{\mathfrak p}}(V_{m,n})_{\mathfrak p}\leq d_{\mathbf{\mathfrak m}}.
$$
Since $V_{m+1,n}\twoheadrightarrow V_{m,n}$, the dimensions $\dim_{(A_n)_{\mathfrak p}}(V_{m,n})_{\mathfrak p}$ are nondecreasing in $m$. They are bounded integers, and hence eventually constant.
Since $A_n$ has only finitely many minimal primes, there is an $m_0$ such that, for every $m\geq m_0$ and every minimal prime $\mathfrak p$ of $A_n$,
$$
\dim_{(A_n)_{\mathfrak p}}(V_{m+1,n})_{\mathfrak p}=\dim_{(A_n)_{\mathfrak p}}(V_{m,n})_{\mathfrak p}.
$$
Localizing the exact sequence
$$
0\longrightarrow\ker(V_{m+1,n}\to V_{m,n})\longrightarrow V_{m+1,n}\longrightarrow V_{m,n}\longrightarrow0
$$
at $\mathfrak p$ therefore gives $\ker(V_{m+1,n}\to V_{m,n})_{\mathfrak p}=0$. Thus the kernel vanishes after localization at every minimal prime of $A_n$. Since the kernel is torsion-free over the reduced ring $A_n$, it follows that $\ker(V_{m+1,n}\to V_{m,n})=0$. Hence, for every fixed $n$, the inverse system $V_{m,n}$, $m\geq n$, eventually stabilizes.
Define $M_n=\varprojlim_{m\geq n}V_{m,n}$. Since the system is eventually constant, $M_n$ is naturally isomorphic to its stable value and is therefore finite over $A_n$.
We next verify the transition property. Choose $m$ sufficiently large that both the systems defining $M_n$ and $M_{n+1}$ have stabilized. Then $M_{n+1}\cong V_m\otimes_{A_m}A_{n+1}$ and $M_n\cong V_m\otimes_{A_m}A_n$. Therefore
$$
\begin{aligned}
M_{n+1}\otimes_{A_{n+1}}A_n&\cong(V_m\otimes_{A_m}A_{n+1})\otimes_{A_{n+1}}A_n\\
&\cong V_m\otimes_{A_m}A_n\\
&\cong M_n.
\end{aligned}
$$
Finally, we identify the inverse limit. First,
$$
D_{(0,b]}^{\psi=1}\cong\varprojlim_nV_n.
$$
Indeed, every element of $D_{(0,b]}^{\psi=1}$ restricts to a compatible family in the $V_n$. Conversely, a compatible family in the $V_n$ is in particular a compatible family of elements of $D_{[a_n,b]}^{\psi=1}$. Since $D_{(0,b]}\cong\varprojlim_nD_{[a_n,b]}$, this family glues to an element of $D_{(0,b]}$, and the equations $\psi-\operatorname{res}=0$ on the family of annuli indexed by $n$ imply that the resulting element belongs to $D_{(0,b]}^{\psi=1}$.
It remains to compare $\varprojlim_nM_n$ with $\varprojlim_nV_n$. By definition,
$$
\varprojlim_nM_n=\varprojlim_n\varprojlim_{m\geq n}V_{m,n}.
$$
The compatible restriction maps identify this with the inverse limit over pairs $(m,n)$ with $m\geq n$, ordered componentwise. The diagonal pairs $(k,k)$ form a cofinal subsystem, and $V_{k,k}=V_k$. Therefore
$$
\varprojlim_nM_n\cong\varprojlim_kV_k.
$$
Combining the preceding identifications gives $D_{(0,b]}^{\psi=1}\cong\varprojlim_nM_n$. Each $M_n$ is finite over $A_n$, and $M_{n+1}\otimes_{A_{n+1}}A_n\cong M_n$. Hence $D_{(0,b]}^{\psi=1}$ is coadmissible over $\Lambda_{R,(0,\infty]}(\Gamma_K)$.
\end{proof}
\begin{lemma}
\label{resolutionfiniteRmod}
    Every continuous $\Lambda_{R, (0, \infty]}(\Gamma_K)$-module $M$ that is $R$-finite lies in $D_{\textrm{perf}}^{-}(\Lambda_{R, (0, \infty]}(\Gamma_K))$, i.e. it has a projective resolution in terms of finite projective $\Lambda_{R, (0, \infty]}(\Gamma_K)$-modules.
\end{lemma}
\begin{proof}
Since $M$ is $R$-finite and the action of $\gamma-1$ is continuous, we can choose a finite $R_0$-lattice $M_0$ in $M$ such that, for $N$ sufficiently large,
$
    (\gamma-1)^N M_0 \subseteq uM_0,
$
where $R_0$ is a ring of definition of $R$ and $u$ is a topologically nilpotent unit. By the determinant trick of Lemma 10.16.2 in \cite{stacks-project}, we can find a monic polynomial $P(\gamma-1)$ annihilating $M$ whose non-leading coefficients lie in $uR_0$. Thus $P$ is distinguished.
Now consider the ring
$B = \dfrac{\Lambda_{R, (0, \infty]}(\Gamma_K)}{P(\gamma-1)}.
$
Since $P$ is distinguished, $B$ is finite over $R$ by Weierstrass division. Moreover, multiplication by $P(\gamma-1)$ is injective, so $B$ has a finite free resolution over $\Lambda_{R, (0, \infty]}(\Gamma_K)$ given by
$
\begin{tikzcd}
    0 & {\Lambda_{R, (0, \infty]}(\Gamma_K)} & {\Lambda_{R, (0, \infty]}(\Gamma_K)} & B & 0
    \arrow[from=1-1, to=1-2]
    \arrow["{.P(\gamma-1)}", from=1-2, to=1-3]
    \arrow[from=1-3, to=1-4]
    \arrow[from=1-4, to=1-5]
\end{tikzcd}
$
Now, both $B$ and $M$ are $R$-finite, and hence $M$ is a finite $B$-module. We can therefore find a resolution of $M$ by finite free $B$-modules
$
    \cdots \longrightarrow B^{\oplus n_1}
    \longrightarrow B^{\oplus n_0}
    \longrightarrow M \longrightarrow 0.
$
Since every finite free $B$-module has a finite free resolution over $\Lambda_{R, (0, \infty]}(\Gamma_K)$ of length one, this resolution can be resolved further over $\Lambda_{R, (0, \infty]}(\Gamma_K)$. Taking the total complex gives a bounded above resolution of $M$ by finite free $\Lambda_{R, (0, \infty]}(\Gamma_K)$-modules. This proves our result.
\end{proof}
\begin{lemma}
\label{rankind}
   The first Iwasawa cohomology group $H^1(C_{\psi}^{\bullet})$ is finitely generated.
\end{lemma}
\begin{proof}
Put $\Lambda=\Lambda_{R,(0,\infty]}(\Gamma_K)$ and
$A_a=\Lambda_{R,[\mu(a),\infty]}(\Gamma_K)$.
By the preceding coadmissibility result,
$$
M_a=H^1(C_\psi^\bullet(D))\otimes_\Lambda A_a
$$
is a finite $A_a$-module. By Lemma~\ref{uniformfinite}, it
suffices to bound its number of generators independently of $a$.

Let $z\in\operatorname{Max}(A_a)$, put
$L_z=A_a/\mathfrak n_{z,a}$, and set
$N_z=(D\widehat\otimes_R L_z)(\eta_z^{-1})$.
By Theorem~\ref{pointbasechange}, the hyper-Tor spectral sequence is
$$
E_2^{-i,j}
=\operatorname{Tor}^{A_a}_i
\left(H^j(C_\psi^\bullet(D))\otimes_\Lambda A_a,L_z\right)
\Longrightarrow H^{j-i}_{\psi,\gamma_K}(N_z).
$$
Because $C_\psi^\bullet(D)$ has cohomology only in degrees $1,2$,
it gives the exact sequence
$$
\operatorname{Tor}^{A_a}_2(H^2_\psi(D)\otimes_\Lambda A_a,L_z)
\longrightarrow M_a\otimes_{A_a}L_z
\longrightarrow H^1_{\psi,\gamma_K}(N_z).
$$

By Proposition~5.2 of \cite{bellovin2024cohomology},
$H^2_\psi(D)$ is finite over $R$.
Lemma~\ref{resolutionfiniteRmod} therefore gives a resolution
$P_\bullet\to H^2_\psi(D)$ by finite projective $\Lambda$-modules.
Choose this resolution once and for all, and choose $t_2$ such that
$P_2$ is a direct summand of $\Lambda^{t_2}$.
Flatness of $A_a$ over $\Lambda$ gives
$$
\dim_{L_z}\operatorname{Tor}^{A_a}_2
(H^2_\psi(D)\otimes_\Lambda A_a,L_z)\leq t_2.
$$

The field case of the Euler characteristic formula
\cite[Corollary~3.14]{bellovin2024cohomology} gives
$$
\dim_{L_z}H^1_{\psi,\gamma_K}(N_z)
=[K:\mathbf Q_p]\operatorname{rank}(N_z)
+\dim_{L_z}H^0_{\psi,\gamma_K}(N_z)
+\dim_{L_z}H^2_{\psi,\gamma_K}(N_z).
$$
If $d$ bounds the ranks of $D$, then $\dim_{L_z}H^0\leq d$,
and Tate duality gives $\dim_{L_z}H^2\leq d$. Hence
\begin{equation}
\label{tatedualitybound}
\dim_{L_z}H^1_{\psi,\gamma_K}(N_z)
\leq([K:\mathbf Q_p]+2)d.
\end{equation}
It follows that
$$
\dim_{L_z}(M_a\otimes_{A_a}L_z)
\leq t_2+([K:\mathbf Q_p]+2)d
$$
for every $a$ and every maximal point $z$.
The Krull dimensions of the $A_a$ are uniformly bounded, so
Forster--Swan gives a uniform bound for the number of generators
of $M_a$. Lemma~\ref{uniformfinite} now implies that
$H^1(C_\psi^\bullet(D))$ is finitely generated over $\Lambda$.
Only field-coefficient cohomology results are used in the fibre
calculation; the perfectness theorem for families is not used.
\end{proof}

\subsubsection{Iwasawa cohomology over a field}

\begin{theorem}
    \label{field}
    Let $R$ be a finite extension of either $\mathbb{Q}_p$ or
    $\mathbb{F}_p((u))$, and let $D$ be a $(\varphi,\Gamma_K)$-module over
    $\Lambda_{R,\mathrm{rig},K}$. Then
    $\operatorname{rank}_{\Lambda_{R,(0,\infty]}(\Gamma_K)}
    H^1(C_\psi^\bullet(D))
    =
    [K:\mathbb{Q}_p]\operatorname{rank}(D),$
    and
    $ \operatorname{rank}_{\Lambda_{R,(0,\infty]}(\Gamma_K)}
    H^2(C_\psi^\bullet(D))
    =
    0.$
\end{theorem}

\begin{proof}
Put $\Lambda=\Lambda_{R,(0,\infty]}(\Gamma_K)$ and
$M=H^1_\psi(D)$. By Lemma~\ref{rankind}, $M$ is finitely
generated and coadmissible. The modules $H^2_\psi(D)$ and
$H^2_\psi(D^*)$ are finite over $R$, and hence torsion over
$\Lambda$ on every character component. In particular,
$H^2_\psi(D)$ has rank zero.

Fix a Banach stage $A_a$ and one domain component of it.
This component is a regular one-dimensional affinoid algebra.
The finite torsion modules
$$
H^2_\psi(D)\otimes_\Lambda A_a,
\qquad H^2_\psi(D^*)\otimes_\Lambda A_a,
\qquad (M\otimes_\Lambda A_a)_{\mathrm{tors}}
$$
have finite support on their respective components.
Inversion of characters preserves the Banach stage radius condition.
We may therefore choose a maximal point $z$ on the chosen
component such that $z$ avoids the supports of the first and
third modules, and $z^{-1}$ avoids the support of the second.
Here $z^{-1}$ has the same coefficient map and the inverse character.

Let $L_z=A_a/\mathfrak n_{z,a}$ and
$N_z=(D\widehat\otimes_R L_z)(\eta_z^{-1})$.
Theorem~\ref{pointbasechange} and the hyper-Tor spectral sequence
give
$$
H^2_{\psi,\gamma_K}(N_z)=0,
\qquad
H^1_{\psi,\gamma_K}(N_z)
\cong(M\otimes_\Lambda A_a)\otimes_{A_a}L_z.
$$
For the second assertion, the Tor terms involving $H^2_\psi(D)$
vanish at $z$, and a residue field on this regular
one-dimensional component has projective dimension one.
Applying the same argument at $z^{-1}$ to $D^*$ gives
$H^2_{\psi,\gamma_K}(N_z^*)=0$.
Tate duality therefore gives $H^0_{\psi,\gamma_K}(N_z)=0$.

The choice of $z$ also ensures that the localization of
$M\otimes_\Lambda A_a$ at $z$ is torsion-free over a DVR,
and hence free. Consequently,
$$
\operatorname{rank}_{A_a}(M\otimes_\Lambda A_a)
=\dim_{L_z}H^1_{\psi,\gamma_K}(N_z)
=[K:\mathbb Q_p]\operatorname{rank}(D),
$$
where the last equality is the field Euler characteristic formula.
Compatibility under flat restriction identifies these Banach stage ranks
with the rank over $\Lambda$ (we use the fact that at every Banach stage a component of $A_a$ is a PID and torsion of $M$ over $\Lambda$ remains torsion over $A_a$, after this we can use structure theorem of modules over PID). The assertion holds on every
character component.
\end{proof}
\subsubsection{Returning to Iwasawa complex over general pseudoaffinoid algebras}

\begin{theorem}
    \label{iwasawacomplexisbounded}
    The Iwasawa complex $C_{\psi}^{\bullet}$ lies in $\text{D}^{-}_{\text{perf}}(\Lambda_{R, (0, \infty]}(\Gamma_K))$. 
\end{theorem}
\begin{proof}
Put $\Lambda=\Lambda_{R,(0,\infty]}(\Gamma_K)$,
$A_a=\Lambda_{R,[\mu(a),\infty]}(\Gamma_K)$, and
$M=H^1_\psi(D)$. By Lemma~\ref{rankind}, $M$ is finitely
generated and coadmissible.
The module $H^2_\psi(D)$ is finite over $R$, so
Lemma~\ref{resolutionfiniteRmod} gives a resolution
$P_\bullet\to H^2_\psi(D)$ by finite projective
$\Lambda$-modules. For each $j$, choose $t_j$ such that $P_j$
is a direct summand of $\Lambda^{t_j}$.

For a maximal point $z$ of $A_a$, use its actual residue field
$L_z=A_a/\mathfrak n_{z,a}$ and put
$N_z=(D\widehat\otimes_R L_z)(\eta_z^{-1})$.
By Theorem~\ref{pointbasechange},
$$
\operatorname{Tor}^{A_a}_i
\left(H^j_\psi(D)\otimes_\Lambda A_a,L_z\right)
\Longrightarrow H^{j-i}_{\psi,\gamma_K}(N_z).
$$
The two nonzero cohomology rows give, for $i\geq1$, the exact segment
$$
\operatorname{Tor}^{A_a}_{i+2}
(H^2_\psi(D)\otimes_\Lambda A_a,L_z)
\longrightarrow\operatorname{Tor}^{A_a}_i(M\otimes_\Lambda A_a,L_z)
\longrightarrow H^{1-i}_{\psi,\gamma_K}(N_z).
$$
The last term vanishes for $i\geq2$, and has dimension at most
$d$ for $i=1$, where $d$ bounds the ranks of $D$.
Flatness of $A_a$ over $\Lambda$ and the chosen resolution give
$$
\dim_{L_z}\operatorname{Tor}^{A_a}_i(M\otimes_\Lambda A_a,L_z)
\leq t_{i+2}+\begin{cases}d,&i=1,\\0,&i\geq2.\end{cases}
$$
For each fixed $i$, this bound is independent of $a$ and $z$.

We construct finite free modules $F_n$ and coadmissible modules
$K_n$ inductively, where $F_n$ would give projective resolution of $M$. Put $K_0=M$. Given a finitely generated
coadmissible $K_n$, choose a surjection $F_n\twoheadrightarrow K_n$
with $F_n$ finite free, and let $K_{n+1}$ be its kernel.
The category of coadmissible modules is abelian, so $K_{n+1}$ is
coadmissible. At a Banach stage and a maximal point we have
$$
\operatorname{Tor}^{A_a}_1(K_n\otimes_\Lambda A_a,L_z)
\longrightarrow K_{n+1}\otimes_\Lambda L_z
\longrightarrow F_n\otimes_\Lambda L_z.
$$
Dimension shifting identifies the first term with
$\operatorname{Tor}^{A_a}_{n+1}(M\otimes_\Lambda A_a,L_z)$.
Thus the maximal-ideal fibres of $K_{n+1}\otimes_\Lambda A_a$
have uniformly bounded dimension. Forster--Swan, the uniform
bound on $\dim A_a$, and Lemma~\ref{uniformfinite} imply that
$K_{n+1}$ is finitely generated.

This constructs a resolution of $M$ by finite free
$\Lambda$-modules. Thus both $H^1_\psi(D)$ and $H^2_\psi(D)$
lie in $\mathbf D_{\mathrm{perf}}^{-}(\Lambda)$.
The truncation triangle
$$
H^1_\psi(D)[-1]\longrightarrow C_\psi^\bullet(D)
\longrightarrow H^2_\psi(D)[-2]\xrightarrow{+1}
$$
gives $C_\psi^\bullet(D)\in\mathbf D_{\mathrm{perf}}^{-}(\Lambda)$.
\end{proof}

Note that the above theorem is first proved for the case when $R$ is reduced. After that we use Remark \ref{reducednesshypo} to get the above theorem for the case when $R$ is not necessarily reduced.

\begin{lemma}(Local criteria of flatness)
\label{flatnesslocalcriteria}
Let $R$ be a Noetherian ring and $M$ a finitely generated $R$-module.
Then the following are equivalent:
\begin{enumerate}
\item $M$ is flat over $R$;
\item for every maximal ideal $\mathfrak p \subset R$, one has $\operatorname{Tor}_1^R(M, \kappa(\mathfrak p)) = 0$ where $\kappa(\mathfrak p) = R_{\mathfrak p}/\mathfrak p R_{\mathfrak p}$.
\end{enumerate}
\end{lemma}

\begin{theorem}
\label{perfectiwa}
        Let $D = D_{(0, b]}$ be a $(\varphi, \Gamma_K)$ module over $\Lambda_{R, (0, b], K}$. 
        \begin{enumerate}
           
            \item The Iwasawa complex $C_{\psi}^{\bullet}(D) \in \textbf{D}_{\text{perf}}^{[0, 2]}(\Lambda_{R, (0, \infty]}(\Gamma_K))$.
            \item  The first cohomology group of the Iwasawa complex is finitely generated over $\Lambda_{R, (0, \infty]}(\Gamma_K)$ of rank $[K: \mathbb{Q}_p] \text{rk}(D)$.
            \item (Euler Characteristic formula) We have $\chi_{\Lambda_{R, (0, \infty]}(\Gamma_K)}(C_{\psi}^{\bullet}(D)) = -[K: \mathbb{Q}_p] \text{rk}(D)$. $H^1(C_{\psi}^{\bullet})$ has generic rank $[K:\mathbb{Q}_p] \text{rank}(D)$ and $H^2(C_{\psi}^{\bullet})$ is torsion. 
        \end{enumerate}
        
\end{theorem}
\begin{proof}
    \begin{enumerate}
    \item Put $\Lambda=\Lambda_{R,(0,\infty]}(\Gamma_K)$. By Theorem~\ref{iwasawacomplexisbounded}, choose a
bounded-above complex $N^\bullet$ of finite projective
$\Lambda$-modules representing $C_\psi^\bullet(D)$, with
$N^i=0$ for $i>2$. Put
$$
Q=\operatorname{coker}(N^{-1}\longrightarrow N^0).
$$
Since there is no negative cohomology, the complex is
quasi-isomorphic to $[Q\longrightarrow N^1\longrightarrow N^2]$.
The module $Q$ is coadmissible and finitely presented.

For $A_a=\Lambda_{R,[\mu(a),\infty]}(\Gamma_K)$, put
$Q_a=Q\otimes_\Lambda A_a$.
Exactness of passage to Banach stages gives a representative
$[Q_a\longrightarrow N_a^1\longrightarrow N_a^2]$.
For each maximal point $z$ of $A_a$, let
$L_z=A_a/\mathfrak n_{z,a}$. The triangle with the projective
subcomplex $[N_a^1\longrightarrow N_a^2]$ gives
$$
\operatorname{Tor}^{A_a}_1(Q_a,L_z)
\cong H^{-1}\left(C_\psi^\bullet(D)
\otimes_\Lambda^{\mathbf L}L_z\right).
$$
By Theorem~\ref{pointbasechange}, the right-hand side is
$H^{-1}_{\psi,\gamma_K}(N_z)=0$.
The local criterion for flatness therefore implies that $Q_a$
is finite projective over $A_a$.
These projective modules are compatible under restriction.
At a fixed maximal point belonging to two stages their fibres
are both canonically $Q\otimes_\Lambda L_z$, so their ranks
agree at that point.

To deduce projectivity over $\Lambda$, let $I\subset\Lambda$
be a finitely generated ideal. It is coadmissible.
A finite presentation $\Lambda^s\to\Lambda^r\to Q\to0$ gives
$$
I\otimes_\Lambda Q=\operatorname{coker}(I^s\to I^r),
$$
so $I\otimes_\Lambda Q$ is coadmissible as well.
At each Banach stage, the map $I\otimes_\Lambda Q\to Q$
becomes
$$
(IA_a)\otimes_{A_a}Q_a\longrightarrow Q_a,
$$
which is injective since $Q_a$ is flat.
Its coadmissible kernel has zero restriction to every stage,
and hence is zero. Thus $I\otimes_\Lambda Q\to Q$ is injective
for every finitely generated ideal $I$.
The ideal criterion for flatness implies that $Q$ is flat.
Since $Q$ is finitely presented, it is finite projective.
This proves the asserted perfect amplitude $[0,2]$.

 \item Finite generation is explained in the Lemma \ref{rankind}. It remains to show the rank. This is explained in the next part.
    \item Since the Euler characteristic can be tested at a closed point on each connected component, the first claim follows from Euler characteristic formula for Galois cohomology as proved in \cite[Corollary~3.14]{bellovin2024cohomology}. For the second claim it is enough to prove it by base changing to every point in $\textrm{Max}(R)$ which is dense in $\textrm{Spa}(R, R_0)$. A point in $\textrm{Max}(R)$ can either lie in generic fiber or special fiber, which would then correspond to proving the result when $R$ is a field extension of $\mathbb{Q}_p$ or $\mathbb{F}_p((u))$, which is done in the preceding section, Theorem \ref{field}.

\end{enumerate}
\end{proof}

Keeping in mind Theorem \ref{basechange}, the following Corollary is an immediate consequence.
\begin{corollary}
\label{perf-galois}
    Let $D$ be a $(\varphi,\Gamma_K)$-module over $\Lambda_{R,(0,b],K}$. 
Then the complex $C_{\varphi,\Gamma}^\bullet(D)$ lies in $\textbf{D}_{\text{perf}}^{[0, 2]}(R)$.
\end{corollary}
\begin{proof}

Take the trivial character
$\eta:\Gamma_K\longrightarrow\mathbb Q_p^\times$ in
Theorem~\ref{basechange}. Then
$\kappa_\eta=D_\lambda$ and $S_\eta=R$, so the assertion
follows from Theorem \ref{perfectiwa}.
\end{proof}

We note that the theorem above is already proved in \cite{bellovin2024cohomology}, Corollary 3.11, but our proof is different, though inspired by an idea developed by Kedlaya-Liu \cite{kedlaya2016finiteness} and further used by Bellovin \cite{bellovin2024cohomology}. We use the results on Iwasawa cohomology to prove this assertion.
\begin{theorem}

    Let $D$ be a $(\varphi, \Gamma_K)$ module over $\Lambda_{R, (0, b], K}$. Suppose $R \to R'$ is a map of  pseudoaffinoid algebras.
    
The canonical morphism:
        $$C_{\psi}^{\bullet}(D) \otimes^{\textbf{L}}_{\Lambda_{R, (0, \infty]}(\Gamma_K)}\Lambda_{R', (0, \infty]}(\Gamma_K) \to C_{\psi}^{\bullet}(D \hat{\otimes} R')$$ is an isomorphism.
   
    \end{theorem}

\begin{proof}
    The proof of Theorem 4.4.3(1), \cite{kedlaya2014cohomology} proves such an assertion using perfectness of Iwasawa complex and pure commutative algebra along with the theory of coadmissible modules over quasi-Stein algebras. There is no mixed characteristic fact, or invertibility of $p$ used there. The proof is purely algebraic. Hence our assertion follows similarly.
\end{proof}

\subsection{Cyclotomic deformation}
\begin{definition}
\label{cyclotomicdeformation}
For $n\geq 1$, let
$$
\Lambda_n=\Lambda_{D_\lambda,[1/p^n,\infty]}(\Gamma_K),
\qquad
\Lambda_n^0=\Lambda_{D_\lambda^0,[1/p^n,\infty]}(\Gamma_K).
$$
Let $X_n$ be the associated pseudoaffinoid space, and put
$X=\bigcup_{n\geq 1}X_n$.
Consider the rank one $(\varphi,\Gamma_K)$-module
$$
\mathbf{Dfm}_n^b
=\Lambda_{\Lambda_n,(0,b],K}\mathbf{e},
$$
where $\varphi(\mathbf{e})=\mathbf{e}$ and
$\gamma(\mathbf{e})=[\gamma^{-1}]\mathbf{e}$ for
$\gamma\in\Gamma_K$.
Let $\mathbf{Dfm}_n$ be its base change to
$\Lambda_{\Lambda_n,\mathrm{rig},K}$.
These modules are compatible under restriction in $n$, and define a
rank one $(\varphi,\Gamma_K)$-module $\mathbf{Dfm}$ over the relative
Robba ring of $X$.

If $x:\Lambda_n\to L$ is a continuous point, put
$\eta_x(\gamma)=x([\gamma])$.
Then
$$
\mathbf{Dfm}_n\widehat{\otimes}_{\Lambda_n,x}L
\cong
\Lambda_{L,\mathrm{rig},K}(\eta_x^{-1}),
$$
where the tensor product is formed on closed annuli before passing to
the relative Robba ring.

For $M$ a $(\varphi,\Gamma_K)$-module over
$\Lambda_{R,\mathrm{rig},K}$, let
$$
A_n=\Lambda_{R,[1/p^n,\infty]}(\Gamma_K),
\qquad
A=\Lambda_{R,(0,\infty]}(\Gamma_K).
$$
Consider the compatible family
$$
M_n=
M\widehat{\otimes}_{\Lambda_{D_\lambda,\mathrm{rig},K}}
\mathbf{Dfm}_n
$$
over
$$
X_R=
X\times_{\mathrm{Spa}(D_\lambda,D_\lambda^0)}
\mathrm{Spa}(R,R_0).
$$
We write
$$
\mathbf{Dfm}(M)=\varprojlim_n M_n
$$
for its cyclotomic deformation.
More precisely, on a closed annulus $[a,b]$ its underlying module is
$$
M_{n,[a,b]}=A_n\widehat{\otimes}_R M_{[a,b]},
$$
with diagonal action
$$
\gamma(\lambda\otimes m)
=[\gamma^{-1}]\lambda\otimes\gamma(m).
$$
We put
$$
M_n^b=\varprojlim_{a\to 0}M_{n,[a,b]},
\qquad
M_n=\varinjlim_{b\to 0}M_n^b.
$$
Thus the inverse limit in the annular parameter $a$ is retained in
the definition.
\end{definition}

\subsection{Duality of Iwasawa cohomology}
Let $\Delta_K$ be the torsion subgroup of $\Gamma_K$, and choose a
topological generator $\gamma_K$ of its torsion-free factor.
We assume that $|\Delta_K|$ is invertible in $R$, as is automatic when
$p\neq 2$.
Consider the involution $\iota$ given by
$\iota([\gamma])=[\gamma^{-1}]$.
For a distribution module $N$, we write $N^\iota$ for the same
underlying abelian group with scalar multiplication
$a\cdot_{N^\iota}m=\iota(a)m$.
We use $M^*$ for the Cartier dual $M^\vee(1)$.

\begin{lemma}
\label{cyclotomicgammaexact}
Let $M$ be a $(\varphi,\Gamma_K)$-module over
$\Lambda_{R,\mathrm{rig},K}$, and choose a model over
$\Lambda_{R,(0,b_0],K}$.
For $0<b\leq b_0/p$, put
$$
N^b=\varprojlim_n M_n^b
=\varprojlim_n\varprojlim_{a\to 0}
\left(A_n\widehat{\otimes}_R M_{[a,b]}\right).
$$
Then the operator $\gamma_K-1$ for the diagonal action is injective
on $N^b$. Moreover, there is a natural strict exact sequence
$$
0\longrightarrow (N^b)^{\Delta_K}
\xrightarrow{\gamma_K-1}(N^b)^{\Delta_K}
\xrightarrow{q_b}M_{(0,b]}
\longrightarrow 0,
$$
where $A$ acts on the first two terms through the coefficient factor (i.e. on the factor on tensor product corresponding to $A_n$),
on the last term through the given $\Gamma_K$-action, and
$$
q_b([\delta]\otimes m)=\delta(m).
$$
The sequence is compatible with restriction in $b$ and with $\psi$.
\end{lemma}

\begin{proof}
Fix a closed annulus $[a,b]$.
The continuous distribution action on $M_{[a,b]}$ extends to an action
of $A_n$ for all sufficiently large $n$, where the bound may depend
on $[a,b]$.
Indeed, continuity of the action and the Banach norm on $M_{[a,b]}$
give some $n_0$ and $C>0$ such that
$$
\|\lambda m\|_{[a,b]}
\leq C\|\lambda\|_{n_0}\|m\|_{[a,b]}
\qquad
(\lambda\in A,\ m\in M_{[a,b]}).
$$
Density of $A$ in $A_{n_0}$ gives the extension to $A_{n_0}$, and
restriction $A_n\to A_{n_0}$ gives the extension for $n\geq n_0$.

For such $n$, define on the dense submodule generated by group-like
elements
$$
\Theta_{n,[a,b]}([\delta]\otimes m)
=[\delta]\otimes\delta^{-1}(m).
$$
This extends to a continuous automorphism of $M_{n,[a,b]}$.
To see this, the coproduct and inversion on the coefficient algebra
are continuous:
$$
[\delta]\longmapsto[\delta]\otimes[\delta],
\qquad
[\delta]\longmapsto[\delta^{-1}].
$$
On the torsion-free factor, with $T=[\gamma_K]-1$, these are given by
$$
T\longmapsto T\otimes 1+1\otimes T+T\otimes T,
\qquad
T\longmapsto-\frac{T}{1+T}.
$$
Both preserve each closed weight disk. Applying the continuous
$A_n$-action on the second tensor factor therefore gives the required
extension. Its inverse is given by
$$
\Theta_{n,[a,b]}^{-1}([\delta]\otimes m)
=[\delta]\otimes\delta(m).
$$

These maps commute with restriction wherever they are defined.
For a prescribed level $n$, choose
$m\geq\max\{n,n_0\}$, apply $\Theta_{m,[a,b]}$, and then restrict
to level $n$. Compatibility makes this independent of $m$.
Doing this for every $[a,b]$ defines continuous inverse automorphisms
$\Theta$ and $\Theta^{-1}$ of $N^b$.
In particular, no bound on $n_0$ independent of $a$ is needed.

Let $U_\gamma$ denote the action only through the inverse universal
character. On group-like elements we have
$$
\begin{aligned}
\gamma\Theta([\delta]\otimes m)
&=[\gamma^{-1}\delta]\otimes\gamma\delta^{-1}(m),\\
\Theta U_\gamma([\delta]\otimes m)
&=[\gamma^{-1}\delta]\otimes\delta^{-1}\gamma(m).
\end{aligned}
$$
Since $\Gamma_K$ is abelian, these are equal. Thus
$$
\gamma\Theta=\Theta U_\gamma,
\qquad
\gamma_K-1=\Theta(U_{\gamma_K}-1)\Theta^{-1}.
$$
This also identifies invariants for the two actions of $\Delta_K$.

For the universal action,
$$
U_{\gamma_K}-1
=[\gamma_K^{-1}]-1
=-[\gamma_K^{-1}]T.
$$
Elements at a Banach weight level have unique convergent power-series
expansions in $T$, with coefficients in $R[\Delta_K]$ tensored with
$M_{[a,b]}$. Consequently, multiplication by $T$ is injective, and
division by $T$ on the submodule with zero constant term is continuous.
These operations commute with both inverse limits, so the same
statements hold on $N^b$.

We now identify the kernel of $q_b$ and verify strictness.
Let $\mathrm{aug}:A\to R$ be the augmentation
$[\delta]\mapsto 1$. In particular,
$\mathrm{aug}(T)=0$ for $T=[\gamma_K]-1$.
The formula for $q_b$ gives
$
q_b=(\mathrm{aug}\otimes 1)\circ\Theta^{-1}.
$

Put
$$
e_{\Delta_K}=\frac{1}{|\Delta_K|}
\sum_{\delta\in\Delta_K}[\delta].
$$
The invariants for the universal $\Delta_K$-action are
$e_{\Delta_K}N^b$, and
$
\Theta:e_{\Delta_K}N^b
\xrightarrow{\sim}(N^b)^{\Delta_K}
$
is a continuous isomorphism with continuous inverse.
On $e_{\Delta_K}N^b$, the map $\mathrm{aug}\otimes 1$
takes the constant coefficient in $T$, under the identification
$e_{\Delta_K}\otimes M_{(0,b]}\cong M_{(0,b]}$.
Consequently,
$
\ker\left(
(\mathrm{aug}\otimes 1)|_{e_{\Delta_K}N^b}
\right)
=Te_{\Delta_K}N^b.
$
Since
$
U_{\gamma_K}-1=-[\gamma_K^{-1}]T
$
and $[\gamma_K^{-1}]$ is a unit, we obtain
$$
\begin{aligned}
\ker(q_b)
&=\Theta(Te_{\Delta_K}N^b)\\
&=\Theta\left((U_{\gamma_K}-1)e_{\Delta_K}N^b\right)\\
&=(\gamma_K-1)(N^b)^{\Delta_K}.
\end{aligned}
$$

To prove surjectivity and strictness of $q_b$, define
$$
s_b:M_{(0,b]}\longrightarrow (N^b)^{\Delta_K},
\qquad
s_b(m)=\Theta(e_{\Delta_K}\otimes m)
=\frac{1}{|\Delta_K|}
\sum_{\delta\in\Delta_K}
[\delta]\otimes\delta^{-1}(m).
$$
This map is continuous and satisfies $q_b\circ s_b=1$.
Thus $q_b$ is surjective, and the induced map
$$
\overline{q}_b:
\frac{(N^b)^{\Delta_K}}
     {(\gamma_K-1)(N^b)^{\Delta_K}}
\longrightarrow M_{(0,b]},
\qquad
[x]\longmapsto q_b(x),
$$
is a continuous bijection. Its inverse is the continuous map
$$
m\longmapsto [s_b(m)],
$$
where the source of $\overline{q}_b$ has the quotient topology.
Hence $\overline{q}_b$ is a homeomorphism, proving that $q_b$
is strict.

Finally, $\gamma_K-1$ is injective by the injectivity of
multiplication by $T$ and the identity
$
\gamma_K-1=\Theta(U_{\gamma_K}-1)\Theta^{-1}.
$
The resulting bijection
$$
\gamma_K-1:(N^b)^{\Delta_K}\longrightarrow\ker(q_b)
$$
has continuous inverse
$$
y\longmapsto
-\Theta\left(
\frac{[\gamma_K]\Theta^{-1}(y)}{T}
\right).
$$
Here $\Theta^{-1}(y)\in Te_{\Delta_K}N^b$, and division by $T$
on this submodule is continuous, as shown above.
Thus $\gamma_K-1$ is also strict. Moreover, its image
$\ker(q_b)$ is closed because $q_b$ is continuous and
$M_{(0,b]}$ is Hausdorff.

We therefore obtain the strict exact sequence
$$
0\longrightarrow (N^b)^{\Delta_K}
\xrightarrow{\gamma_K-1}(N^b)^{\Delta_K}
\xrightarrow{q_b}M_{(0,b]}
\longrightarrow 0.
$$

The formula for $q_b$ shows that it is $A$-linear for the stated
actions. Finally, $\psi$ commutes with $\Gamma_K$, so this formula
also proves compatibility with $\psi$. Compatibility with restriction
in $b$ follows from the construction.
\end{proof}

\begin{theorem}
\label{cyclotomicdeform}
Let $M$ be a $(\varphi,\Gamma_K)$-module over
$\Lambda_{R,\mathrm{rig},K}$.
Then there is a natural quasi-isomorphism of complexes of $A$-modules
$$
C^\bullet_{\psi,\gamma_K}(\mathbf{Dfm}(M))
\;\tilde{\longrightarrow}\;
C^\bullet_\psi(M),
$$
where the action on the left comes from the coefficient algebras
$A_n$, and the action on the right comes from the
$(\varphi,\Gamma_K)$-action on $M$.
Consequently, there is a natural Iwasawa duality quasi-isomorphism
$$
C^\bullet_\psi(M^*)
\;\tilde{\longrightarrow}\;
\mathbf{R}\!\operatorname{Hom}_A
\left(C^\bullet_\psi(M),A\right)^\iota[-2].
$$
\end{theorem}

\begin{proof}
Choose a model of $M$ over $\Lambda_{R,(0,b_0],K}$ and fix
$0<b\leq b_0/p$.
We first explain why this one radius works for every $n$.
For $0<c\leq b$, the operator $\psi$ sends $M_n^c$ into $M_n^{pc}$.
Therefore its induced action on $M_n/M_n^b$ is locally nilpotent:
every element is represented in some $M_n^c$, and sufficiently many
applications of $\psi$ bring it into $M_n^b$.
It follows that $\psi-1$ is invertible on this quotient, with inverse
$$
-(1+\psi+\psi^2+\cdots),
$$
where the sum is finite on each element of the quotient.
Thus
$$
C^\bullet_\psi(M_n^b)
\;\tilde{\longrightarrow}\;
C^\bullet_\psi(M_n),
$$
and, since these maps commute with $\Gamma_K$,
$$
C^\bullet_{\psi,\gamma_K}(M_n^b)
\;\tilde{\longrightarrow}\;
C^\bullet_{\psi,\gamma_K}(M_n).
$$
The same argument gives
$C^\bullet_\psi(M_{(0,b]})\simeq C^\bullet_\psi(M)$.

At this fixed $b$, all the modules are inverse limits of Banach
modules on closed annuli. Restriction in the weight variable and in
the annular variable has dense image: this follows from density for
the coefficient rings and a finite projective presentation of the
annular model of $M$.
After choosing a countable sequence $a_j\to 0$, the topological
Mittag--Leffler lemma applied first in $a_j$ and then to the cofinal
diagonal system in $(n,j)$ shows that these inverse limits have no
higher derived inverse limits.
Taking $\Delta_K$-invariants preserves this property, since it is
the image of the continuous averaging idempotent.
Consequently, with $N^b$ as in Lemma \ref{cyclotomicgammaexact},
$$
\mathbf{R}\!\varprojlim_n
C^\bullet_{\psi,\gamma_K}(M_n^b)
\simeq C^\bullet_{\psi,\gamma_K}(N^b).
$$
Moreover, the vanishing of $\mathbf{R}^1\!\varprojlim_n M_n^b$
shows that taking inverse limits in
$$
0\longrightarrow M_n^b\longrightarrow M_n
\longrightarrow M_n/M_n^b\longrightarrow 0
$$
gives an exact sequence
$$
0\longrightarrow N^b\longrightarrow\mathbf{Dfm}(M)
\longrightarrow\varprojlim_n(M_n/M_n^b)\longrightarrow 0.
$$
The inverses of $\psi-1$ on $M_n/M_n^b$ commute with restriction
in $n$, so $\psi-1$ is also invertible on their inverse limit.
It follows that $N^b\to\mathbf{Dfm}(M)$ induces a quasi-isomorphism
on Herr complexes. Together with the finite-level radius comparison,
this gives
$$
\begin{aligned}
C^\bullet_{\psi,\gamma_K}(\mathbf{Dfm}(M))
&\simeq
\mathbf{R}\!\varprojlim_n
C^\bullet_{\psi,\gamma_K}(M_n)\\
&\simeq
C^\bullet_{\psi,\gamma_K}(N^b).
\end{aligned}
$$
In particular, no interchange of an inverse limit in $n$ with an
inductive limit in $b$ is being used.

By Lemma \ref{cyclotomicgammaexact},
$$
\left[(N^b)^{\Delta_K}
\xrightarrow{\gamma_K-1}(N^b)^{\Delta_K}\right]
\;\tilde{\longrightarrow}\;
M_{(0,b]}[-1].
$$
This quasi-isomorphism commutes with $\psi$.
Taking the total complexes after applying $\psi-1$ gives
$$
C^\bullet_{\psi,\gamma_K}(N^b)
\;\tilde{\longrightarrow}\;
C^\bullet_\psi(M_{(0,b]}).
$$
For the convention in which the Herr differentials are
$(\psi-1,\gamma_K-1)$ and $(\gamma_K-1,1-\psi)$, the map in degree
$1$ is $(x,y)\mapsto q_b(y)$ and the map in degree $2$ is
$z\mapsto-q_b(z)$.
The exact sequence of the lemma shows that the kernel total complex
is acyclic. Combining the preceding quasi-isomorphisms proves
$
C^\bullet_{\psi,\gamma_K}(\mathbf{Dfm}(M))
\;\tilde{\longrightarrow}\;
C^\bullet_\psi(M).
$
The $A$-linearity follows from the $A$-linearity of $q_b$.

We next obtain the finite-level comparison needed for duality.
Put
$
K_n^\bullet=C^\bullet_{\psi,\gamma_K}(M_n).
$
Finiteness and coefficient base change for Herr cohomology over
pseudoaffinoid algebras give
$$
K_n^\bullet\in D_{\mathrm{perf}}^{[0,2]}(A_n),
\qquad
A_n\otimes_{A_{n+1}}^{\mathbf L}K_{n+1}^\bullet
\simeq K_n^\bullet.
$$
Here the coefficient base change is $A_{n+1}\to A_n$ for the
$(\varphi,\Gamma_K)$-module $M_{n+1}$.
Since this restriction is flat, the finite modules
$H^i(K_n^\bullet)$ form a coadmissible system.
For the Fr\'echet--Stein presentation $A=\varprojlim_n A_n$,
the recovery and exactness properties for coadmissible modules
therefore give
$$
\begin{aligned}
H^i\!\left(\mathbf{R}\!\varprojlim_m K_m^\bullet\right)
&\cong\varprojlim_m H^i(K_m^\bullet),\\
A_n\otimes_A\varprojlim_m H^i(K_m^\bullet)
&\cong H^i(K_n^\bullet).
\end{aligned}
$$

By the perfectness of Iwasawa cohomology proved above, choose a
bounded complex $P^\bullet$ of finite projective $A$-modules
representing $C^\bullet_\psi(M)$.
The first part of the proof identifies $P^\bullet$ with
$\mathbf{R}\!\varprojlim_m K_m^\bullet$.
Projection to $K_n^\bullet$ and derived tensor--Hom adjunction give
a natural morphism
$
A_n\otimes_A P^\bullet\longrightarrow K_n^\bullet.
$
Every term of $P^\bullet$ is coadmissible, and kernels, images and
cokernels of morphisms of coadmissible modules are coadmissible.
Exactness of localization on this category shows that
$
H^i(A_n\otimes_A P^\bullet)
\cong A_n\otimes_A H^i(P^\bullet)
\cong H^i(K_n^\bullet).
$
Thus the preceding morphism is a quasi-isomorphism, and we have
proved
$
A_n\otimes_A^{\mathbf L}C^\bullet_\psi(M)
\simeq C^\bullet_{\psi,\gamma_K}(M_n).
$
This argument uses coadmissibility of $H^i(P^{\bullet})$.

We now prove duality. Tate duality for the Herr complex over $A_n$
and the identity
$$
\left(
M\widehat{\otimes}_{\Lambda_{D_\lambda,\mathrm{rig},K}}
\mathbf{Dfm}_n
\right)^*
\cong
\left(
M^*\widehat{\otimes}_{\Lambda_{D_\lambda,\mathrm{rig},K}}
\mathbf{Dfm}_n
\right)^\iota
$$
give compatible quasi-isomorphisms
$$
C^\bullet_{\psi,\gamma_K}
\left(
M^*\widehat{\otimes}_{\Lambda_{D_\lambda,\mathrm{rig},K}}
\mathbf{Dfm}_n
\right)
\simeq
\mathbf{R}\!\operatorname{Hom}_{A_n}
(K_n^\bullet,A_n)^\iota[-2].
$$
Here the involution changes the inverse universal character into
the universal character. The $\varphi$-Herr and $\psi$-Herr complexes
are identified by their natural quasi-isomorphism when applying
Tate duality.

By the finite-level comparison just proved,
$K_n^\bullet\simeq A_n\otimes_A P^\bullet$.
Since $P^\bullet$ is bounded and finite projective, duality commutes
with this base change:
$
\mathbf{R}\!\operatorname{Hom}_{A_n}
(A_n\otimes_A P^\bullet,A_n)
\simeq
A_n\otimes_A\operatorname{Hom}_A(P^\bullet,A).
$
Every term of $\operatorname{Hom}_A(P^\bullet,A)$ is again finite
projective and is recovered from its Banach realizations, with
vanishing higher derived inverse limits.
Taking derived inverse limits in the duality quasi-isomorphisms
therefore gives
$$
C^\bullet_{\psi,\gamma_K}(\mathbf{Dfm}(M^*))
\simeq
\mathbf{R}\!\operatorname{Hom}_A
\left(C^\bullet_\psi(M),A\right)^\iota[-2].
$$
Applying the first part of the theorem to $M^*$ proves the desired
quasi-isomorphism
$$
C^\bullet_\psi(M^*)
\;\tilde{\longrightarrow}\;
\mathbf{R}\!\operatorname{Hom}_{\Lambda_{R,(0,\infty]}(\Gamma_K)}
\left(
C^\bullet_\psi(M),
\Lambda_{R,(0,\infty]}(\Gamma_K)
\right)^\iota[-2].
$$
\end{proof}

\section{An application to Selmer complexes in pseudorigid families}
In \cite{nekovar2006selmer}, Nekovar explains the importance of Selmer complexes for base change and duality. For example, \cite{buyukboduk2025functional} uses determinants of Selmer complexes to establish functional equations for algebraic $p$-adic $L$-functions. Perfectness is an important input in such arguments. We explain how our local results give perfect Selmer complexes over pseudoaffinoid algebras, under the integral hypothesis below, and then study the loci where their first cohomology jumps. We hope to understand applications to eigenvarieties in subsequent work.

Let $S$ be a finite set of finite primes of $\mathbf Q$
containing $p$. We write $G_{\mathbf Q,S}$ for the
Galois group of the maximal extension unramified outside
$S$ and the archimedean place. The local-condition sums
below run over the finite primes in $S$.
Throughout this section, let $R$ be a pseudoaffinoid algebra, and put $X=\mathrm{Spa}(R,R^+)$. Let $V$ be a finite projective $R$-module admitting an action of $G_{\mathbb{Q}, S}$.
Choose a finite rational covering
$$
X=\bigcup_{i=1}^r X_i,\qquad X_i=\mathrm{Spa}(R_i,R_i^+),
$$
on which $V_i=V\otimes_RR_i$ is free. By Lemma~4.13 of
\cite{bellovin2023galois}, after refining this covering by the
affine charts of admissible formal blow-ups, we may choose
Noetherian $u$-adically complete rings of definition $R_{i,0}$
and $G_{\mathbb Q,S}$-stable finite projective lattices
$$
V_{i,0}\subset V_i,\qquad V_{i,0}[1/u]=V_i.
$$
Here $u$ denotes its image in each coefficient algebra.
The coverings and lattices are chosen on the coefficient space
and will be kept fixed as the weight radius varies.

\begin{definition}
\label{nekcomplex}
Let $A$ be a coefficient algebra and $M$ a continuous $A$-linear representation of $G_{\mathbb{Q},S}$. A bounded local condition at $v\in S$ is a complex $U_v^\bullet$ with bounded cohomology and a morphism $i_v:U_v^\bullet\to C^\bullet_{\mathrm{cont}}(G_{\mathbb{Q}_v},M)$. For $\Delta=(U_v^\bullet,i_v)_{v\in S}$, define
$$C_f^\bullet(G_{\mathbb{Q},S},M;\Delta)=\mathrm{Cone}\left(C^\bullet_{\mathrm{cont}}(G_{\mathbb{Q},S},M)\oplus\bigoplus_{v\in S}U_v^\bullet\xrightarrow{\oplus_v(\mathrm{res}_v-i_v)}\bigoplus_{v\in S}C^\bullet_{\mathrm{cont}}(G_{\mathbb{Q}_v},M)\right)[-1].$$
The Iwasawa Selmer complex is obtained by applying this construction to the cyclotomic deformation and its corresponding local conditions.
\end{definition}

We first give the integral argument needed to apply this construction over $R$ and over its distribution algebra. Write $\Gamma=\mathrm{Gal}(\mathbb{Q}(\mu_{p^\infty})/\mathbb{Q})=\Delta\times\Gamma_1$, choose a generator $\gamma$ of $\Gamma_1\simeq\mathbb{Z}_p$, and put $T=[\gamma]-1$ and $\Lambda_R=\Lambda_{R,(0,\infty]}(\Gamma)$. The action on $V_{\mathrm{Iw}}=V\widehat\otimes_R\Lambda_R$ is
$$g(v\otimes\lambda)=g(v)\otimes[\overline g]^{-1}\lambda,$$
where $\overline g$ is the image of $g$ in $\Gamma$. We use the same inverse universal character in $\mathbf{Dfm}(D)$.

\begin{lemma}
\label{selmergaloisperfect}
The global and local Galois cohomology complexes of $V$ are
perfect over $R$ with amplitude $[0,2]$. Their cyclotomic
deformations are perfect over $\Lambda_R$ with amplitude $[0,2]$.
These statements are compatible with derived coefficient base change.
Moreover, $R\Gamma(I_l,V)$ is perfect over $R$ with amplitude
$[0,1]$ for $l\neq p$ and satisfies the same base-change property.
\end{lemma}
\begin{proof}
Let $G$ be $G_{\mathbb Q,S}$, $G_{\mathbb Q_v}$, or $I_l$
with $l\neq p$. Put $e=2$ in the first two cases and $e=1$
in the inertia case. These groups have $p$-cohomological
dimension at most $e$ and finite cohomology on finite discrete
$p$-primary modules. We use the fixed coefficient covering
and projective lattices $V_{i,0}$ chosen above.

Choose a finite rational covering
$$
X=\operatorname{Spa}(R,R^+)=\bigcup_{i=1}^r X_i,
\qquad
X_i=\operatorname{Spa}(R_i,R_i^+),
$$
on which $V_i=V\otimes_R R_i$ admits a $G$-stable finite
projective lattice $V_{i,0}$ over a suitable ring of definition
$R_{i,0}$. For each $i$, the pair $(R_{i,0},V_{i,0})$ satisfies
Hypotheses A of \cite{pottharst2013analytic}: the ring is
Noetherian and complete for $(u)$, a power of $p$ belongs
to $(u)$, the action is continuous, and $V_{i,0}$ is finite
projective, hence flat. Corollary~1.2 and Theorem~1.4 of that
paper provide 
$
K_i^\bullet\simeq R\Gamma(G,V_i),
$
where $K_i^\bullet$ consists of finite projective $R_i$-modules
in degrees $[0,e]$.
For each nonempty subset $J\subseteq\{1,\ldots,r\}$, put
$$
X_J=\bigcap_{i\in J}X_i,
\qquad
R_J=\mathcal O_X(X_J),
\qquad
V_J=V\otimes_R R_J,
$$
Derived coefficient base change
gives, for $i\in J$,
$
K_i^\bullet\otimes_{R_i}^{\mathbf L}R_J
\simeq R\Gamma(G,V_J).
$
In particular, on a pairwise intersection we have canonical
equivalences
$
K_i^\bullet\otimes_{R_i}^{\mathbf L}R_{ij}
\simeq R\Gamma(G,V_{ij})
\simeq K_j^\bullet\otimes_{R_j}^{\mathbf L}R_{ij}.
$
These equivalences, together with their higher compatibilities
are induced by restriction of continuous cochains. Thus the
$K_i^\bullet$ form descent data for perfect complexes of
Tor-amplitude $[0,e]$.

By analytic descent for perfect complexes of prescribed
Tor-amplitude
\cite[Theorem~1.4]{andreychev2021pseudocoherent}, there exists
$
K^\bullet\in\mathbf D_{\mathrm{perf}}^{[0,e]}(R)
$
whose restrictions, with their descent identifications, are
the complexes $K_i^\bullet$. It remains to identify
$K^\bullet$ with $R\Gamma(G,V)$.

Tate acyclicity gives a strictly exact augmented Cech complex
$
0\longrightarrow V
\longrightarrow\prod_i V_i
\longrightarrow\prod_{i<j}V_{ij}
\longrightarrow\cdots.
$
For a profinite set $S$, the functor
$\operatorname{Map}_{\mathrm{cont}}(S,-)$ preserves strictly
exact sequences of complete Banach modules.\footnote{ Strictness
provides uniformly bounded lifts under a surjection.
Let $q:M\to N$ be a strict surjection of complete Banach
modules, and let $f:S\to N$ be continuous. Choose locally
constant functions $f_n:S\to N$ converging uniformly to $f$,
and put $f_0=0$. Strictness gives a constant $C>0$ such that
every $y\in N$ admits a lift $x\in M$ satisfying
$\|x\|\leq C\|y\|$.

For each $n\geq1$, the locally constant function
$f_n-f_{n-1}$ has finite image. Choosing a lift of each of
its values gives a locally constant function $g_n:S\to M$
such that
$$
q\circ g_n=f_n-f_{n-1},
\qquad
\|g_n\|_{\sup}\leq C\|f_n-f_{n-1}\|_{\sup}.
$$
Since $\|f_n-f_{n-1}\|_{\sup}\to0$, completeness and the
nonarchimedean triangle inequality imply that
$
g=\sum_{n=1}^{\infty}g_n
$
converges uniformly. Hence $g$ is continuous, and
$
q\circ g
=\lim_{N\to\infty}\sum_{n=1}^N(f_n-f_{n-1})
=\lim_{N\to\infty}f_N
=f.
$
Thus every continuous function $S\to N$ has a continuous
lift to $M$.}

Applying this with $S=G^q$, for every $q\geq0$, we have Cech complex
$$
0\longrightarrow C^q_{\mathrm{cont}}(G,V)
\longrightarrow\prod_i C^q_{\mathrm{cont}}(G,V_i)
\longrightarrow\prod_{i<j}C^q_{\mathrm{cont}}(G,V_{ij})
\longrightarrow\cdots.
$$
The restriction maps commute with the Galois differentials.
Since there are finitely many non zero terms in Cech complex, totalization therefore gives
$$
R\Gamma(G,V)\simeq
\operatorname{Tot}\left(
\prod_i R\Gamma(G,V_i)
\longrightarrow
\prod_{i<j}R\Gamma(G,V_{ij})
\longrightarrow\cdots
\right).
$$
On the other hand, $K^\bullet$ admits a bounded representative
by finite projective $R$-modules. Applying Tate acyclicity
termwise to this representative gives
$
K^\bullet\simeq
\operatorname{Tot}\left(
\prod_i(K^\bullet\otimes_R^{\mathbf L}R_i)
\longrightarrow
\prod_{i<j}(K^\bullet\otimes_R^{\mathbf L}R_{ij})
\longrightarrow\cdots
\right).
$
The descent identifications identify this total complex with
the preceding one. Hence
$
K^\bullet\simeq R\Gamma(G,V),
$
and consequently
$
R\Gamma(G,V)\in\mathbf D_{\mathrm{perf}}^{[0,e]}(R).
$
The same argument proves derived coefficient base change:
after pulling back the rational covering to a pseudoaffinoid
coefficient algebra $R'$, it is the local base-change
isomorphism given by Pottharst, and descent identifies
these local isomorphisms with
$$
R\Gamma(G,V)\otimes_R^{\mathbf L}R'
\simeq R\Gamma(G,V\otimes_RR').
$$
To apply the integral statement, choose compatible rings of
definition on each pulled-back chart; pullback preserves
the projectivity of $V_{i,0}$.

We next treat the full distribution algebra.
Write $\Gamma=\Delta\times\Gamma_1$, put $T=[\gamma]-1$, and set
$$
B_{i,0}=R_{i,0}[\Delta][[T]],\qquad
\mathbb V_{i,0}=V_{i,0}\otimes_{R_{i,0}}B_{i,0}.
$$
Give $\mathbb V_{i,0}$ the inverse universal character action
$
g(v\otimes b)=g(v)\otimes[\overline g]^{-1}b.
$
The ring $B_{i,0}$ is Noetherian and complete for $(u,T)$.
The module $\mathbb V_{i,0}$ is finite projective, and its
$G$-action is continuous for that topology.
Thus Pottharst's results give one complex
$
Q_{G,i,0}^\bullet\simeq R\Gamma(G,\mathbb V_{i,0})
$
of finite projective $B_{i,0}$-modules in degrees $[0,e]$.

Choose integers $m_n\to\infty$ for each $n$ and Banach stage integral models
$$
A_{n,0}
=R_0[\Delta]\langle T,Y\rangle/(T^{m_n}-uY),
\qquad A_n=A_{n,0}[1/u].
$$
We have
$\Lambda_R=\varprojlim_nA_n$.
On $X_i$, use the corresponding rings $A_{n,i,0}$ and
$A_{n,i}=A_{n,i,0}[1/u]$ with $R_{i,0}$ in place of $R_0$.
These are the rational pullbacks of the coefficient covering
to the $n$th weight stage. Since $T^{m_n}\in uA_{n,i,0}$,
the $(u,T)$-adic and $u$-adic topologies on $A_{n,i,0}$
are equivalent. Integral base change and inversion of $u$ give
$
Q_{G,i,0}^\bullet\otimes_{B_{i,0}}A_{n,i}
\simeq R\Gamma(G,V_i\otimes_{R_i}A_{n,i}),
$
where the representation on the right has the same inverse
universal character action. The module $R\Gamma(G,V_i\otimes_{R_i}A_{n,i})$ is also perfect as $V_{i, 0}$ is a projective module over $R_{i, 0}$.

Using analytic descent at every stage $n$ from \cite{andreychev2021pseudocoherent}, we get compatible complexes
$$
C_n=R\Gamma(G,V\otimes_RA_n)
\in\mathbf D_{\mathrm{perf}}^{[0,e]}(A_n),
\qquad
C_{n+1}\otimes_{A_{n+1}}^{\mathbf L}A_n\simeq C_n,
$$
whose restrictions to $A_{n,i}$ are given by $R\Gamma(G, V_i \otimes_{R_i} A_{n, i})$ (which is quasi-isomorphic to a perfect complex as above).

For each $i$ and $j$, choose $t_{i,j}$ such that
$Q_{G,i,0}^j$ is a direct summand of $B_{i,0}^{t_{i,j}}$.
Let $z$ be a maximal point of $A_n$ and put $L_z=\kappa(z)$.
Its coefficient point belongs to some $X_i$.
Evaluation at $z$ then factors through $A_{n,i}$, and the
local model gives
$
C_n\otimes_{A_n}^{\mathbf L}L_z
\simeq Q_{G,i,0}^\bullet\otimes_{B_{i,0}}L_z.
$
Consequently,
$
\dim_{L_z}H^j(C_n\otimes_{A_n}^{\mathbf L}L_z)
\leq\max_i t_{i,j}.
$
This bound is independent of $n$ and of $z$ because the
coefficient covering is finite and the complexes $Q_{G,i,0}^\bullet$
are fixed. Lemma~\ref{uniformperfectcomplex} therefore gives
a single finite projective complex in degrees $[0,e]$ representing
$\mathbf R\varprojlim_nC_n$ over $\Lambda_R$.

It remains to identify this limit with the stated cohomology.
The modules $V\otimes_RA_n$ have dense transition maps.
Locally constant functions on the compact profinite spaces
$G^j$ are dense in continuous functions, so the transition
maps on continuous cochains are dense as well.
The topological Mittag--Leffler lemma makes the higher
inverse-limit terms vanish. Continuous maps into the inverse
limit are the compatible systems of continuous maps into
its stages. Hence
$
R\Gamma(G,V_{\mathrm{Iw}})
\simeq
\mathbf R\varprojlim_n R\Gamma(G,V\otimes_RA_n)
\in\mathbf D_{\mathrm{perf}}^{[0,e]}(\Lambda_R).
$

Finally, let $R\to R'$ be a continuous coefficient map.
As above, pull back the coefficient covering and choose compatible
Banach stages. The local integral base-change theorem and descent
give a quasi-isomorphism at each stage. Since the full distribution
complexes have finite projective representatives, the derived
tensor product is computed termwise and restricts to these
stagewise comparisons. The cone of the full comparison map
has coadmissible cohomology, which vanishes at every target stage
and hence vanishes globally. Thus
$$
R\Gamma(G,V_{\mathrm{Iw}})
\otimes_{\Lambda_R}^{\mathbf L}\Lambda_{R'}
\simeq
R\Gamma(G,(V\otimes_RR')_{\mathrm{Iw}}).
$$
This proves all assertions, with $e=1$ for inertia and $e=2$
for the other groups.
\end{proof}

\begin{proposition}
\label{selmerconeperfect}
Let $A$ be a ring. Suppose that the global and local Galois cohomology complexes in Definition \ref{nekcomplex}, and all the local conditions $U_v^\bullet$, belong to $\mathbf{D}_{\mathrm{perf}}^{[0,2]}(A)$. The resulting Selmer complex is then perfect over $A$ with amplitude $[0,3]$.
\end{proposition}
\begin{proof}
Choose finite projective representatives in degrees $[0,2]$ for the source and target of the morphism defining the Selmer complex, and represent this morphism by a chain map. Its cone shifted by $[-1]$ has finite projective terms in degrees $[0,3]$.
\end{proof}

\begin{remark}
This is the cone argument underlying Theorem 1.11 of \cite{pottharst2013analytic}. Its application here uses Lemma \ref{selmergaloisperfect}. In particular, the passage from $R_0$ to $R_0[1/u]$ and the finite projective model over $\Lambda_R$ are part of the argument; neither follows simply by calling a pseudoaffinoid algebra an affinoid algebra.
\end{remark}

\subsection{Hodge-Tate weights for trianguline $(\varphi,\Gamma)$-modules}
Let $L$ be a finite extension of $\mathbb{Q}_p$. We use the convention that the cyclotomic character has Hodge-Tate weight $-1$. A continuous character $\delta:\mathbb{Q}_p^\times\to L^\times$ is Hodge-Tate of weight $s\in\mathbb{Z}$ if its restriction to an open subgroup of $\mathbb{Z}_p^\times$ is $z\mapsto z^{-s}$. Equivalently, $\delta|_{\mathbb{Z}_p^\times}=\eta(z)z^{-s}$ for a finite-order character $\eta$. There is no restriction on $\delta(p)$ in this definition for a rank-one $(\varphi,\Gamma)$-module.

Fix a saturated triangulation $0=\mathrm{Fil}^0D\subset\mathrm{Fil}^1D\subset\cdots\subset\mathrm{Fil}^dD=D$ with $\mathrm{gr}^iD\simeq\mathcal{R}_L(\delta_i)$, and put $s_i=-\delta_i'(1)$. For $D=D_{\mathrm{rig}}(V)$, the $s_i$ are the Hodge-Tate-Sen weights of $V$, with multiplicity. Integrality of these numbers alone does not imply that $V$ is Hodge-Tate when weights repeat. We call the fixed triangulation \textit{non-critical} if $s_i\in\mathbb{Z}$ for all $i$ and $s_1<\cdots<s_d$. In this situation $D$ is de Rham. The order here is the order of the graded pieces of the given triangulation; we do not reorder these pieces. We call $\delta$ \textit{exceptional} if $\delta(z)=z^n$ or $\delta(z)=z^n|z|$ for some $n\in\mathbb{Z}$, and call the triangulation non-exceptional if every $\delta_i$ is non-exceptional.

\begin{remark}
For an equal-characteristic coefficient field $L=\mathbb{F}_q((u))$, we say it has a \textit{tame weight} $k\bmod p-1$ for a character satisfying $\delta|_{\mathbb{Z}_p^\times}=\overline\chi_{\mathrm{cyc}}^{-k}$, using $\mathbb{Z}_p^\times\simeq\Gamma_{\mathbb{Q}_p}$. This requires triviality on $1+p\mathbb{Z}_p$ and imposes no condition on $\delta(p)$, so it does not imply that the whole character has finite image. The characteristic-$p$ universal character with $\delta(1+p)=1+u$ does not satisfy this condition. We explain why a character as above should not be considered Hodge-Tate below. Consider a character $\delta: \Gamma_K \to \mathbb{Z}_p[[u]][\dfrac{1}{u}]$ given by $\delta(x) = \kappa_{\chi} \circ \chi_{cyc}(x)$, where $\kappa_{\chi}$ is the universal character as defined in \cite{andreatta2018halo}. This means that if $\chi_{cyc}(x) = 1+p$, then $\kappa_{\chi} \circ  \chi_{cyc}(x) = 1+ u$. Since such characters naturally show up in the description of trianguline $(\varphi, \Gamma)$-modules attached to modular forms in the boundary of the eigencurve, it is natural to wonder why we don't allow such characters to be Hodge-Tate.
    \begin{enumerate}
        \item We first notice that if such a character is Hodge-Tate in the generic fiber and its finite-order factor is trivial on $\chi_{cyc}^{-1}(1+p\mathbb{Z}_p)$, then it must lie inside the disk $|u| \leq |p|$ in $(\textrm{Spa } \mathbb{Z}_p[[u]])^{an}$. This locus is disjoint from the boundary. At the boundary, the character given above deforms to $\bar{\delta}: \Gamma_K \to \mathbb{F}_p((u))^{\times}$ taking $\bar{\delta}(x) = (1+ u)$. It is unusual to have a sudden jump at the boundary where such a character is allowed to be Hodge-Tate.

        \item In the setting of ordinary families of overconvergent eigenforms, whose associated Galois representations are studied in \cite{kisin2003overconvergent}, consider specializations approaching the boundary of the eigencurve in weight space. If the finite-order part of the weight character on $1+p\mathbb{Z}_p$ is fixed, then the Hodge--Tate specializations remain in a closed subdisk strictly contained in the open unit disk. Consequently, sufficiently close to the boundary, the associated Galois representations are no longer Hodge--Tate, and hence are neither de Rham nor crystalline.
\item The universal character $\kappa_{\chi}$ as discussed above is not locally analytic at the boundary as explained in \cite{andreatta2018halo}.

  \end{enumerate}

A residue class modulo $p-1$ has no well-defined sign, so this convention does not define a number of negative Hodge-Tate weights at a characteristic-$p$ point. The specialization and rank statements below apply to such points without assigning them Hodge-Tate weights. The Hodge-Tate and Bloch-Kato comparisons are made only at characteristic-zero points.
\end{remark}

\subsection{Selmer groups}
\begin{definition}
Let $K$ be a finite extension of $\mathbb{Q}_p$ and $V$ a finite-dimensional $L$-linear representation of $G_K$, where $L/\mathbb{Q}_p$ is finite. The local Bloch-Kato condition is
$$H^1_f(K,V)=\ker\left(H^1(G_K,V)\longrightarrow H^1(G_K,V\otimes_{\mathbb{Q}_p}B_{\mathrm{cris}})\right).$$
Thus it consists of extension classes that split after tensoring with $B_{\mathrm{cris}}$. If $V$ is crystalline, these are exactly the extensions $0\to V\to E\to L\to0$ with crystalline middle term. For $l\neq p$, put $H^1_f(\mathbb{Q}_l,V)=H^1_{\mathrm{ur}}(\mathbb{Q}_l,V)=\ker(H^1(G_{\mathbb{Q}_l},V)\to H^1(I_l,V))$.
\end{definition}

\begin{definition}[Bloch-Kato Selmer group]
For a finite-dimensional $L$-linear representation $V$ of $G_{\mathbb{Q},S}$, define
$$H^1_f(\mathbb{Q},V)=\ker\left(H^1(G_{\mathbb{Q},S},V)\longrightarrow\bigoplus_{v\in S}H^1(\mathbb{Q}_v,V)/H^1_f(\mathbb{Q}_v,V)\right).$$
The maps are localization followed by the quotient maps. Since $V$ is unramified outside $S$, this is the usual global Bloch-Kato Selmer group.
\end{definition}

\begin{definition}
Let $V$ be a finite projective $R$-linear representation of $G_{\mathbb{Q},S}$ and let $D\subset D_{\mathrm{rig}}(V|_{G_{\mathbb{Q}_p}})$ be a saturated $(\varphi,\Gamma)$-submodule with finite projective quotient. At $p$, put $U_p^+=C^\bullet_{\varphi,\Gamma}(D)$, with the morphism induced by inclusion and the comparison with Galois cohomology. For $l\neq p$, put $U_l^+=[V^{I_l}\xrightarrow{\mathrm{Frob}_l-1}V^{I_l}]$ in degrees $0$ and $1$, with its inflation morphism to $C^\bullet_{\mathrm{cont}}(G_{\mathbb{Q}_l},V)$. Define
$$C_S^\bullet(V,D)=\mathrm{Cone}\left(C^\bullet_{\mathrm{cont}}(G_{\mathbb{Q},S},V)\oplus\bigoplus_{v\in S}U_v^+\xrightarrow{\oplus_v(\mathrm{res}_v-i_v)}\bigoplus_{v\in S}C^\bullet_{\mathrm{cont}}(G_{\mathbb{Q}_v},V)\right)[-1].$$
For the Iwasawa complex, use $V_{\mathrm{Iw}}$ with the action specified above and set $U_{p,\mathrm{Iw}}^+=C^\bullet_{\varphi,\Gamma}(\mathbf{Dfm}(D))$ and $U_{l,\mathrm{Iw}}^+=[V_{\mathrm{Iw}}^{I_l}\xrightarrow{\mathrm{Frob}_l-1}V_{\mathrm{Iw}}^{I_l}]$ for $l\neq p$. The same cone defines $C_{S,\mathrm{Iw}}^\bullet(V,D)$. Here Frobenius acts on both factors of $V_{\mathrm{Iw}}$, including the inverse universal character. We suppress $S$ from now on.
\end{definition}

We first record a crystalline case of the comparison with the Bloch-Kato condition.
\begin{proposition}
Let $V$ be a representation over a finite extension $L/\mathbb{Q}_p$ which is crystalline at $p$, and let $D\subset D_{\mathrm{rig}}(V)$ be a saturated crystalline $(\varphi,\Gamma)$-submodule. Put $E=\mathbf{D}_{\mathrm{cris}}(D)\subset\mathbf{D}_{\mathrm{cris}}(V)$ and $F=\mathbf{D}_{\mathrm{cris}}(V)/E$, with the induced and quotient filtrations. Suppose that $\mathrm{Fil}^0E=0$, $\mathrm{Fil}^0F=F$, $F^{\varphi=1}=0$, and $E/(p\varphi-1)E=0$. Then the inclusion induces an isomorphism $H^1_{\varphi,\Gamma}(D)\xrightarrow{\sim}H^1_f(\mathbb{Q}_p,V)$, and $H^1(C^\bullet(V,D))\simeq H^1_f(\mathbb{Q},V)$.
\end{proposition}
\begin{proof}
The crystalline comparison of local conditions gives the first assertion under precisely these filtration and Frobenius hypotheses; see Proposition 3.7(3) of \cite{pottharst2013analytic} and Proposition 5.19 of \cite{kuccuk2023factorization}. The notation distinguishes the filtered $\varphi$-module $E$ from the Robba-ring module $D$ whose Herr cohomology occurs in the local condition. Moreover, $H^0_{\varphi,\Gamma}(D)=0$ by $\mathrm{Fil}^0E=0$, and $H^0_{\varphi,\Gamma}(D_{\mathrm{rig}}(V)/D)=0$ by $F^{\varphi=1}=0$. Hence $H^0(\mathbb{Q}_p,V)=0$ and the local map in degree $1$ is injective. At $l\neq p$, inflation identifies $H^1(U_l^+)$ with $H^1_{\mathrm{ur}}(\mathbb{Q}_l,V)$ and is an isomorphism in degree $0$. The long exact sequence of the defining cone therefore identifies $H^1(C^\bullet(V,D))$ with the kernel imposing these local conditions, which is $H^1_f(\mathbb{Q},V)$.
\end{proof}

For the family statements assume from now on that $H^1(I_l,V)$ is projective over $R$ for every $l\in S$, $l\neq p$.
\begin{proposition}
\label{selmerperfect}
The complexes $C^\bullet(V,D)$ and $C^\bullet_{\mathrm{Iw}}(V,D)$ belong to $\mathbf{D}_{\mathrm{perf}}^{[0,3]}(R)$ and $\mathbf{D}_{\mathrm{perf}}^{[0,3]}(\Lambda_R)$, respectively.
\end{proposition}
\begin{proof}
The Herr-cohomology perfectness theorem gives $U_p^+\in\mathbf{D}_{\mathrm{perf}}^{[0,2]}(R)$, and Theorem \ref{perf02iw}, together with the cyclotomic comparison, gives $U_{p,\mathrm{Iw}}^+\in\mathbf{D}_{\mathrm{perf}}^{[0,2]}(\Lambda_R)$. Here we use the finite projective representative over $\Lambda_R$ given by that theorem.

For $l\neq p$, Lemma \ref{selmergaloisperfect} represents $R\Gamma(I_l,V)$ by $[P_l^0\xrightarrow{d_l}P_l^1]$, with finite projective terms. Since $\mathrm{coker}(d_l)=H^1(I_l,V)$ is projective, the exact sequence $0\to\mathrm{im}(d_l)\to P_l^1\to H^1(I_l,V)\to0$ splits, so $\mathrm{im}(d_l)$ is projective. The sequence $0\to V^{I_l}\to P_l^0\to\mathrm{im}(d_l)\to0$ then splits as well. Thus $V^{I_l}$ is finite projective and its formation commutes with coefficient base change. Consequently, $U_l^+\in\mathbf{D}_{\mathrm{perf}}^{[0,1]}(R)$.

The cyclotomic extension is unramified at $l\neq p$, so inertia acts trivially on the distribution factor and $V_{\mathrm{Iw}}^{I_l}=V^{I_l}\otimes_R\Lambda_R$. This gives a two-term finite projective representative for $U_{l,\mathrm{Iw}}^+$, with the deformed Frobenius action. Lemma \ref{selmergaloisperfect} provides the global and local Galois complexes, and Proposition \ref{selmerconeperfect} proves both assertions.
\end{proof}

\begin{corollary}
\label{selmerbasechange}
Let $R\to R'$ be a continuous morphism of pseudoaffinoid algebras, put $V'=V\otimes_RR'$ and $D'=D\widehat\otimes_RR'$, and write $\Lambda_{R'}=\Lambda_{R',(0,\infty]}(\Gamma)$. There are natural quasi-isomorphisms
$$C^\bullet(V,D)\otimes_R^{\mathbf L}R'\simeq C^\bullet(V',D'),\qquad C^\bullet_{\mathrm{Iw}}(V,D)\otimes_{\Lambda_R}^{\mathbf L}\Lambda_{R'}\simeq C^\bullet_{\mathrm{Iw}}(V',D').$$
\end{corollary}
\begin{proof}
Lemma \ref{selmergaloisperfect} gives derived base change for the Galois cohomology complexes. The Herr and Iwasawa cohomology base-change theorems give it for the local condition at $p$. For $l\neq p$, the splitting in the preceding proof gives $V^{I_l}\otimes_RR'\simeq(V')^{I_l}$, compatibly with Frobenius, and gives the analogous statement for cyclotomic deformation. These identifications respect the localization and local-condition morphisms. Applying derived tensor product to the defining cone proves the result. The displayed tensor products are computed termwise on finite projective representatives.
\end{proof}

\begin{lemma}
\label{complexblochkato}
Let $L/\mathbb{Q}_p$ be finite and $V$ a finite-dimensional continuous representation of $G_{\mathbb{Q},S}$ over $L$. Put $D=D_{\mathrm{rig}}(V|_{G_{\mathbb{Q}_p}})$ and fix a saturated non-exceptional triangulation $\mathrm{Fil}^\bullet D$. Then:
\begin{enumerate}
\item For every $0\leq a\leq d$, there is a natural isomorphism
$$H^1(C^\bullet(V,\mathrm{Fil}^aD))\simeq\ker\left(H^1(G_{\mathbb{Q},S},V)\longrightarrow\bigoplus_{\substack{l\in S\\l\neq p}}\frac{H^1(\mathbb{Q}_l,V)}{H^1_{\mathrm{ur}}(\mathbb{Q}_l,V)}\oplus H^1_{\varphi,\Gamma}(D/\mathrm{Fil}^aD)\right).$$
The component at $p$ is localization followed by the map induced by $D\to D/\mathrm{Fil}^aD$.
\item If the fixed triangulation is non-critical and $a$ is the number of strictly negative Hodge-Tate weights, then $H^1(C^\bullet(V,\mathrm{Fil}^aD))\simeq H^1_f(\mathbb{Q},V)$.
\end{enumerate}
\end{lemma}
\begin{proof}
We first prove (a). Non-exceptionality gives $H^0_{\varphi,\Gamma}(\mathrm{gr}^iD)=H^2_{\varphi,\Gamma}(\mathrm{gr}^iD)=0$ for every $i$. Applying the long exact sequences along the triangulation gives $H^0_{\varphi,\Gamma}(D)=H^0_{\varphi,\Gamma}(D/\mathrm{Fil}^aD)=H^2_{\varphi,\Gamma}(\mathrm{Fil}^aD)=0$. Hence
$$0\longrightarrow H^1_{\varphi,\Gamma}(\mathrm{Fil}^aD)\longrightarrow H^1_{\varphi,\Gamma}(D)\longrightarrow H^1_{\varphi,\Gamma}(D/\mathrm{Fil}^aD)\longrightarrow0.$$
In particular, the local-condition map at $p$ is injective in degree $1$ and $H^0(\mathbb{Q}_p,V)=0$. At $l\neq p$, inflation is an isomorphism in degree $0$ and an injection in degree $1$ with image $H^1_{\mathrm{ur}}(\mathbb{Q}_l,V)$.

Write $A^\bullet=C^\bullet_{\mathrm{cont}}(G_{\mathbb{Q},S},V)\oplus\bigoplus_{v\in S}U_v^+$ and $B^\bullet=\bigoplus_{v\in S}C^\bullet_{\mathrm{cont}}(G_{\mathbb{Q}_v},V)$. The map $H^0(A^\bullet)\to H^0(B^\bullet)$ is surjective: the summands at $l\neq p$ are supplied by $H^0(U_l^+)$, and the summand at $p$ is zero. The long exact sequence of $\mathrm{Cone}(A^\bullet\to B^\bullet)[-1]$ therefore gives $H^1(C^\bullet(V,\mathrm{Fil}^aD))\simeq\ker(H^1(A^\bullet)\to H^1(B^\bullet))$. Every local-condition map is injective in degree $1$. Projection to global cohomology consequently identifies this kernel with the classes whose localizations lie in the prescribed local subspaces. The short exact sequence above identifies the quotient at $p$ with $H^1_{\varphi,\Gamma}(D/\mathrm{Fil}^aD)$, proving (a).

For (b), non-criticality means that the weights in the fixed triangulation satisfy $s_1<\cdots<s_a<0\leq s_{a+1}<\cdots<s_d$, with the evident interpretation when $a=0$ or $a=d$. The local comparison in Proposition 5 of \cite{bellaiche2012ranks} identifies $\mathrm{im}(H^1_{\varphi,\Gamma}(\mathrm{Fil}^aD)\to H^1(\mathbb{Q}_p,V))$ with $H^1_f(\mathbb{Q}_p,V)$. We need strict inequality, an extension of the trivial module by $\mathrm{Fil}^aD$ has ordered weights $s_1<\cdots<s_a<0$ followed by $0$, so it is de Rham. Its image is therefore in $H^1_g(\mathbb{Q}_p,V)$. Non-exceptionality gives $H^1_g(\mathbb{Q}_p,V)=H^1_f(\mathbb{Q}_p,V)$, and both this space and $H^1_{\varphi,\Gamma}(\mathrm{Fil}^aD)$ have dimension $a$. Here $a=\dim_L\mathbf{D}_{\mathrm{dR}}(V)/\mathrm{Fil}^0$, so it counts strictly negative weights. Together with the unramified conditions away from $p$, (a) now gives the global Bloch-Kato Selmer group.
\end{proof}

The rank statements below are about the complexes attached to a fixed step of a fixed triangulation. They do not require non-criticality. That hypothesis is used when identifying their characteristic-zero ordinary fibers with Bloch-Kato Selmer groups. This distinction is needed in comparing our statements with Theorem 1 of \cite{bellaiche2012ranks}. We treat ordinary dimensions and Iwasawa ranks separately.

\subsubsection{When $R$ has dimension $1$}
In this subsection, let $R$ be a regular pseudoaffinoid domain of dimension $1$. A one-dimensional Noetherian ring satisfying $(R_1)$ is regular, so the domain assumption can be imposed by working on its connected components. For Iwasawa ranks, fix one character component of the finite group $\Delta$ and continue to write $\Lambda_R$ for that component. Its Banach stages $A_n$ and their fibers $A_n/\mathfrak mA_n$ are domains. Ranks of compatible coadmissible modules are computed at these stages.

\begin{lemma}
\label{dimension1Rpseudonull}
Fix a Banach stage $A=A_n$ and a finite pseudo-null $A$-module $M$. For every maximal ideal $\mathfrak m\subset R$ and every $i\geq0$, one has $\mathrm{rank}_{A/\mathfrak mA}\mathrm{Tor}_i^R(M,R/\mathfrak m)=0$.
\end{lemma}
\begin{proof}
Put $\mathfrak p=\mathfrak mA$, which is prime because $A/\mathfrak mA$ is a domain. The local ring $R_{\mathfrak m}$ is a DVR; let $\pi$ be a uniformizer. Since $A$ is flat over $R$, $\pi$ is a non-zero-divisor after localizing over $\mathfrak m$, and the maximal ideal of $A_{\mathfrak p}$ is generated by $\pi$. Thus $A_{\mathfrak p}$ is a DVR and $\mathrm{ht}(\mathfrak p)=1$. Pseudo-nullity gives $M_{\mathfrak p}=0$. Localizing a flat resolution of $R/\mathfrak m$ at the generic point of $A/\mathfrak mA$ therefore makes the complex computing $\mathrm{Tor}_i^R(M,R/\mathfrak m)$ zero. Its rank over $A/\mathfrak mA$ is consequently zero.
\end{proof}

\begin{corollary}
\label{corodim1R}
If $M\to N$ is a pseudo-isomorphism of finite $A_n$-modules, then, for every $i\geq0$,
$$\mathrm{rank}_{A_n/\mathfrak mA_n}\mathrm{Tor}_i^R(M,R/\mathfrak m)=\mathrm{rank}_{A_n/\mathfrak mA_n}\mathrm{Tor}_i^R(N,R/\mathfrak m).$$
The same statement holds stagewise for compatible coadmissible modules.
\end{corollary}
\begin{proof}
The kernel and cokernel vanish after localizing at the height-one prime $\mathfrak mA_n$. Thus $M_{\mathfrak mA_n}\to N_{\mathfrak mA_n}$ is an isomorphism, and the localized complexes computing Tor are isomorphic.
\end{proof}

\begin{theorem}
\label{selmerrankjump}
Retain the hypotheses of Proposition \ref{selmerperfect}, and suppose that $D_{\mathrm{rig}}(V|_{G_{\mathbb{Q}_p}})$ has a fixed saturated triangulation with finite projective graded pieces. Fix $0\leq a\leq d$, put $D^a=\mathrm{Fil}^aD_{\mathrm{rig}}(V|_{G_{\mathbb{Q}_p}})$, and write $C_\star^\bullet=C_\star^\bullet(V,D^a)$ for $\star\in\{\varnothing,\mathrm{Iw}\}$. For $x\in\mathrm{Max}(R)$, define
$$h_\varnothing(x)=\dim_{\kappa(x)}H^1(C^\bullet(V_x,D_x^a)),\qquad h_{\mathrm{Iw}}(x)=\mathrm{rank}_{\Lambda_{\kappa(x)}}H^1(C^\bullet_{\mathrm{Iw}}(V_x,D_x^a)).$$
There are integers $n_a^\star\geq0$ such that $h_\star(x)\geq n_a^\star$ for every $x$, with equality outside a finite set. In the ordinary case, the jump set is
$$V(I_\varnothing)\cap\mathrm{Max}(R),\qquad I_\varnothing=\mathrm{Char}_R(H^1(C^\bullet)_{\mathrm{tors}})\,\mathrm{Char}_R(H^2(C^\bullet)_{\mathrm{tors}}).$$
In the Iwasawa case, choose a Banach stage $A=A_n$, put $C_n^\bullet=C_{\mathrm{Iw}}^\bullet\otimes_{\Lambda_R}^{\mathbf L}A$ and $T_{i,n}=H^i(C_n^\bullet)_{\mathrm{tors}}$, and set
$$S_{\mathrm{Iw}}=\{\mathfrak m\in\mathrm{Max}(R):(T_{1,n})_{\mathfrak mA}\neq0\ \text{or}\ (T_{2,n})_{\mathfrak mA}\neq0\},\qquad I_{\mathrm{Iw}}=\prod_{\mathfrak m\in S_{\mathrm{Iw}}}\mathfrak m.$$
The set $S_{\mathrm{Iw}}$ is finite and independent of the stage, and is exactly the Iwasawa rank-jump set. The empty product is $R$.

Let $Z_a$ be the set of characteristic-zero points where the specialized fixed triangulation is non-critical and non-exceptional, with exactly $a$ strictly negative Hodge-Tate weights. For $x\in Z_a$, one has $h_\varnothing(x)=\dim_{\kappa(x)}H^1_f(\mathbb{Q},V_x)$. The rank conclusions hold at every point of $\overline{Z}_a$ by restriction, including its characteristic-$p$ points.
\end{theorem}
\begin{proof}
First consider the ordinary complex, and put $H^i=H^i(C^\bullet)$. Corollary \ref{selmerbasechange} identifies its derived fiber with $C^\bullet(V_x,D_x^a)$. Since $R_{\mathfrak m_x}$ is a DVR, the base-change spectral sequence has only Tor degrees $0$ and $1$, and gives
$$0\longrightarrow H^1\otimes_R\kappa(x)\longrightarrow H^1(C^\bullet(V_x,D_x^a))\longrightarrow\mathrm{Tor}_1^R(H^2,\kappa(x))\longrightarrow0.$$
The torsion-free quotients of the finite $R$-modules $H^i$ are projective. Therefore, with $n_a^\varnothing=\mathrm{rank}_R H^1$, we have
$$h_\varnothing(x)=n_a^\varnothing+\dim_{\kappa(x)}(H^1_{\mathrm{tors}}\otimes_R\kappa(x))+\dim_{\kappa(x)}\mathrm{Tor}_1^R(H^2_{\mathrm{tors}},\kappa(x)).$$
For a finite torsion module over a DVR, either of the last two terms is nonzero precisely when the corresponding localized torsion module is nonzero. These supports are the zero sets of their characteristic ideals, proving the ordinary assertion.

For the Iwasawa complex, fix $A=A_n$, put $\mathfrak p=\mathfrak m_xA$, and let $F_x=\mathrm{Frac}(A/\mathfrak m_xA)$. As in Lemma \ref{dimension1Rpseudonull}, $A_{\mathfrak p}$ is a DVR with residue field $F_x$. Base change and localization identify the rank of the Iwasawa fiber with $\dim_{F_x}H^1(C_n^\bullet\otimes_A^{\mathbf L}F_x)$. Applying the same two-column spectral sequence over $A_{\mathfrak p}$ gives
$$h_{\mathrm{Iw}}(x)=n_a^{\mathrm{Iw}}+\dim_{F_x}((T_{1,n})_{\mathfrak p}\otimes_{A_{\mathfrak p}}F_x)+\dim_{F_x}\mathrm{Tor}_1^{A_{\mathfrak p}}((T_{2,n})_{\mathfrak p},F_x),$$
where $n_a^{\mathrm{Iw}}=\mathrm{rank}_A H^1(C_n^\bullet)$. Both additional terms are nonnegative, and their sum is positive exactly when $\mathfrak m_x\in S_{\mathrm{Iw}}$. Flat restriction between Banach stages commutes with cohomology and gives injective maps between their fiber domains. It preserves these generic ranks, so $n_a^{\mathrm{Iw}}$ and $S_{\mathrm{Iw}}$ are independent of the stage.

Each $T_{i,n}$ is a finite torsion $A$-module, so a nonzero element of $A$ annihilates it. Only finitely many height-one primes can contain such an element. Since each $\mathfrak mA$ has height one and contracts to $\mathfrak m$, the set $S_{\mathrm{Iw}}$ is finite. This proves the Iwasawa assertion. Only the vertical primes $\mathfrak mA$ enter this calculation; torsion supported at other primes has zero generic rank on the fiber. Finally, Lemma \ref{complexblochkato} gives the Bloch-Kato identification on $Z_a$.
\end{proof}

\subsubsection{A general pseudoaffinoid algebra}
We now give a determinantal description for a general $R$. Keep the perfectness and base-change hypotheses above, and fix a character component of $\Delta$ in the Iwasawa case.

\begin{theorem}
\label{thm:determinantal-rank-locus}
Let $C_\star^\bullet$ denote either of the Selmer complexes above, and choose a finite projective representative
$$P_\star^0\xrightarrow{d_\star^0}P_\star^1\xrightarrow{d_\star^1}P_\star^2\xrightarrow{d_\star^2}P_\star^3.$$
Its coefficient ring is $A_\varnothing=R$ in the ordinary case and $A_{\mathrm{Iw}}=\Lambda_R$ in the Iwasawa case. Work componentwise so that $n_{1,\star}=\mathrm{rank}_{A_\star}P_\star^1$ is constant. Define $h_\star(x)$ from the specialized Selmer complexes as in Theorem \ref{selmerrankjump}, using dimension in the ordinary case and distribution-algebra rank in the Iwasawa case. For every $q\geq0$, there is a finitely generated ideal $\mathfrak J_{\star,q}\subset R$ such that
$$\{x\in\mathrm{Max}(R):h_\star(x)\geq q\}=V(\mathfrak J_{\star,q})\cap\mathrm{Max}(R).$$
More explicitly, choose finite free modules in which the $P_\star^i$ are direct summands and extend each differential by zero on a complement. Let $I_{r+1}(d_\star^i)$ be generated by the $(r+1)\times(r+1)$ minors of the resulting finite matrix. In the ordinary case put $\mathfrak a_{\varnothing,i,r}=I_{r+1}(d_\varnothing^i)$. In the Iwasawa case expand each of these finitely many minors as $f(T)=\sum_{j\geq0}c_jT^j$ and let $\mathfrak a_{\mathrm{Iw},i,r}$ be the ideal of $R$ generated by all their coefficients $c_j$. Then one can take
$$\mathfrak J_{\star,q}=\prod_{\substack{r,s\geq0\\r+s\leq n_{1,\star}-q}}(\mathfrak a_{\star,0,r}+\mathfrak a_{\star,1,s}),$$
with the empty product equal to $R$. On each irreducible component, $h_\star$ has a minimum generic value $q_\star$, and its jump locus is the closed subset defined by $\mathfrak J_{\star,q_\star+1}$ on that component.
\end{theorem}
\begin{proof}
Fix $x\in\mathrm{Max}(R)$. By derived base change, specializing $P_\star^\bullet$ represents the Selmer complex of the fiber. In the Iwasawa case, further extend scalars to the fraction field of the chosen component of $\Lambda_{\kappa(x)}$. Since $d_\star^1d_\star^0=0$, elementary linear algebra gives
$$h_\star(x)=n_{1,\star}-\mathrm{rank}(d_{\star,x}^0)-\mathrm{rank}(d_{\star,x}^1).$$
Extending a differential by zero on a projective complement does not change its rank after passing to a field. Thus $\mathrm{rank}(d_{\varnothing,x}^i)\leq r$ exactly when all $(r+1)$-minors vanish modulo $\mathfrak m_x$, or equivalently $\mathfrak a_{\varnothing,i,r}\subseteq\mathfrak m_x$.

In the Iwasawa case, uniqueness of the power-series expansion implies that $f(T)$ specializes to zero in $\Lambda_{\kappa(x)}$ exactly when every coefficient of $f$ lies in $\mathfrak m_x$. The same minors criterion therefore gives $\mathrm{rank}(d_{\mathrm{Iw},x}^i)\leq r$ exactly when $\mathfrak a_{\mathrm{Iw},i,r}\subseteq\mathfrak m_x$. The minors can have infinitely many coefficients, the ideal they generate in $R$ is finitely generated because $R$ is Noetherian.

The rank formula shows that $h_\star(x)\geq q$ exactly when there are $r,s\geq0$ with $r+s\leq n_{1,\star}-q$, $\mathrm{rank}(d_{\star,x}^0)\leq r$, and $\mathrm{rank}(d_{\star,x}^1)\leq s$. Hence,
$$\{x:h_\star(x)\geq q\}=\bigcup_{\substack{r,s\geq0\\r+s\leq n_{1,\star}-q}}V(\mathfrak a_{\star,0,r}+\mathfrak a_{\star,1,s})=V(\mathfrak J_{\star,q}),$$
where the equality is on maximal points of $R$. On an irreducible component, each differential attains its maximum rank on a nonempty open subset, as detected by a nonvanishing minor, or by a nonvanishing coefficient of a minor in the Iwasawa case. These two open subsets intersect. The rank formula therefore gives a minimum generic value $q_\star$, and the locus of larger values is $V(\mathfrak J_{\star,q_\star+1})$ on that component.
\end{proof}
\appendix
\section{Modules in Iwasawa theory}
Fix $\Gamma=\Gamma_K$, and write
$$
\Lambda
=
\Lambda_{R,(0,\infty]}(\Gamma),
\qquad
\Lambda_a
=
\Lambda_{R,[\mu(a),\infty]}(\Gamma).
$$
All assertions below are understood componentwise with respect to the
idempotent decomposition coming from the torsion subgroup of $\Gamma$.

For the ordinary characteristic-ideal statements in this
appendix, assume that $R$ and every Banach stage $\Lambda_a$
are regular. All assertions are understood componentwise
on their domain components.
In particular, for every
height-one prime ideal $\mathfrak p$ of $\Lambda_a$, the local ring
$(\Lambda_a)_{\mathfrak p}$ is a discrete valuation ring.

\begin{definition}
Let $M$ be a coadmissible $\Lambda$-module, and put
$$
M_a=M\otimes_{\Lambda}\Lambda_a.
$$
We say that $M$ is \emph{pseudo-null} if
$$
(M_a)_{\mathfrak p}=0
$$
for every $a$ and every prime ideal $\mathfrak p$ of $\Lambda_a$ of
height at most one.

Equivalently, since $M_a$ is finitely generated over the Noetherian
ring $\Lambda_a$, the support of $M_a$ has codimension at least two for
every $a$. If $M_a\neq 0$, this is equivalent to
$$
\operatorname{ht}\bigl(\operatorname{Ann}_{\Lambda_a}(M_a)\bigr)
\geq 2.
$$
\end{definition}

\begin{definition}
A morphism $f:M\to N$ of coadmissible $\Lambda$-modules is called a
\emph{pseudo-isomorphism} if both $\ker(f)$ and $\operatorname{coker}(f)$
are pseudo-null. If such a morphism exists, we write
$$
M\sim N.
$$
\end{definition}

The following standard structure theorem will be used repeatedly.

\begin{lemma}\label{serretorsion}
Let $A$ be a Noetherian ring satisfying Serre's condition $(R_1)$, and
let $M$ be a finitely generated torsion $A$-module. Then there is a
pseudo-isomorphism
$$
M
\sim
\bigoplus_{\substack{\mathfrak p\in\operatorname{Spec}(A)\\
\operatorname{ht}(\mathfrak p)=1}}
\ \bigoplus_{i\geq 1}
\left(A/\mathfrak p^i\right)^{n(\mathfrak p,i)},
$$
where the integers $n(\mathfrak p,i)$ are nonnegative and only finitely
many of them are nonzero.
\end{lemma}

\begin{proof}
At every height-one prime $\mathfrak p$, the local ring
$A_{\mathfrak p}$ is a DVR, so the elementary-divisor
theorem gives a finite direct sum of
$A_{\mathfrak p}/\mathfrak p^iA_{\mathfrak p}$.
Only finitely many height-one primes occur in the
support of the torsion module $M$.
The local decompositions determine the claimed
pseudo-isomorphism by
\cite[Proposition~2.10.18 and Corollary~2.10.19]{nekovar2006selmer}.
\end{proof}

We record some elementary examples. Let $A$ be a Noetherian domain
satisfying $(R_1)$.

\textbf{Examples:}\\
1. A nonzero finite free $A$-module is not pseudo-null, since its
localization at every prime ideal is nonzero.

2. Let $\mathfrak p_1$ and $\mathfrak p_2$ be distinct height-one prime
ideals of $A$. The natural diagonal map
$$
A/(\mathfrak p_1\mathfrak p_2)
\longrightarrow
A/\mathfrak p_1\oplus A/\mathfrak p_2
$$
is an isomorphism after localization at every prime ideal of height at
most one. Its kernel and cokernel are therefore pseudo-null, and hence
$$
A/(\mathfrak p_1\mathfrak p_2)
\sim
A/\mathfrak p_1\oplus A/\mathfrak p_2.
$$

3.Let $\mathfrak p_1$ and $\mathfrak p_2$ be distinct height-one prime
ideals of $A$. No height-one prime ideal contains both
$\mathfrak p_1$ and $\mathfrak p_2$. Consequently,
$$
\operatorname{ht}(\mathfrak p_1+\mathfrak p_2)\geq 2,
$$
and
$$
A/(\mathfrak p_1+\mathfrak p_2)
$$
is pseudo-null.

These examples apply to each $\Lambda_a$, and hence to coadmissible
$\Lambda$-modules by passage through the Banach stages.

\begin{definition}
Let $M$ be a torsion coadmissible $\Lambda$-module. For every $a$,
Lemma \ref{serretorsion} gives a pseudo-isomorphism
$$
M_a
\sim
\bigoplus_{\substack{\mathfrak p\in\operatorname{Spec}(\Lambda_a)\\
\operatorname{ht}(\mathfrak p)=1}}
\ \bigoplus_{i\geq 1}
\left(\Lambda_a/\mathfrak p^i\right)^{n_a(\mathfrak p,i)}.
$$
For a height-one prime ideal $\mathfrak p$ of $\Lambda_a$, put
$$
e_a(\mathfrak p)
=
\sum_{i\geq 1}i\,n_a(\mathfrak p,i).
$$
Equivalently,
$
e_a(\mathfrak p)
=
\operatorname{length}_{(\Lambda_a)_{\mathfrak p}}
\left((M_a)_{\mathfrak p}\right).
$
The characteristic ideal of $M_a$ is
$
\operatorname{char}_{\Lambda_a}(M_a)
=
\prod_{\substack{\mathfrak p\in\operatorname{Spec}(\Lambda_a)\\
\operatorname{ht}(\mathfrak p)=1}}
\mathfrak p^{e_a(\mathfrak p)}.
$

Regular Noetherian rings are locally factorial.
Consequently each height-one prime ideal is invertible,
and the above characteristic ideal is an invertible
ideal determined by its codimension-one multiplicities.
Under a flat restriction map, its extension remains
invertible. The localized DVR length formula identifies
its multiplicities with those of the restricted module.
Hence
$\operatorname{char}_{\Lambda_a}(M_a)\Lambda_{a'}
=\operatorname{char}_{\Lambda_{a'}}(M_{a'}).$
These finite ideals therefore form a compatible coherent
ideal sheaf on the quasi-Stein space; its global sections
give the corresponding coadmissible ideal of $\Lambda$.
\end{definition}
Since $\Lambda_a$ is flat over $R$ and one variable power series ring over $R$, a prime ideal $\mathfrak{q}$ of $R$ of height 1 gives us a height 1 prime ideal $\mathfrak{q} \Lambda_a$ in $\Lambda_a$.
Let
$e_{\mathfrak q}(M)
= e_a(\mathfrak q\Lambda_a)$
whenever $\mathfrak q\Lambda_a$ is proper. Compatibility of
characteristic ideals under base change shows that this integer is
independent of the sufficiently small choice of $a$. We define
$$
\operatorname{char}_{R}(M)
=
\prod_{\substack{\mathfrak q\in\operatorname{Spec}(R)\\
\operatorname{ht}(\mathfrak q)=1}}
\mathfrak q^{e_{\mathfrak q}(M)}.
$$
Only finitely many factors in this product have nonzero exponent.
We now specialize to coefficients in a field.

\begin{lemma}\label{soobvious}
Suppose that $A$ is either a finite extension of $\mathbb Q_p$ or a
finite extension of $\mathbb F_p((u))$. Then every idempotent component
of
$\Lambda_{A,[a,\infty]}(\Gamma_K)
$
has Krull dimension one. Consequently, every nonzero prime ideal in
such a component is maximal. The maximal ideals of
$\Lambda_{A,[a,\infty]}(\Gamma_K)$
are precisely the point-evaluation ideals corresponding to continuous
characters satisfying the relevant radius condition.
\end{lemma}

\begin{proof}
After fixing an idempotent component, the coefficient
ring is a finite field extension of $A$, and the
torsion-free factor contributes one analytic variable.
Thus $\Lambda_{A,[a,\infty]}(\Gamma_K)$ on that component
is a one-variable closed-disk affinoid algebra.
It is a domain of Krull dimension one.
Consequently every nonzero prime ideal is maximal.
Its maximal ideals are the evaluation ideals described
in Lemma~\ref{max}, with coefficient
maps extending the given map from $A$ and characters
satisfying the stage radius condition.
\end{proof}

\begin{lemma}\label{sosoobvious}
Suppose that $A$ is either a finite extension of $\mathbb Q_p$ or a
finite extension of $\mathbb F_p((u))$. Let $\mathfrak m_\eta$ be the
maximal ideal of $\Lambda_{A,(0,\infty]}(\Gamma_K)$
corresponding to a continuous character $\eta$. Then, for every
$i\geq 1$, the quotient
$\Lambda_{A,(0,\infty]}(\Gamma_K)/\mathfrak m_\eta^i$
is finite-dimensional over $A$. The same conclusion holds for a
Banach stage whenever $\eta$ satisfies the corresponding radius
condition.
\end{lemma}

\begin{proof}
For $i=1$, the quotient is the residue field at the point defined by
$\eta$, and this is a finite extension of $A$.
On the idempotent component containing $\mathfrak m_\eta$, the maximal
ideal $\mathfrak m_\eta$ is finitely generated; in fact, in the
one-variable description it is principal. Therefore
$
\mathfrak m_\eta^j/\mathfrak m_\eta^{j+1}
$
is a finitely generated module over
$
\Lambda_{A,(0,\infty]}(\Gamma_K)/\mathfrak m_\eta,
$
and is consequently finite-dimensional over $A$. The exact sequences
$$
0
\longrightarrow
\mathfrak m_\eta^{i-1}/\mathfrak m_\eta^i
\longrightarrow
\Lambda_{A,(0,\infty]}(\Gamma_K)/\mathfrak m_\eta^i
\longrightarrow
\Lambda_{A,(0,\infty]}(\Gamma_K)/\mathfrak m_\eta^{i-1}
\longrightarrow
0
$$
now prove the assertion by induction on $i$. The argument for a
Banach stage is identical.
\end{proof}

\printbibliography
\end{document}

%% file: packages.tex
\usepackage{amsfonts, amsmath, amssymb, stmaryrd, color, enumerate, relsize, mathrsfs,array}
\usepackage[utf8]{inputenc}
\DeclareUnicodeCharacter{0308}{}
\DeclareUnicodeCharacter{FFFD}{}
\usepackage[all]{xy}
\usepackage{float}

\usepackage[
    backend=biber,
    style=alphabetic,    
    sorting=nyt,         
    maxbibnames=10
]{biblatex}
\usepackage{xcolor}
\DeclareNameAlias{author}{family-given}

\usepackage[dvipsnames]{xcolor}

\usepackage{hyperref}
\hypersetup{colorlinks,linkcolor={blue},citecolor={blue},urlcolor={red}} 
\usepackage[capitalise]{cleveref}
\usepackage{microtype}

\usepackage{tikz-cd}

\tikzcdset{
  cells={font=\everymath\expandafter{\the\everymath\displaystyle}},
}
